\documentclass{amsart}
\usepackage{tikz, pgfplots, epsfig, float, mathrsfs, bm, enumitem, setspace}
\usepackage[linesnumbered,ruled]{algorithm2e}
\usepackage[noend]{algpseudocode}
\usepackage[margin=1in]{geometry}
\pgfplotsset{compat=1.14,width=7.5cm}
\usepackage[colorlinks=true]{hyperref}
\hypersetup{urlcolor=black, citecolor=red}
\usepackage[super]{cite}
\DeclareMathOperator*{\esssup}{ess\,sup}
\newcommand{\longweak}{\relbar\joinrel\rightharpoonup}
\let\div\relax\DeclareMathOperator{\div}{div}
\usepgfplotslibrary{groupplots}
\SetLabelAlign{Center}{\hfil#1\hfil}
\newtheorem{remark}{Remark}

\newtheorem{lemma}{Lemma}
\newtheorem{theorem}{Theorem}
\newtheorem{definition}{Definition}

\newtheorem{corollary}{Corollary}

\pgfplotsset{colormap={special}{[1pt]
		rgb(0pt)    =(0.231373,0.298039, 0.752941),
		rgb(160pt)=(0.27451,0.337255,0.760784),
		rgb(307pt)=(0.364706,0.419608,0.776471),
		rgb(432pt)=(0.494118,0.533333,0.8),
		rgb(500pt)=(0.865003,0.865003,0.865003),
		rgb(594pt)=(0.772549,0.376471,0.45098),
		rgb(722pt)=(0.741176,0.211765,0.313725),
		rgb(850pt)=(0.721569,0.101961,0.223529),
		rgb(1000pt)=(0.705882,0.0156863, 0.14902),
	},
}

\usepackage[foot]{amsaddr}
\makeatletter
\renewcommand{\email}[2][]{%
	\ifx\emails\@empty\relax\else{\g@addto@macro\emails{,\space}}\fi%
	\@ifnotempty{#1}{\g@addto@macro\emails{\textrm{(#1)}\space}}%
	\g@addto@macro\emails{#2}%
}
\makeatother

\begin{document}
	
	\title[Unsteady Darcy--Forchheimer--Brinkman in tumor growth models]{ON THE UNSTEADY DARCY--FORCHHEIMER--BRINKMAN EQUATION \\ IN LOCAL AND NONLOCAL TUMOR GROWTH MODELS}
	\author{Marvin Fritz${}^{1,*}$}
	\address{${}^1$Department of Mathematics, Technical University of Munich, Germany}
	\email{\{fritzm,wohlmuth\}@ma.tum.de}
	\author{Ernesto A. B. F. Lima${}^2$}
	\address{${}^2$Oden Institute for Computational Engineering and Sciences, The University of Texas at Austin, USA}
	\email{\{lima,oden\}@oden.utexas.edu}
	\author{J. Tinsley Oden${}^2$}
	\author{Barbara Wohlmuth${}^1$}
	\thanks{${}^*$Corresponding author}
	\maketitle

\begin{abstract}
		A mathematical analysis of local and nonlocal phase-field models of tumor growth is presented that includes time-dependent Darcy--Forchheimer--Brinkman models of convective velocity fields and models of long-range cell interactions. A complete existence analysis is provided. In addition, a parameter-sensitivity analysis is described that quantifies the sensitivity of key quantities of interest to changes in parameter values. Two sensitivity analyses are examined; one employing statistical variances of model outputs and another employing the notion of active subspaces based on existing observational data. Remarkably, the two approaches yield very similar conclusions on sensitivity for certain quantities of interest. The work concludes with the presentation of numerical approximations of solutions of the governing equations and results of numerical experiments on tumor growth produced using finite element discretizations of the full tumor model for representative cases.
\end{abstract}



	\section{Introduction}
	
	In this exposition, we present a mathematical analysis of a general class of local and nonlocal multispecies phase-field models of tumor growth, derived using the balance laws of continuum mixture theory and employing mesoscale versions of the Ginzburg--Landau energy functional involving cell-species volume fractions. The model class involves nonlinear characterizations of cell velocity obeying a time-dependent Darcy--Forchheimer--Brinkman law. In addition, to account for long-range interactions characterizing such effects as cell-to-cell adhesion, a class of nonlocal models of tumor growth is considered, which involves systems of integro-differential equations.
	
	This contribution generalizes the analyses of Cahn--Hilliard--Darcy models studied in Refs.~\citen{garcke2016global,garcke2018cahn} and includes a detailed analysis of existence of weak solutions of the governing system of fourth-order integro-partial-differential equations. Beyond the mathematical analysis of this class of models, we also explore the sensitivity of model outputs to perturbations in parameters for the local models. These analyses draw from active subspace methods when observational data are available and on variance sensitivity analysis when models outputs due to parameter variations are considered. Finite element approximations of the general local tumor models are presented and the results of numerical experiments are given that depict the role of linear and nonlinear flow laws on the evolution and structure of solid tumors.
	
	This study is intended to contribute to a growing body of mathematical and computational work accumulated over the last two decades on tumor growth, decline, and invasion in living organisms. To date, the bulk of the models proposed are phenomenological models, designed to depict phenomena at the macro- or meso-scale where microenvironmental constituents are represented by fields describing volume fractions or mass concentrations of various species. The models studied here are in this category. Reviews and surveys of recent literature of tumor growth modeling can be found in Refs.~\citen{bellomo2008foundations,oden2016toward,deisboeck2010multiscale}, and in surveys of work of the last decade in Refs.~\citen{hatzikirou2005mathematical,quaranta2005mathematical,byrne2006modelling,graziano2007mechanics,roose2007mathematical,wise2008three}.
	
	Prominent among more recent proposed models are those involving diffuse-interface or phase-field representations designed to capture morphological instabilities in the form of phase changes driven by cell necrosis and non-uniform cell proliferation. These effects result in tumor growth made possible by increases in the surface areas at the interface of cell species\cite{wise2008three}. Models that can replicate such phenomena usually involve Ginzburg--Landau free energy functionals\cite{gurtin1996generalized} of species concentrations or volume fractions, nutrient concentrations, and, importantly, gradients of species concentrations as a representation of surface energies, a feature that leads to Cahn--Hilliard type models. The use of such phase-field formulations eliminates the need for enforcing conditions across interfaces between species and for tracking the interface, the locations of which are intrinsic features of the solution. Such non-sharp interfaces are often better characterizations of the actual moving interfaces between multiple species within a tumor than models employing sharp interfaces.
	
	Multiphase models of tumor growth have been proposed by several authors over the last decade. We mention as examples the multicomponent models of Araujo and McElwain\cite{araujo2004history} and the four, six, and ten species models of Refs.~\citen{hawkins2012numerical,lima2015analysis,lima2014hybrid}, respectively, the Cahn--Hilliard--Darcy models of Garcke \textit{et al}\cite{garcke2016global,garcke2018cahn} and the three-dimensional nonlinear multispecies models of Wise \textit{et al}\cite{wise2008three}. Additional references can be found in these works.
	
	Following this introduction, we introduce a general class of multispecies, phase-field models developed from balance laws and accepted cell-biological phenomena observed in cancer, in which cell velocity, present in mass convection, is modeled via at time-dependent, nonlinear flow field governed by a Darcy--Forchheimer--Brinkman law\cite{burcharth1995one,gu1991gravity,hall1995comparison,laushey1968darcy,zhu2014study}, which can be obtained by the means of mixture theory.\cite{rajagopal2007hierarchy,srinivasan2014thermodynamic}. In Sections \ref{Section_Notation}--\ref{Section_Nonlocal}, we develop a complete mathematical analysis of these classes of models proving existence of weak solutions through compactness arguments. We first consider the so-called local theory with nonlinear flow in which the evolution of the tumor volume fractions are influenced only by events in the neighborhood of each spatial point in the tumor mass and then we address nonlocal effects to depict long-range interactions of cell species. 
	
	To address the fundamental question of sensitivity of solutions of such nonlinear systems to variation in model parameters, we provide a detailed analysis of sensitivity for local models in Section \ref{Section_Sens}, calling on both statistical methods\cite{oden_2018,saltelli2000sensitivity,saltelli2008global,saltelli2010variance} of sensitivity analysis and data-dependent methods based on the notion of active subspaces\cite{constantine2014active,constantine2015active}. In Section \ref{Section_Numerical} we take up representative finite element approximations and numerical algorithms and present the results of numerical experiments, particularly focusing  on the effects of time-dependent nonlinear flow regimes on the evolution of tumor morphology. In a final section, we provide concluding remarks of the study. 
	
	\section{A Class of Local Models of Tumor Growth} \label{Section_Tumor}
	
	We consider a solid tumor mass $\mathcal{T}$ evolving in the interior of a bounded Lipschitz domain $\Omega\subset\mathbb{R}^d$, $d \leq 3$, over a time period $[0,T]$. At each point $x\in\Omega$, several cell species and other constituents exist which are differentiated according to their volume fractions, $\phi_\alpha$, $\alpha=1,2,\ldots,N$. The volume fractions of tumor cells is given by the scalar-valued field $\phi_T=\phi_T(x,t)$ and the volume-averaged velocity is denoted by $v$. The mass density of all $N$ species is assumed to be a single constant field, and the evolution of the tumor cells is governed by the evolution of proliferative cells with volume fraction $\phi_P$, hypoxic cells $\phi_H$, and necrotic cells $\phi_N$. The nutrient supply to the tumor is characterized by a constituent with volume fraction $\phi_\sigma=\phi_\sigma(x,t)$. 
	
	The tumor mass is conserved during its evolution, and this is assumed to be captured by the convective phase-field equation\cite{lima2014hybrid,lima2015analysis}
	\begin{eqnarray}
	\partial_t \phi_T+\text{div}(\phi_T v)=\text{div}(m_T(\phi_T,\phi_\sigma)\nabla\mu)+\lambda_T\phi_\sigma \phi_T(1-\phi_T)-\lambda_A\phi_T, 
	\label{Eq_DerivationT}
	\end{eqnarray}
	where $m_T$ is the mobility function, $\mu$ the chemical potential, and $\lambda_T$, $\lambda_A$ are the proliferation and apoptosis rates, respectively. Following Ref.~\citen{hawkins2012numerical,lima2015analysis,lima2014hybrid,wise2008three}, we consider
	\begin{eqnarray}
	\mu=\frac{\delta \mathcal{E}}{\delta \phi_T}=\Psi'(\phi_T)-\varepsilon_T^2\Delta\phi_T-\chi_0\phi_\sigma,
	\label{Eq_DerivationMu}
	\end{eqnarray}
	where $\delta \mathcal{E}/\delta \phi_T$ denotes the first variation of the Ginzburg--Landau free energy functional,
	\begin{eqnarray}
	\mathcal{E}(\phi_T,\phi_\sigma)= \int_\Omega \Psi(\phi_T)+\frac{\varepsilon_T^2}{2} |\nabla \phi_T|^2 - \chi_0 \phi_\sigma \phi_T \, \text{d}x.
	\label{Eq_DerivationE}
	\end{eqnarray}
	In (\ref{Eq_DerivationE}), $\Psi$ is a double-well potential with a prefactor $\overline{E}$ (such as $\Psi(\phi_T)=\overline{E}\phi_T^2(1-\phi_T)^2$), $\varepsilon_T$ is a parameter associated with the interface thickness separating cell species, and $\chi_0$ is the chemotaxis parameter. 
	The velocity $v$ is assumed to obey the time-dependent incompressible Darcy--Forchheimer--Brinkman law\cite{burcharth1995one,gu1991gravity,hall1995comparison,laushey1968darcy,zhu2014study}
	\begin{equation}\begin{aligned}
	\partial_t v+\alpha v &=\div(\nu(\phi_T,\phi_\sigma) \textup{D} v)-F_1|v|v-F_2|v|^2 v-\nabla p+(\mu+\chi_0\phi_\sigma) \nabla \phi_T,\\
	\div v &=0,
	\end{aligned}
	\label{Eq_DerivationV}
	\end{equation}
	where $\textup{D} v = \frac{1}{2} (\nabla v + (\nabla v)^\top)$ denotes the deformation-rate tensor.
	The nutrient concentration $\phi_\sigma$ is assumed to obey a convection-reaction-diffusion equation of the form,
	\begin{eqnarray}
	\partial_t \phi_\sigma+\div(\phi_\sigma v)=\div\!\big(m_\sigma(\phi_T,\phi_\sigma)(\delta^{-1}_\sigma\nabla\phi_\sigma-\chi_0\nabla\phi_T)\big)-\lambda_\sigma\phi_T\phi_\sigma,
	\label{Eq_DerivationS}
	\end{eqnarray}
	with $m_\sigma$ a mobility function and $\delta_\sigma$ and $\lambda_\sigma$ positive parameters. 
	
	Collecting \eqref{Eq_DerivationT},\eqref{Eq_DerivationMu},\eqref{Eq_DerivationV} and \eqref{Eq_DerivationS}, we arrive at a model governed by the system,
	\begin{equation} \begin{aligned}
	\partial_t \phi_T+ \div(\phi_T v) &=  \div (m_T(\phi_T,\phi_\sigma) \nabla \mu) + S_T(\phi_T,\phi_\sigma), \\
	\mu &= \Psi'(\phi_T)-\varepsilon_T^2 \Delta \phi_T -\chi_0 \phi_\sigma, \\
	\partial_t \phi_\sigma + \div(\phi_\sigma v)&=  \div\!\big(m_\sigma(\phi_T,\phi_\sigma) (\delta_\sigma^{-1} \nabla \phi_\sigma - \chi_0   \nabla \phi_T)\big) + S_\sigma(\phi_T,\phi_\sigma), \\
	\partial_t v + \alpha v &= \div (\nu(\phi_T,\phi_\sigma) \textup{D} v)- F_1 |v| v - F_2 |v|^2 v - \nabla p  +  S_v(\phi_T,\phi_\sigma),\\
	\div v &= 0,
	\end{aligned} 
	\label{Eq_Model} \end{equation}
	in the time-space domain $(0,T)\times \Omega$ with source functions $S_T, S_\sigma, S_v$ with properties laid down in Theorem \ref{Thm_Existence} of Section \ref{Section_Analysis}. We supplement the system with the following boundary and initial conditions,
	\begin{equation} \begin{aligned}\partial_n \phi_T = \partial_n \mu &= 0 &&\text{on } (0,T) \times \partial\Omega, \\
	\phi_\sigma &=1 &&\text{on }  (0,T) \times \partial\Omega, \\
	v &= 0 &&\text{on } (0,T) \times \partial\Omega, \\
	\phi_T(0)&=\phi_{T,0} &&\text{in } \Omega, \\
	\phi_\sigma(0)&=\phi_{\sigma,0} &&\text{in } \Omega, \\
	v(0) &= v_0 &&\text{in } \Omega,
	\end{aligned}
	\label{Eq_InitialBoundary} \end{equation}
	where $\phi_{T,0}, \phi_{\sigma,0}, v_0$ are given functions. Here, $\partial_n f = \nabla f \cdot n$ denotes the normal derivative of a function $f$ at the boundary $\partial \Omega$ with the outer unit normal $n$. 
	
	\section{Notation and Auxiliary Results} \label{Section_Notation}
	Notationally, we suppress the domain $\Omega$ when denoting various Banach spaces and write simply $L^p, H^m, W^{m,p}$. We equip these spaces with the norms $|\cdot|_{L^p}$, $|\cdot|_{H^m}$, $|\cdot|_{W^{m,p}}$, and, to simplify notation, we denote by $(\cdot,\cdot)$ the scalar product in $L^2$. The brackets $\langle \cdot,\cdot \rangle_{H^1}$ denote the duality pairing on $(H^1)'\times H^1$ and in the same way for the other spaces. In the case of $d$-dimensional vector functions, we write $[L^p]^d$ and in the same way for the other Banach spaces, but we do not make this distinction in the notation of norms, scalar products and applications with its dual. 
	
	Throughout this paper, $C<\infty$ stands for a generic constant. We recall the Poincar\'e and Korn inequalities\cite{roubicek},
	\begin{alignat*}{2}
	|f-\overline{f}|_{L^2} &\leq C |\nabla f|_{L^2} &&\quad\text{for all } f \in H^1, \\
	|f|_{L^2} &\leq C |\nabla f|_{L^2} &&\quad\text{for all } f \in H^1_0, \\
	|\nabla f |_{L^2} &\leq C |\textup{D} f|_{L^2} &&\quad \text{for all } f\in H_0^1,
	\end{alignat*}
	where $\overline{f}=\frac{1}{|\Omega|}\int_\Omega f(x) \,\textup{d} x$ is the mean of $f$. We define the spaces $H$, $V$, $V^\perp$, $L_0^2$ as follows:
	\begin{alignat*}{3}
	&H &&= \{ u \in [L^2]^d : \div  u = 0,  ~u \cdot n |_{\partial \Omega} =0\}, \\
	&V &&= \{ u \in [H^1_0]^d : \div  u = 0\}, \\
	&V^\perp &&=\{f \in [H^{-1}]^d : \langle f,u\rangle_{H^{-1} \times H_0^1} = 0 \text{ for all } u \in V\}, \\
	&L_0^2 &&= \{ u \in L^2 : \overline{u} = 0\},
	\end{alignat*}
	where, for $u \in H$, the divergence is meant in a distributional sense and its trace operator is well-defined; see Ref.~\citen{girault1979finite}.
	For a given Banach space $X$, we define the Bochner space\cite{evans2010partial,roubicek}
	$$L^p(0,T;X)=\{ u :(0,T) \to X: u \text{ Bochner measurable, } \int_0^T |u(t)|_X^p \,\text{d$t$} < \infty \},$$
	where $1 \leq p < \infty$, with the norm $\|u\|_{L^p X}^p = \int_0^T |u(t)|_X^p \,\textup{d} t$. For $p=\infty$, we equip $L^\infty(0,T;X)$ with the norm $\|u\|_{L^\infty X} = \esssup_{t\in (0,T)} |u(t)|_X$. We introduce the Sobolev--Bochner space,
	$$W^{1,p}(0,T;X)=\{ u \in L^p(0,T;X) : \partial_t u \in L^p(0,T;X) \},$$
	and the inverse Sobolev--Bochner space
	$$W^{-1,p}(0,T;X)=\mathscr{L}(W_0^{1,q}(0,T);X),$$
	where $q=1/(1-1/p)$ is the H\"older conjugate of $p$.
	
	Let $X$, $Y$, $Z$ be Banach spaces such that $X$ is compactly embedded in $Y$ and $Y$ is continuously embedded in $Z$, i.e. $X\hookrightarrow \hookrightarrow Y \hookrightarrow Z$. In the proof of the existence theorem below we make use of the Aubin--Lions compactness lemma, see Corollary 4 in Ref.~\citen{simon1986compact}, 
	\begin{equation}\begin{alignedat}{2} L^p(0,T;X) \cap W^{1,1}(0,T;Z) &\hookrightarrow \hookrightarrow L^p(0,T;Y), &&\quad 1\leq p<\infty, \\
	L^\infty(0,T;X) \cap W^{1,r}(0,T;Z) &\hookrightarrow \hookrightarrow C([0,T];Y), &&\quad r >1,
	\label{Eq_AubinLions}
	\end{alignedat}
	\end{equation}
	and of the following continuous embeddings, see Theorem 3.1, Chapter 1 in Ref.~\citen{lions2012non} and Theorem 2.1 in Ref.~\citen{strauss1966continuity},
	\begin{alignat}{2} L^2(0,T;Y) &\cap H^1(0,T;Z) &&\hookrightarrow C([0,T];[Y,Z]_{1/2}), \label{Eq_ContEmbedding1}\\
	L^\infty(0,T;Y) &\cap C_w([0,T];Z) &&\hookrightarrow C_w([0,T];Y),\label{Eq_ContEmbedding2}
	\end{alignat}
	where $[Y,Z]_{1/2}$ denotes the interpolation space between $Y$ and $Z$; see Definition 2.1, Chapter 1 in Ref.~\citen{lions2012non} for a precise definition. Here, $C_w([0,T];Y)$ denotes the space of the weakly continuous functions on the interval $[0,T]$ with values in $Y$. 
	\begin{lemma}[Gronwall, cf. Lemma 3.1 in Ref.~\citen{garcke2017well}] Let $u,v \in C([0,T];\mathbb{R}_{\geq 0})$. If there are constants $C_1,C_2<\infty$ such that
		$$u(t)+v(t) \leq C_1 + C_2 \int_0^t u(s) \, \textup{d}s \quad \text{ for all } t\in [0,T],$$
		then it holds that $u(t) + v(t) \leq C_1 e^{C_2 T}$ for all $t \in [0,T]$.
		\label{Lem_Gronwall}
	\end{lemma}
	\begin{lemma}[De Rham, cf. Lemma 7 in Ref.~\citen{tartar1976nonlinear}] If $f \in V^\perp$,  then there is a unique $q \in L^2_0$ such that $f=\nabla q$ and $|q|_{L^2} \leq C |f|_{H^{-1}}$. In other words, there exists an operator $L \in \mathscr{L}(V^\perp,L_0^2)$ such that $\nabla \circ L = \textup{Id}$.
		\label{Lem_Rham}
	\end{lemma}
	Introducing the Nemyzki operator of $L$,
	\begin{gather*}
	\mathcal{N}_L : W^{-1,\infty}(0,T;V^\perp) \to W^{-1,\infty}(0,T;L_0^2),\\
	(\mathcal{N}_L w)(\eta) = L\big(w(\eta)\big) \text{ for all } \eta \in W_0^{1,1}(0,T),
	\end{gather*}
	we get the following corollary of the de Rham lemma.
	\begin{corollary}
		If $w \in W^{-1,\infty}(0,T;V^\perp)$, then there exists a unique element $p \in W^{-1,\infty}(0,T;L_0^2)$ such that $w=\nabla p$ in $W^{-1,\infty}(0,T;H^{-1})$.
		\label{Cor_Rham}
	\end{corollary}
	\begin{proof}
		Since $w(\eta) \in V^\perp$ for all $\eta \in W_0^{1,1}(0,T)$, we have
		$$w(\eta)=\nabla L\big(w(\eta)\big) = \nabla (\mathcal{N}_L w)(\eta),  $$
		and we define $p=\mathcal{N}_L w \in W^{-1,\infty}(0,T;L_0^2)$, which is clearly unique.
	\end{proof}
	
	\section{Analysis of the Local Model} \label{Section_Analysis}
	In this section, we first state the definition of a weak solution to the system (\ref{Eq_Model}) with the boundary and initial conditions (\ref{Eq_InitialBoundary}), and then we establish the existence of a weak solution by employing the Faedo--Galerkin method. 
	
	To simplify notation, we introduce the abbreviations
	\begin{alignat*}{6}
	&m_T &&= m_T(\phi_T,\phi_\sigma),
	\quad &&m_\sigma &&= m_\sigma(\phi_T,\phi_\sigma), \quad &&\nu &&= \nu(\phi_T,\phi_\sigma), \\
	&S_T &&= S_T(\phi_T,\phi_\sigma), \quad &&S_\sigma &&= S_\sigma(\phi_T,\phi_\sigma), \quad &&S_v &&= S_v(\phi_T,\phi_\sigma). 
	\end{alignat*}
	\begin{definition}[Weak solution] We call a quadruple $(\phi_T,\mu,\phi_\sigma,v)$ a weak solution of system \eqref{Eq_Model} if the functions $\phi_T,\mu,\phi_\sigma: (0,T) \times \Omega \to \mathbb{R}$, $v:(0,T) \times \Omega \to \mathbb{R}^d$ have the regularity
		\begin{alignat*}{2}
		&\phi_T &&\in H^1(0,T;(H^1)') \cap L^2(0,T;H^1), \\
		&\mu &&\in L^2(0,T;H^1), \\[-0.15cm]
		&\phi_\sigma &&\in W^{1,\tfrac{4}{d}}(0,T;H^{-1}) \cap \big(1+L^2 (0,T;H^1_0)\big) , \\[-0.15cm] 
		&v &&\in W^{1,\tfrac{4}{d}}(0,T;V') \cap L^4(0,T;[L^4]^d) \cap L^2 (0,T;V), \end{alignat*}
		fulfill the initial data $\phi_T(0)=\phi_{T,0}$, $\phi_{\sigma}(0)=\phi_{\sigma,0}$, $v(0)=v_0$, and satisfy the following variational form of \eqref{Eq_Model},
		\begin{equation}
		\begin{aligned}
		\langle \partial_t \phi_T, \varphi_1 \rangle_{H^1}&= (\phi_T v, \nabla \varphi_1) - (m_T \nabla \mu, \nabla \varphi_1) + (S_T,\varphi_1), \\
		(\mu, \varphi_2) &= (\Psi'(\phi_T),\varphi_2)+\varepsilon_T^2 (\nabla \phi_T, \nabla \varphi_2) - \chi_0 (\phi_\sigma, \varphi_2), \\
		\langle \partial_t \phi_\sigma, \varphi_3 \rangle_{H_0^1} &= (\phi_\sigma v,\nabla \varphi_3) - \delta_\sigma^{-1} (m_\sigma \nabla \phi_\sigma,\nabla \varphi_3) + \chi_0 (m_\sigma \nabla \phi_T, \nabla \varphi_3) + (S_\sigma, \varphi_3), \\
		\langle \partial_t v, \varphi_4 \rangle_V &= - \alpha( v, \varphi_4) - (\nu \textup{D} v, \textup{D} \varphi_4) - F_1( |v| v, \varphi_4) - F_2 ( |v|^2 v, \varphi_4)  + (S_v,\varphi_4),
		\end{aligned}
		\label{Eq_ModelWeak}
		\end{equation}
		for all $\varphi_1,\varphi_2 \in H^1, \varphi_3 \in H_0^1, \varphi_4 \in V$. 
		\label{Def_Weak}
	\end{definition}
	In the variational form, we use the divergence-free space $V$ as the test function space of the Darcy--Forchheimer--Brinkman equation. Therefore, the pressure $p$ is eliminated from the equation. After proving the exsitence of a weak solution, we can associate a distributional pressure to the solution quadruple using the de Rham lemma; see Lemma \ref{Lem_Rham}. 
	
	A first principal result of this paper involves stating the existence of a weak solution to the model (\ref{Eq_Model}) in the sense of Definition \ref{Def_Weak}. 
	
	\begin{theorem}[Existence of a global weak solution] Let the following assumptions hold: \\
		\textup{(A1)} $\Omega \subset \mathbb{R}^d$, $d\in \{2,3\}$, is a bounded Lipschitz domain and $T>0$.\\
		\textup{(A2)} $\phi_{T,0} \in H^1$, $\phi_{\sigma,0} \in L^2$, $v_0 \in H$. \\
		\textup{(A3)} $m_T,m_\sigma,\nu \in C_b(\mathbb{R}^2)$ such that $m_0 \leq m_T(x),m_\sigma(x),\nu(x) \leq m_\infty$ for positive constants $m_0,m_\infty<\infty$. \\
		\textup{(A4)} $S_T,S_\sigma,S_v$ are of the form $S_T=\lambda_T  \phi_\sigma g(\phi_T) - \lambda_A \phi_T$, $S_\sigma=-\lambda_\sigma \phi_\sigma h(\phi_T)$ for $g,h \in C_b(\mathbb{R})$, 
		and $S_v=(\mu+\chi_0 \phi_\sigma) \nabla \phi_T$, $\lambda_T,\lambda_A,\lambda_\sigma \geq 0$. \\
		\textup{(A5)} $\Psi \in C^2(\mathbb{R})$ is such that $\Psi(x)\geq C (|x|^2-1)$, $|\Psi'(x)| \leq C(|x|+1)$, and $|\Psi''(x)|\leq C(|x|^4+1)$. \\
		Then there exists a weak solution quadruple $(\phi_T, \mu, \phi_\sigma, v)$ to \eqref{Eq_Model} in the sense of Definition \ref{Def_Weak}. Moreover, the solution quadruple has the regularity:
		\begin{alignat*}{2}
		&\phi_T &&\in C([0,T];L^2) \cap C_w([0,T];H^1), \\
		&\phi_\sigma &&\in \begin{cases} C([0,T];L^2), &d=2, \\ C([0,T];H^{-1})  \cap C_w([0,T];L^2), &d=3,\end{cases} \\ &v &&\in \begin{cases} C([0,T];H), &d=2, \\ C([0,T];V') \cap C_w([0,T];H), &d=3. \end{cases} 
		\end{alignat*}
		Additionally, there is a unique $p \in W^{-1,\infty}(0,T;L_0^2)$ such that $(\phi_T,\mu,\phi_\sigma,v,p)$ is a solution quintuple to \eqref{Eq_Model} in the distributional sense.
		\label{Thm_Existence}
	\end{theorem}
	\begin{proof} 
		To prove the existence of a weak solution, we use the Faedo--Galerkin method\cite{evans2010partial,salsa2016partial} and semi-discretize the original problem in space. The discretized model can be formulated as an ordinary differential equation system and by the Cauchy--Peano theorem we conclude the existence of a discrete solution. Having energy estimates, we deduce from the Banach--Alaogulu theorem the existence of limit functions which eventually form a weak solution. This method has fared popularly in the analysis of tumor models, e.g. see Refs.~\citen{frigeri2015diffuse,garcke2016global,garcke2016cahn,garcke2018cahn,ebenbeck2018analysis,garcke2017well,jiang2015well,lam2017thermodynamically,lowengrub2013analysis}. 
		
		\subsection*{Discretization in space} 
		We introduce the discrete spaces 
		\begin{align*} 
		W_k &=\text{span}\{ w_1,\dots,w_k\}, \\
		Y_k &=\text{span}\{ y_1,\dots,y_k\}, \\
		Z_k &=\text{span}\{ z_1,\dots,z_k\}, 
		\end{align*}
		where $w_j, y_j : \Omega \to \mathbb{R}, z_j : \Omega \to \mathbb{R}^d$ are the eigenfunctions to the eigenvalues $\lambda_j^w, \lambda_j^y, \lambda_j^z \in \mathbb{R}$ of the following respective problems
		\begin{alignat*}{1} 
		&\begin{cases} \begin{aligned}
		-\Delta w_j &= \lambda_j^w w_j &&\text{in } \Omega, \\
		\partial_n w_j &= 0 &&\text{on } \partial\Omega,
		\end{aligned} \end{cases} \\
		&\begin{cases} \begin{aligned}
		-\Delta y_j &= \lambda_j^y y_j &&\text{in } \Omega, \\
		y_j &= 0 &&\text{on } \partial\Omega,
		\end{aligned} \end{cases} \\
		&\begin{cases} \begin{aligned}
		-\Delta z_j &= \lambda_j^z z_j &&\text{in } \Omega, \\
		\div z_j &=0 &&\text{in } \Omega, \\
		z_j &= 0 &&\text{on } \partial\Omega.
		\end{aligned} \end{cases} 
		\end{alignat*}
		Since the Laplace operator is a compact, self-adjoint, injective operator, we conclude by the spectral theorem\cite{boyer2012mathematical,brezis2010functional,robinson2001infinite} that 
		\begin{alignat*}{3}
		&\{w_j\}_{j \in \mathbb{N}} &&\text{ is an orthonormal basis in } L^2 &&\text{ and orthogonal in } H^1, \\
		&\{y_j\}_{j \in \mathbb{N}} &&\text{ is an orthonormal basis in } L^2 &&\text{ and orthogonal in } H_0^1, \\
		&\{z_j\}_{j \in \mathbb{N}} &&\text{ is an orthonormal basis in } H &&\text{ and orthogonal in } V.
		\end{alignat*}
		Exploiting orthonormality of the eigenfunctions, we deduce that $W_k$ and $Y_k$ are dense in $L^2$, and $Z_k$ is dense in $H$. We introduce the orthogonal projections,
		$$\Pi_{W_k}: L^2 \to W_k, \quad \Pi_{Y_k} : L^2 \to Y_k, \quad  \Pi_{Z_k} : H \to Z_k,$$
		which can be written as $$\Pi_{W_k} u = \sum_{j=1}^k (u,w_j) w_j \quad \text{for all } u \in H^1,$$ and analogously for $\Pi_{Y_k}$ and $\Pi_{Z_k}$.
		
		We next consider the Galerkin approximations
		\begin{equation}\begin{aligned}
		\phi_T^k (t,x) &= \sum_{j=1}^k \alpha_j(t) w_j(x),
		&&\quad \mu^k (t,x) = \sum_{j=1}^k \beta_j(t) w_j(x), \\
		\phi_\sigma^k (t,x) &= 1+ \sum_{j=1}^k \gamma_j(t) y_j(x), 
		&&\quad v^k (t,x) = \sum_{j=1}^k \delta_j(t) z_j(x), 
		\end{aligned}
		\label{Eq_GalerkinApprox}
		\end{equation}
		where $\alpha_j, \beta_j, \gamma_j, \delta_j : (0,T) \to \mathbb{R}$ are coefficient functions for $j \in \{1,\dots,k\}$. 
		To simplify the notation we set
		\begin{alignat*}{5}
		m_T^k &= m_T(\phi_T^k,\phi_\sigma^k), \quad &m_\sigma^k &= m_\sigma(\phi_T^k,\phi_\sigma^k), \quad &\nu^k &=\nu(\phi_T^k,\phi_\sigma^k), \\
		S_T^k &= S_T(\phi_T^k,\phi_\sigma^k), \quad
		&S_\sigma^k &= S_\sigma(\phi_T^k,\phi_\sigma^k), \quad
		&S_v^k &= S_v(\phi_T^k, \phi_\sigma^k).
		\end{alignat*}
		The Galerkin system of the model \eqref{Eq_ModelWeak} then reads
		\begin{subequations}
			\label{Eq_ModelGalerkin}
			\begin{align}
			\langle \partial_t \phi_T^k,w_j \rangle_{H^1} &= (\phi_T^k v^k, \nabla w_j) - (m_T^k \nabla \mu^k, \nabla w_j) + (S_T^k,w_j), \label{Eq_ModelGalerkinT}\\
			(\mu^k, w_j) &= (\Psi'(\phi_T^k),w_j)+\varepsilon_T^2 (\nabla \phi_T^k, \nabla w_j) - \chi_0 (\phi_\sigma^k, w_j),\label{Eq_ModelGalerkinMu} \\
			\langle \partial_t \phi_\sigma^k,y_j \rangle_{H_0^1} &= (\phi_\sigma^k v^k,\nabla y_j) - \delta_\sigma^{-1}(m_\sigma^k \nabla \phi_\sigma^k,\nabla y_j) + \chi_0 (m_\sigma^k \nabla \phi_T^k, \nabla y_j) \label{Eq_ModelGalerkinS}  + (S_\sigma^k, y_j), \\
			\langle \partial_t v^k, z_j \rangle_V &= - \alpha (v^k,z_j) - (\nu^k \textup{D} v^k, \textup{D}  z_j) - F_1 (|v^k| v^k, z_j) \label{Eq_ModelGalerkinV} -F_2 (|v^k|^2 v^k, z_j) + (S_v^k,z_j),
			\end{align}
		\end{subequations}
		for all $j \in \{1,\dots,k\}$. We equip this system with the initial data 
		\begin{equation}
		\begin{aligned} \phi_{T}^k(0) &= \Pi_{W_k} \phi_{T,0} &&\text{in } L^2, \\
		\phi_\sigma^k(0) &=1+ \Pi_{Y_k} \phi_{\sigma,0} &&\text{in } L^2, \\
		v^k(0) &= \Pi_{Z_k} v_0 &&\text{in } H. \end{aligned}
		\label{Eq_GalerkinInitial} \end{equation}
		After inserting the Galerkin ansatz functions \eqref{Eq_GalerkinApprox} into the system \eqref{Eq_ModelGalerkin}, one can see that the Galerkin system is equivalent to a system of nonlinear ordinary differential equations in the $4k$ unknowns $\{ \alpha_j, \beta_j, \gamma_j, \delta_j \}_{1\leq j \leq k}$ with the initial data
		$$\begin{aligned} \alpha_j(0) &= (\phi_{T,0},w_j) &&\text{in } \Omega, \\
		\gamma_j(0) &= (\phi_{\sigma,0},y_j) &&\text{in } \Omega, \\
		\delta_j(0) &= (v_0,z_j) &&\text{in } \Omega. \end{aligned}$$
		Due to the continuity of the nonlinear functions $\Psi', m_T, m_\sigma, \nu$, the existence of solutions to (\ref{Eq_ModelGalerkin}) with data (\ref{Eq_GalerkinInitial}) follows from the standard theory of ordinary differential equations, according to the Cauchy--Peano theorem, see Theorem 1.2, Chapter 1 in Ref.~\citen{coddington1955theory}. We thus have local-in-time existence of a continuously differentiable solution quadruple, $$\begin{aligned} (\phi_T^k,\mu^k,\phi_\sigma^k,v^k) &\in C^1([0,T_k];W_k) \times C^1([0,T_k];W_k) \times \big(1+C^1([0,T_k];Y_k)\big) \times C^1([0,T_k];Z_k)\end{aligned}$$ to the Galerkin problem \eqref{Eq_ModelGalerkin} on some sufficiently short time interval $[0, T_k]$. 
		\subsection*{Energy estimates} Next, we extend the existence interval to $[0,T]$ by deriving $k$-independent estimates. In particular, these estimates allow us to deduce that the solution sequences converge to some limit functions as $k \to \infty$. It will turn out that exactly these limit functions will form a weak solution to our model \eqref{Eq_Model} in the sense of Definition \ref{Def_Weak}. 
		
		Testing \eqref{Eq_ModelGalerkinT} with $\mu^k+\chi_0 \phi_\sigma^k$, \eqref{Eq_ModelGalerkinMu} with $-\partial_t \phi_T^k$, \eqref{Eq_ModelGalerkinS} with $K(\phi_\sigma^k-1)$, $K>0$ to be specified, and \eqref{Eq_ModelGalerkinV} with $v^k$, gives the equation system,
		\begin{align*}
		\langle \partial_t \phi_T^k, \mu^k \rangle_{H^1} + \chi_0\langle \partial_t \phi_T^k, \phi_\sigma^k \rangle_{H^1} &= (\phi_T^k v^k, \nabla \mu^k+\chi_0 \nabla \phi_\sigma^k) \! - \! (m_T^k \nabla \mu^k, \nabla \mu^k+\chi_0 \nabla \phi_\sigma^k)+ (S_T^k,\mu^k+\chi_0 \phi_\sigma^k), \\
		-\langle \partial_t \phi_T^k,\mu^k \rangle_{H^1} &= -\langle\partial_t \phi_T^k,\Psi'(\phi^k_T)\rangle_{H^1} + \chi_0 \langle \partial_t \phi_T^k,\phi^k_\sigma\rangle_{H^1} - \varepsilon_T^2 \langle \nabla \partial_t \phi_T^k,\nabla \phi^k_T\rangle_{H^1}, \\
		\langle \partial_t\phi_\sigma^k,K(\phi_\sigma^k-1) \rangle_{H_0^1} &= (\phi_\sigma^k v^k,  K\nabla\phi_\sigma^k) \! - \! K(m_\sigma^k (\delta_\sigma^{-1}\nabla\phi_\sigma^k - \chi_0 \nabla \phi_T^k), \nabla \phi_\sigma^k ) + K (S_\sigma^k, \phi_\sigma^k -1), \\
		\langle \partial_t v^k,v^k\rangle_V &= -  (\nu^k \textup{D} v^k, \textup{D} v^k) - F_1 (|v^k| v^k, v^k) - F_2 (|v^k|^2 v^k,v^k)  - (\alpha v^k, v^k) +(S_v^k,v^k).
		\end{align*}
		We observe that the tested convective terms cancel each other together with the tested source term in the velocity equation. Indeed, noting that $v^k$ is divergence-free, we have by integration by parts 
		\begin{align*}
		(\phi_T^k v^k, \nabla \mu^k+\chi_0 \nabla \phi_\sigma^k)\!+\!(\phi_\sigma^k v^k,K \nabla\phi_\sigma^k) &=-(v^k \cdot \nabla \phi_T^k, \mu^k +\chi_0 \phi_\sigma^k ) \! - \! K (v^k \cdot \nabla \phi_\sigma^k, \phi_\sigma^k) \\ &= - (v^k, S_v^k),
		\end{align*}
		see also (A4) of Theorem \ref{Thm_Existence} for the assumed form of $S_v^k$. Here, we also used that $(\phi_\sigma^k v^k, \nabla \phi_\sigma^k)=-(v^k \cdot \nabla \phi_\sigma^k,\phi_\sigma^k)$ and therefore this term vanishes.
		
		Adding the four tested equations results in
		\begin{equation}
		\begin{aligned}
		&\frac{\text{d} }{\text{d} t} \bigg[ |\Psi(\phi_T^k)|_{L^1} +\frac{\varepsilon_T^2}{2} |\nabla \phi_T^k|^2_{L^2} + \frac{K}{2} |\phi_\sigma^k-1|^2_{L^2} +\frac12 |v^k|_{L^2} \bigg] 
		+ \delta_\sigma^{-1} K (m_\sigma^k, |\nabla \phi_\sigma^k|^2) \\ &\quad\,+ (m_T^k,|\nabla \mu^k|^2)+ (\nu^k,|\textup{D} v^k|^2) + F_1 |v^k|_{L^3}^3 + F_2 |v^k|_{L^4}^4 + \alpha |v^k|^2_{L^2} 
		\\ &=-\chi_0(m_T^k \nabla \mu^k, \nabla \phi_\sigma^k)+(S_T^k,\mu^k+\chi_0 \phi_\sigma^k) +  \chi_0 K (m_\sigma^k \nabla \phi_T^k, \nabla \phi_\sigma^k) +  K (S_\sigma^k,\phi_\sigma^k),
		\end{aligned}
		\label{Eq_Energy1}
		\end{equation}
		where $K>0$ can still be chosen appropriately.
		
		We now estimate the terms on the right hand side of this inequality. The Poincar\'e inequality applied to $\mu^k$ and $\phi_\sigma^k$ gives
		\begin{alignat*}{2} 
		|\mu^k|_{L^2} &\leq |\mu^k-\overline{\mu^k}|_{L^2} + |\overline{\mu^k}|_{L^2} &&\leq C |\nabla \mu^k|_{L^2} + |\Omega|^{-1/2} |\mu^k|_{L^1},\\
		|\phi_\sigma^k|_{L^2} &\leq |\phi_\sigma^k-1|_{L^2} + |1|_{L^2} &&\leq C |\nabla \phi_\sigma^k|_{L^2} + |\Omega|^{1/2}.
		\end{alignat*} 
		Therefore, we can estimate the tested source term $S_T^k$ using its assumed representation, see (A4) of Theorem \ref{Thm_Existence}, the Poincar\'e inequality, and the $\varepsilon$-Young inequality,
		\begin{equation}
		\begin{aligned}
		(S_T^k,\mu^k+\chi_0 \phi_\sigma^k) &= (\lambda_T \phi_\sigma^k g(\phi_T^k) - \lambda_A \phi_T^k, \mu^k+\chi_0 \phi_\sigma^k)
		\\ &= \lambda_T (\phi_\sigma^k g(\phi_T^k),\mu^k) + \lambda_T \chi_0 \big(|\phi_\sigma^k|^2,g(\phi_T^k)\big) - \lambda_A (\phi_T^k,\mu^k) 
		 - \lambda_A \chi_0 (\phi_T^k,\phi_\sigma^k)
		\\ &\leq C\left( |\phi_\sigma^k|_{L^2} |\mu^k|_{L^2} +  |\phi_\sigma^k|_{L^2}^2 + |\phi_T^k|_{L^2} |\mu^k|_{L^2} + |\phi_T^k|_{L^2} |\phi_\sigma^k|_{L^2}    \right)
		\\ &\leq \frac{m_0}{4} |\nabla \mu^k|_{L^2}^2 + C\left(1+ |\phi_\sigma^k|_{L^2}^2+|\mu^k|_{L^1}^2+|\phi_T^k|_{L^2}^2 \right).
		\end{aligned}
		\label{Eq_EnergyST}
		\end{equation}
		Testing \eqref{Eq_ModelGalerkinMu} with $1 \in H^1$, gives, together with the growth assumption (A5) on $\Psi'$,
		\begin{equation}
		\begin{aligned}
		|\mu^k|_{L^1} \leq |\Psi'(\phi_T^k)|_{L^1} + \chi_0 |\phi_\sigma^k|_{L^1}  &\leq C\left(1+ |\phi_T^k|_{L^2} + |\phi_\sigma^k|_{L^2}\right).
		\end{aligned}
		\label{Eq_EnergyMu}
		\end{equation}
		Similarly, using the representation of $S_\sigma^k$, see (A4), we get
		\begin{equation} K (S_\sigma^k,\phi_\sigma^k)= -K \lambda_\sigma (h(\phi_T^k),|\phi_\sigma^k|^2) \leq C K |\phi_\sigma^k|_{L^2}^2,
		\label{Eq_EnergySs}
		\end{equation}
		$K$ being a positive constant appearing in (\ref{Eq_Energy1}) to be chosen below.
		
		Now, using \eqref{Eq_EnergyST}--\eqref{Eq_EnergySs}, the H\"older inequality and the $\varepsilon$-Young inequality, we can estimate the right-hand side in \eqref{Eq_Energy1} as follows:
		\begin{align*}
		\text{(RHS)}  &\leq \chi_0 m_\infty |\nabla \mu^k|_{L^2} |\nabla \phi_\sigma^k|_{L^2}+\frac{m_0}{4} |\nabla \mu^k|_{L^2}^2 +C \left( 1+|\phi_\sigma^k|_{L^2}^2 + |\phi_T^k|_{L^2}^2 \right) \\
		&\quad\, + \chi_0 K m_\infty |\nabla \phi_T^k|_{L^2} |\nabla \phi_\sigma^k|_{L^2} + C K |\phi_\sigma^k|_{L^2}^2
		\\ &\leq \frac{m_0}{2} |\nabla \mu^k|^2_{L^2} \! + CK \left(1+ |\phi_\sigma^k-1|_{L^2}^2 + |\nabla \phi_T^k|_{L^2}^2 \right) \! + C \left( 1+ |\nabla \phi_\sigma^k|_{L^2}^2 + |\phi_T^k|_{L^2}^2 \right)\! .
		\end{align*}
		We note that in this inequality $C$ is independent of $K$. This is important since we choose $K$ such that $\delta_\sigma^{-1} K m_0 >C$ so that we can absorb $|\nabla \phi_\sigma^k|_{L^2}$ from the right hand side. Indeed, we get the following inequality from \eqref{Eq_Energy1},
		\begin{align*}
		&\frac{\textup{d} }{\textup{d} t} \bigg[ |\Psi(\phi_T^k)|_{L^1} + \frac{\varepsilon_T^2}{2} |\nabla \phi_T^k|^2_{L^2} +  \frac{K}{2} |\phi_\sigma^k-1|^2_{L^2} + \frac12 |v^k|_{L^2}^2 \bigg] 
		\\ &\quad\, + \left( \frac{K m_0}{\delta_\sigma} - C \right) |\nabla \phi_\sigma^k|^2_{L^2} + \frac{m_0}{2} |\nabla \mu^k|^2_{L^2} + m_0 |\textup{D} v^k|^2_{L^2} +  F_2 |v^k|_{L^4}^4
		\\ &\leq  C \left(1+ |\nabla \phi_T^k|^2_{L^2} + |\phi_\sigma^k-1|_{L^2}^2+ |\phi_T^k|_{L^2}^2 \right).
		\end{align*}
		Integrating this inequality over $(0,t)$, $t\in (0,T_k)$, and using the growth assumption (A5) on $\Psi$ gives
		\begin{align*}
		& |\phi_T^k(t)|_{L^2}^2 + |\nabla \phi_T^k(t)|^2_{L^2} +  |\phi_\sigma^k(t)-1|^2_{L^2} + |v^k(t)|_{L^2}^2 +  \|\nabla \phi_\sigma^k\|^2_{L^2(0,t;L^2)}  \\ &\quad\,+ \|\nabla \mu^k\|^2_{L^2(0,t;L^2)}  +  \|\textup{D} v^k\|^2_{L^2(0,t;L^2)} +  \|v^k\|_{L^4(0,t;L^4)}^4  \\ &\quad\, - C\left( \|\nabla \phi_T^k\|^2_{L^2(0,t;L^2)} + \|\phi_\sigma^k-1\|_{L^2(0,t;L^2)}^2 + \|\phi_T^k\|_{L^2(0,t;L^2)}\right)
		\\ &\leq  C \left(1+|\Psi(\phi^k_T(0))|_{L^1} + |\nabla \phi^k_T(0)|_{L^2}^2 + |\phi^k_{\sigma}(0)-1|_{L^2}^2+|v^k(0)|_{L^2}^2  \right).
		\end{align*}
		Applying the Gronwall lemma and taking the essential supremum over $t \in (0,T_k)$, gives
		\begin{equation}
		\begin{aligned}
		& \|\phi_T^k\|_{L^\infty(0,T_k;L^2)}^2 + \|\nabla \phi_T^k\|^2_{L^\infty(0,T_k;L^2)} +  \|\phi_\sigma^k-1\|^2_{L^\infty(0,T_k;L^2)} + \|v^k\|_{L^\infty(0,T_k;L^2)}^2
		\\ &\quad \, + \|\nabla \phi_\sigma^k\|^2_{L^2(0,T_k;L^2)}  + \|\nabla \mu^k\|^2_{L^2(0,T_k;L^2)} +   \|\textup{D} v^k\|^2_{L^2(0,T_k;L^2)} +  \|v^k\|_{L^4(0,T_k;L^4)}^4
		\\ &\leq  C \left(1+|\Psi(\phi_{T}^k(0))|_{L^1} + |\nabla \phi_T^k(0)|^2_{L^2} + |\phi_{\sigma}^k(0)-1|^2_{L^2}+|v^k(0)|^2_{L^2} \right) e^{CT}.
		\end{aligned}
		\label{Eq_Nonuniform}
		\end{equation}
		
		We have chosen the initial values of the Galerkin approximations as the orthogonal projections of the initial values of their counterpart, see \eqref{Eq_GalerkinInitial}. The operator norm of an orthogonal projection is bounded by $1$ and, therefore, uniform estimates are obtained in (\ref{Eq_Nonuniform}); for example
		$$|v^k(0)|_{L^2}^2 = |\Pi_{Z_k} v_{0}|_{L^2}^2 \leq |v_0|_{L^2}^2.$$
		We note that we have to invoke the growth estimate (A5) to treat the term involving $\Psi$ in the following way:
		$$|\Psi(\phi_T^k(0))|_{L^1} \leq C+ C |\phi_T^k(0)|_{L^2}^2 = C + C |\Pi_{W_k} \phi_{T,0}|_{L^2}^2 \leq C + C |\phi_{T,0}|_{L^2}^2.$$
		Now, these $k$-independent estimates allow us to extend the time interval by setting $T_k=T$ for all $k \in \mathbb{N}$.
		Therefore, we have the final uniform energy estimate,
		\begin{equation}\label{Eq_FinalEnergy}
		\begin{aligned}
		& \|\phi_T^k\|_{L^\infty H^1}^2 + \|\mu^k\|^2_{L^2H^1}+  \|\phi_\sigma^k\|^2_{L^\infty L^2}  +  \|\phi_\sigma^k-1\|^2_{L^2H^1_0} + \|v^k\|_{L^\infty H}^2  + \|v^k\|^2_{L^2 V} +  \|v^k\|_{L^4L^4}^4
		\\ &\leq  C \times \left(1+ |\phi_{T,0}|^2_{H^1} + |\phi_{\sigma,0}|^2_{L^2}+|v_0|^2_{L^2} \right)  \times \exp(CT).
		\end{aligned}
		\end{equation}
		
		\subsection*{Weak convergence} 
		Next, we prove that there are subsequences of $\phi_T^k,\mu^k,\phi_\sigma^k,v^k$, which converge to a weak solution of our model (\ref{Eq_Model}) in the sense of Definition \ref{Def_Weak}. From the energy estimate (\ref{Eq_FinalEnergy}) we deduce that
		\begin{equation}\begin{aligned}
		\{\phi_T^k\}_{k \in \mathbb{N}} &\text{ is bounded in } L^\infty(0,T;H^1), \\
		\{\mu^k\}_{k \in \mathbb{N}} &\text{ is bounded in } L^2(0,T;H^1), \\
		\{\phi_\sigma^k\}_{k \in \mathbb{N}} &\text{ is bounded in } L^\infty(0,T;L^2) \cap \big(1+L^2(0,T;H^1_0)\big), \\
		\{v^k\}_{k \in \mathbb{N}} &\text{ is bounded in } L^\infty(0,T;H) \cap L^2(0,T;V) \cap L^4(0,T;[L^4]^d),
		\end{aligned} \label{Eq_Bounds1} \end{equation}
		and, by the Banach--Alaoglu theorem, these bounded sequences have weakly convergent subsequences. By a typical abuse of notation, we drop the subsequence index. Consequently, there are functions $\phi_T, \mu, \phi_\sigma:(0,T)\times \Omega \to \mathbb{R}$, $v :(0,T) \times \Omega \to \mathbb{R}^d$ such that
		\begin{equation}
		\begin{alignedat}{2}
		\phi_T^k &\longweak \phi_T &&\text{ weakly-$*$ in } L^\infty(0,T;H^1),
		\\
		\mu^k &\longweak \mu &&\text{ weakly\phantom{-*} in } L^2(0,T;H^1), \\
		\phi_\sigma^k &\longweak \phi_\sigma &&\text{ weakly-$*$ in } L^\infty(0,T;L^2) \cap \big(1+L^2(0,T;H^1_0)\big), \\
		v^k &\longweak v &&\text{ weakly-$*$ in } L^\infty(0,T;H) \cap L^2(0,T;V) \cap L^4(0,T;[L^4]^d),
		\end{alignedat}
		\label{Eq_WeakConv}
		\end{equation}
		as $k \to \infty$. 
		
		\subsection*{Strong convergence} We now consider taking the limit $k \to \infty$ in the Galerkin system \eqref{Eq_ModelGalerkin} in hopes to attain the initial variational system \eqref{Eq_ModelWeak}. Since the equations in \eqref{Eq_ModelGalerkin} are nonlinear in $\phi_T^k,\phi_\sigma^k,v^k$, we want to achieve strong convergence of these sequences before we take the limit in \eqref{Eq_ModelGalerkin}. Therefore, our goal is to bound their time derivatives and applying the Aubin--Lions lemma (\ref{Eq_AubinLions}).
		
		Let $(\zeta,\eta,\xi)$ be such that $\zeta \in L^2(0,T;H^1)$, $\eta  \in L^{4/(4-d)}(0,T;H_0^1)$, $\xi \in L^{4/(4-d)}(0,T;V)$ and 
		$$\Pi_{W_k} \zeta = \sum_{j=1}^k \zeta_j^k w_j, \quad \Pi_{Y_k} \varphi = \sum_{j=1}^k \eta_j^k y_j, \quad \Pi_{Z_k} \xi = \sum_{j=1}^k \xi_j^k z_j,$$
		with coefficients $\{ \zeta_j^k \}_{j=1}^k,\{ \eta_j^k \}_{j=1}^k,\{ \xi_j^k \}_{j=1}^k$.  Multiplying equation \eqref{Eq_ModelGalerkinT} by $\zeta_j^k$, \eqref{Eq_ModelGalerkinS} by $\eta_j^k$ and \eqref{Eq_ModelGalerkinV} by $\xi_j^k$, we sum up each equation from $j=1$ to $k$ and integrate in time over $(0,T)$, to obtain the equation system,
		\begin{subequations}
			\begin{align}
			\int_0^T \langle \partial_t \phi_T^k,\zeta \rangle_{H^1} \textup{d} t &= \int_0^T (\phi_T^k v^k, \nabla \Pi_{W_k} \zeta) - (m_T^k \nabla \mu^k, \nabla \Pi_{W_k} \zeta) \label{Eq_DerivativeT}   + (S_T^k,\Pi_{W_k} \zeta) \, \textup{d} t, 
			\\
			\int_0^T\langle \partial_t \phi_\sigma^k, \varphi \rangle_{H_0^1} \textup{d} t &=\int_0^T (\phi_\sigma^k v^k,\nabla \Pi_{Y_k} \varphi) - \delta_\sigma^{-1}(m_\sigma^k \nabla \phi_\sigma^k,\nabla \Pi_{Y_k} \varphi)  \label{Eq_DerivativeS}  + \chi_0 (m_\sigma^k \nabla \phi_T^k, \nabla \Pi_{Y_k} \varphi) \\[-0.4cm] &\phantom{=\int_0^T} \, + (S_\sigma^k, \Pi_{Y_k} \varphi) \, \textup{d} t, \notag \\
			\int_0^T\langle \partial_t v^k,  \xi \rangle_V  \textup{d} t&= \int_0^T - \alpha (v^k,\Pi_{Z_k} \xi)  - (\nu^k \textup{D}  v^k, \textup{D}  \Pi_{Z_k} \xi)   \label{Eq_DerivativeV}  -  F_1 (|v^k| v^k, \Pi_{Z_k} \xi) \\[-0.4cm] &\phantom{=\int_0^T} \, -F_2 (|v^k|^2 v^k, \Pi_{Z_k} \xi) + (S_v^k,\Pi_{Z_k} \xi) \, \textup{d} t. \notag
			\end{align}
			\label{Eq_Derivative}
		\end{subequations}
		Each equation in \eqref{Eq_Derivative} can be estimated using the typical inequalities and the boundedness of the orthogonal projection. From (\ref{Eq_DerivativeT}), we find
		\begin{equation} \begin{aligned} \langle \partial_t \phi_T^k ,\zeta\rangle_{L^2(H^1)' \times L^2 H^1}  &\leq  \|\nabla \phi_T^k\|_{L^\infty L^2} \|v\|_{L^2 L^4} \|\Pi_{W_k} \zeta \|_{L^2 L^4} + m_\infty \|\nabla \mu^k\|_{L^2L^2} \|\nabla \Pi_{W_k} \zeta\|_{L^2L^2} \\ &\, \quad + \|S_T^k\|_{L^2L^2} \|\Pi_{W_k} \zeta \|_{L^2L^2}
		\\ &\leq C \|\zeta\|_{L^2 H^1}, 
		\end{aligned} \label{Eq_DerivativeTBound}
		\end{equation}
		and from (\ref{Eq_DerivativeS}), we get
		\begin{equation} \begin{aligned}
		\langle \partial_t \phi_\sigma^k ,\varphi\rangle_{L^{4/d} H^{-1} \times L^{4/(4-d)} H^1_0}  &\leq \|\phi_\sigma^k\|_{L^2 L^4} \|v\|_{L^4 L^4} \|\nabla \Pi_{Y_k} \varphi \|_{L^4 L^2} + \|S_\sigma^k\|_{L^2L^2} \|\Pi_{Y_k} \varphi\|_{L^2L^2}\\ &\quad \, + m_\infty \left( \delta_\sigma^{-1} \|\nabla \phi_\sigma^k\|_{L^2 L^2}+\chi_0 \|\nabla \phi_T^k\|_{L^2L^2} \right) \|\nabla \Pi_{Y_k}\varphi\|_{L^2L^2}  \\ &\leq C \|\varphi \|_{L^{4/(4-d)} H^1_0},
		\end{aligned}\label{Eq_DerivativeSBound}
		\end{equation}
		and (\ref{Eq_DerivativeV}) results in
		\begin{equation} \begin{aligned}
		\langle \partial_t v^k, \xi \rangle_{L^{4/3} V' \times L^4 V} &\leq C\big( \alpha \|v^k\|_{L^2L^2} \!+\! m_\infty \|\textup{D}  v^k\|_{L^2 L^2} \!+\! F_1 \||v^k| v^k\|_{L^2L^2}  \!+\! F_2 \||v^k|^2 v^k\|_{L^{4/3} L^{4/3}}   \\ &\quad\,+ \|\mu^k\|_{L^2L^4} \|\nabla \phi_T^k\|_{L^\infty L^2}  + \chi_0 \|\phi_\sigma^k\|_{L^2L^4} \|\nabla \phi_T^k\|_{L^\infty L^2} \big) \|\xi \|_{L^4 V}
		\\ &= C\big( \alpha \|v^k\|_{L^2L^2} + m_\infty \|\textup{D}  v^k\|_{L^2 L^2} + F_1 \|v^k\|_{L^4L^4}^2   + F_2 \|v^k\|_{L^4L^4}^3 \\ &\quad\,+  \|\mu^k\|_{L^2L^4} \|\nabla \phi_T^k\|_{L^\infty L^2}   + \chi_0 \|\phi_\sigma^k\|_{L^2L^4} \|\nabla \phi_T^k\|_{L^\infty L^2} \big) \|\xi \|_{L^4 V}
		\\ &\leq C\|\xi\|_{L^4 V}.
		\end{aligned}\label{Eq_DerivativeVBound}\end{equation}
		We note that we could have estimated the convective term in the nutrient equation without using the fact that $\{v^k\}_{k\in \mathbb{N}}$ is bounded in $L^4(0,T;[L^4]^d)$. This is particularly interesting for the case $F_2=0$ where that regularity is missing; see Remark \ref{Remark}. As a substitute, we apply the Gagliardo--Nirenberg inequality\cite{roubicek}
		$$|f|_{L^3} \leq C |\nabla f|_{L^2}^{1/2} |f|_{L^2}^{1/2} \quad \text{for all } f\in H_0^1,$$
		on $v^k$ and use the Sobolev embedding $H^1 \hookrightarrow L^6$ for $d\leq 3$ to get
		$$\begin{aligned} \int_0^T (\phi_\sigma^k v^k, \nabla \Pi_{Y_k} \xi)\textup{d} t &\leq \int_0^T |\nabla \phi_\sigma^k|_{L^2} |v^k|_{L^3} |\xi|_{L^6} \, \textup{d} t \\ &\leq \|\phi_\sigma^k\|_{L^2 H^1}^2 \|v^k\|_{L^2 V}^{1/2} \|v^k\|_{L^\infty H}^{1/2} \|\xi\|_{L^4 H^1}.   \end{aligned}$$
		From the inequalities (\ref{Eq_DerivativeTBound})--(\ref{Eq_DerivativeVBound}) and the bounds derived earlier, see (\ref{Eq_Bounds1}), we conclude that
		\begin{alignat*}{2}
		&\{\phi_T^k\}_{k \in \mathbb{N}} &&\text{ is bounded in } L^\infty(0,T;H^1) \cap H^1(0,T;(H^1)'), \\[-0.15cm]
		&\{\phi_\sigma^k\}_{k \in \mathbb{N}} &&\text{ is bounded in } L^\infty(0,T;L^2) \cap \big(1+L^2(0,T;H^1_0)\big) \cap W^{1,\tfrac{4}{d}}(0,T;H^{-1}),  \\[-0.15cm]
		&\{v^k_{\phantom{k}}\}_{k \in \mathbb{N}}  &&\text{ is bounded in } L^\infty(0,T;H) \! \cap \! L^4(0,T;[L^4]^d) \! \cap \! L^2(0,T;V) \! \cap \! W^{1,\tfrac{4}{d}}(0,T;V').
		\end{alignat*}
		Making use of the Aubin--Lions compactness lemma (\ref{Eq_AubinLions}),
		giving compact embeddings to achieve the strong convergences, we have
		\begin{equation} \label{Eq_StrongConv}
		\begin{alignedat}{2}
		\phi_T^k &\longrightarrow \phi_T &&\text{ strongly in } C([0,T];L^2), \\
		\phi_\sigma^k &\longrightarrow \phi_\sigma &&\text{ strongly in } L^2(0,T;L^2)\cap C([0,T];H^{-1}), \\
		v^k &\longrightarrow v &&\text{ strongly in } L^2(0,T;H) \cap C([0,T];V'),
		\end{alignedat}
		\end{equation}
		as $k \to \infty$.
		The strong convergence $\phi_T^k \to \phi_T$ in $C([0,T];L^2)$ implies $\phi_T(0)=\phi_{T,0}$ in $L^2$ and similarly $\phi_\sigma(0)=\phi_{\sigma,0}$ in $H^{-1}$ and $v(0)=v_0$ in $V'$. Therefore, the limit functions $(\phi_T,\mu,\phi_\sigma,v)$ of the Galerkin approximations already fulfill the initial data of the system \eqref{Eq_Model}. 
		
		In the case of $d=2$, we can also conclude $\phi_\sigma \in C([0,T];L^2)$ and $v \in C([0,T];H)$ due to the continuous embedding (\ref{Eq_ContEmbedding1}). Here, we used $[H^1,(H^1)']_{1/2}=L^2$ and $[V,V']_{1/2}=H$. In contrast, in the three-dimensional case, we deduce the respective weak continuities of $\phi_\sigma$ and $v$ due to the continuous embedding (\ref{Eq_ContEmbedding2}).
		
		\subsection*{Limit process} It remains to be shown that the limit functions also fulfill the variational form \eqref{Eq_ModelWeak}, as defined in Definition \ref{Def_Weak}. Multiplying the Galerkin system \eqref{Eq_ModelGalerkin} by $\eta \in C_c^\infty(0,T)$ and integrating from $0$ to $T$, gives
		\begin{alignat*}{2}
		\int_0^T \langle \partial_t \phi_T^k, w_j \rangle_{H^1} \eta(t) \, \textup{d} t &= \int_0^T (-\textcolor{black}{m_T^k \nabla \mu^k}+\textcolor{black}{\phi_T^k v^k},\nabla w_j) \eta(t)+(\textcolor{black}{S_T^k},w_j)_{L^2}\eta(t) \, \textup{d} t,  \\
		\int_0^T (\mu^k,w_j) \eta(t) \, \textup{d} t &= \int_0^T (\textcolor{black}{\Psi'(\phi^k_T)}-\chi_0 \phi^k_\sigma,w_j)_{L^2} \eta(t) + \varepsilon_T^2 (\nabla \phi^k_T,\nabla w_j)_{L^2} \eta(t) \, \textup{d} t,  \\
		\int_0^T \langle \partial_t \phi^k_\sigma, y_j \rangle_{H_0^1} \eta(t) \, \textup{d} t &= \int_0^T (-\textcolor{black}{m_\sigma^k(\delta_\sigma^{-1} \nabla \phi_\sigma^k -\chi_0 \nabla \phi^k_T)} + \textcolor{black}{\phi_\sigma^k v^k},\nabla y_j)_{L^2}\eta(t)   + (\textcolor{black}{S_\sigma^k},y_j)_{L^2}\eta(t)  \, \textup{d} t , \\
		\int_0^T \langle \partial_t v^k,z_j \rangle_{V} \eta(t) \, \textup{d} t &= \int_0^T (\nu^k \textup{D}  v^k, \textup{D}  z_j) \eta(t)+ (\textcolor{black}{F_1 |v^k| v^k}+\textcolor{black}{F_2 |v_k|^2v_k},z_j) \eta(t)  + (\alpha v^k+\textcolor{black}{S_v^k},z_j) \eta(t) \, \textup{d} t,
		\end{alignat*}
		for each $j\in \{1,\dots,k\}$. We take the limit $k \to \infty$ in each equation. The linear terms can be treated directly in the limit process since they can be justified via the weak convergences (\ref{Eq_WeakConv}), e.g. the functional
		$$\mu^k \mapsto \int_0^T (\mu^k,w_j) \eta(t) \, \textup{d} t \leq \|\mu^k\|_{L^2L^2} |w_j|_{L^2} |\eta|_{L^2} $$
		is linear and continuous on $L^2(0,T;L^2)$ and, hence, as $k\to \infty$
		$$\int_0^T (\mu^k,w_j) \eta(t) \, \textup{d} t  \longrightarrow \int_0^T (\mu,w_j) \eta(t) \, \textup{d} t.$$
		Thus, it remains to examine the \textcolor{black}{nonlinear} terms. We do so in the steps (i)--(vi) as follows. \\[0.1cm]
		(i) We have derived the convergences, see (\ref{Eq_StrongConv}), $$\begin{aligned}\phi_T^k \longrightarrow \phi_T &\text{ in } L^2(0,T;L^2) \cong L^2((0,T)\times \Omega) \\ \phi_\sigma^k \longrightarrow \phi_\sigma &\text{ in } L^2(0,T;L^2) \cong L^2((0,T)\times \Omega)\end{aligned}$$ as $k \to \infty$ and, consequently, we have by the continuity and boundedness of $m_T$, $$m_T^k=m_T\big(\phi^k_T(t,x),\phi_\sigma^k(t,x)\big) \longrightarrow m_T\big(\phi_T(t,x),\phi_\sigma(t,x)\big)=:m_T \text{ a.e. in } (0,T)\times \Omega$$ as $k \to \infty.$ 
		Applying the Lebesgue dominated convergence theorem, gives
		$$m_T^k  \nabla w_j \eta \longrightarrow m_T \nabla w_j \eta \text{ in } L^2((0,T) \times \Omega;\mathbb{R}^d)$$
		as $k \to \infty$ and, together with $\nabla \mu^k \rightharpoonup \nabla \mu$ weakly in $L^2((0,T)\times \Omega;\mathbb{R}^d)$ as $k\to\infty$, we have
		$$m_T(\phi^k_T) \eta \nabla w_j \cdot \nabla \mu^k \longrightarrow m_T(\phi_T) \eta \nabla w_j \cdot \nabla \mu \text{ in } L^1((0,T)\times \Omega)$$
		as $k \to \infty$. We use here the fact that the product of a strongly and a weakly converging sequence in $L^2$ converges strongly in $L^1$. The same procedure can be used with terms which involve the functions $m_\sigma^k$ and $\nu^k$. \\[0.1cm]
		(ii) By (\ref{Eq_StrongConv}), we have $\phi_T^k \to \phi_T$ in $L^2((0,T)\times \Omega)$ and $v^k \to v$ in $L^2((0,T)\times \Omega;\mathbb{R}^d)$ as $k\to \infty$, hence, as $k \to \infty$, $$\phi_T^k v^k  \cdot \nabla w_j \eta \longrightarrow \phi_T v \cdot \nabla w_j \eta \text{ in } L^1((0,T)\times \Omega).$$ 
		(iii) By the continuity and the growth assumptions on $\Psi'$, we have
		\begin{gather*} \Psi'\big(\phi^k_T(t,x)\big) \longrightarrow \Psi'\big(\phi_T(t,x)\big) \text{ a.e. in } (0,T)\times \Omega \text{ as } k\to \infty, \\
		|\Psi'(\phi^k_T) \eta w_j| \leq C(1+|\phi^k_T|)|\eta w_j|,
		\end{gather*}
		and the Lebesgue dominated convergence theorem yields as $k \to \infty$,
		$$\Psi'(\phi^k_T) \eta w_j \longrightarrow \Psi'(\phi_T) \eta w_j \text{ in } L^1((0,T)\times \Omega).$$
		(iv) Using the triangle inequality, we conclude that
		$$\begin{aligned}\big| |v^k| v^k - |v| v \big| &= \big| |v^k| v^k - |v^k| v + |v^k| v - |v| v \big| \\ &\leq |v^k| \cdot |v^k -v| + |v| \cdot \big| |v^k|-|v| \big| \\ &\leq |v^k-v| \big(|v^k|+|v|\big).
		\end{aligned}$$
		and, thus, taking the limit $k \to \infty$, results in
		$$|v^k| v^k\cdot z_j \eta \longrightarrow |v| v \cdot z_j \eta \text{ in } L^1((0,T)\times \Omega).$$
		(v) Similarly to (iv), we apply the triangle inequality to deduce that
		$$\begin{aligned}
		\big| |v_k|^2 v_k - |v|^2 v \big| &= \big| |v^k|^2 v_k - |v^k|^2 v + |v^k|^2 v - |v|^2 v^k + |v|^2 v^k - |v|^2 v \big|
		\\ &\leq |v^k - v | \big(|v_k|^2 + |v|^2 + |v^k| \cdot |v| \big)
		\end{aligned}$$
		and, again, taking the limit as $k \to \infty$ gives
		$$|v^k|^2 v^k\cdot z_j \eta \longrightarrow |v|^2 v \cdot z_j \eta \text{ in } L^1((0,T)\times \Omega).$$ 
		(vi) We have the strong convergence of $\mu^k$ and $\phi_\sigma^k$ in $L^2((0,T)\times \Omega)$. Together with the weak convergence of $\nabla \phi_T^k$ in $L^2((0,T)\times \Omega;\mathbb{R}^d)$ it is enough to conclude the convergence of the term involving $S_v^k=(\mu+\chi_0 \phi_\sigma^k)\nabla \phi_T^k$. 
		
		Using the density of $\cup_{k \in \mathbb{N}} W_k$ in $H^1$, $\cup_{k \in \mathbb{N}} Y_k$ in $H^1_0$, $\cup_{k \in \mathbb{N}} Z_k$ in $V$ and the fundamental lemma of the calculus of variations, we obtain a solution $(\phi_T,\mu,\phi_\sigma,v)$ to our model \eqref{Eq_Model} in the weak sense as defined in Definition \ref{Def_Weak}. 
	\end{proof}
	
	\begin{remark} We can associate a pressure function to the velocity so that we have a quintuple $(\phi_T,\mu,\phi_\sigma,v,p)$, which solves \eqref{Eq_Model} in distributional sense. See also Refs.~\citen{temam2001navier,simon1999existence} for a similar argument. Let 
		$$ w = -\partial_t v + \div (\nu \textup{D}  v) - \alpha v - F_1 |v| v - F_2 |v|^2 v + S_v.$$
		Then $w \in W^{-1,\infty}(0,T;H^{-1})=\mathscr{L}(W_0^{1,1}(0,T);H^{-1})$ and $\langle w(\psi), \xi \rangle_V = 0 $ for all $\psi \in W_0^{1,1}(0,T), \xi \in V$. Thus, by the corollary of the de Rham lemma, see Corollary \ref{Cor_Rham}, there is a unique $p \in W^{-1,\infty}(0,T;L_0^2)$ such that $\nabla p = w$. 
	\end{remark}
	
	\begin{remark}  If $F_2=0$, then we only have $u \in L^3(0,T;[L^3]^d)$ instead of $L^4(0,T;[L^4]^d)$. We have to control $$\int_0^T (|v| v, \zeta) \,\textup{d}t \leq \int_0^t |v|^2 |\zeta| \,\textup{d}t \leq |v^2|_{L^{3/2}L^{3/2}} |\zeta|_{L^3 L^3} \leq C |v|_{L^3L^3}^2 |\zeta|_{L^3 H^1}.$$
		Hence, we are even able to bound $\partial_t v$ in $L^{3/2}(0,T;V')$ instead of only $L^{4/3}(0,T;V')$. That means the velocity itself has less regularity, but its derivative's regularity is larger. Still, via the Aubin--Lions lemma, we are able to extract a subsequence of the Galerkin approximation $\{v^k\}_{k\in\mathbb{N}},$ which converges strongly to a function $v$ in $L^2(0,T;H)$.
		\label{Remark}
	\end{remark}
	
	\section{Nonlocal Effects} \label{Section_Nonlocal}
	In this section, we consider nonlocal  effects in a tumor growth model with a convective velocity, which obeys the unsteady Darcy--Forchheimer--Brinkman law. In biological models nonlocal  terms are used to describe competition for space\cite{szymanska2009mathematical}, cell-to-cell adhesion and cell-to-matrix adhesion\cite{chaplain2011mathematical,gerisch2008mathematical}, and the inclusion of such nonlocal effects in mesoscale models of tumor growth leads to systems of nonlinear integro-differential equations. 
	
	Models, that account for cell--matrix adhesion effects, involve matrix degrading enzymes, which erode the extracellular matrix and therefore, allow the migration of cells into tissue. Such systems have been analyzed in-depth in Refs.~\citen{stinner2014global,engwer2017structured,chaplain2011mathematical}. There, the tumor volume fraction is modelled by a reaction-diffusion equation, in contrast to the fourth order Cahn--Hilliard phase field equation in our setting.
	
	Following Refs.~\citen{chaplain2011mathematical,frigeri2017diffuse}, we consider cell--cell adhesion effects, which are responsible for the binding of one or more cells to each other through the reaction of proteins on the cell surfaces. It is reasonable to take cell-to-cell adhesion into account since the Ginzburg--Landau free energy functional (\ref{Eq_DerivationE}) leads to separation and surface tension effects\cite{frigeri2017diffuse}, which implies that the tumor cells prefer to adhere to each other rather than to the healthy cells. Moreover, cell-to-cell adhesion is a key factor in tissue formation, stability, and the breakdown of tissue.
	
	The well-known local Cahn--Hilliard equation has an phenomenological background\cite{cahn1958free} and in the search for a physical derivation Giacomin and Lebowitz studied the problem of phase separation from a microscopic background using the methods of statistical mechanics, see Refs.~\citen{giacomin1996exact,giacomin1997phase}. They obtained a nonlocal version of the Cahn--Hilliard equation with the underlying free energy functional
	\begin{equation} \label{helm}
	\int_\Omega \Psi(\phi_T(x))\, \text{d}x + \frac14 \int_\Omega \int_\Omega J(x-y) \big(\phi_T(x)-\phi_T(y)\big)^2 \, \text{d}y \text{d}x,
	\end{equation}
	which is also called the nonlocal Helmholtz free energy functional\cite{chen2001derivation,gal2017nonlocal}. Here, $J:\mathbb{R}^d \to \mathbb{R}$ is assumed to be a convolution kernel such that $J(x)=J(-x)$. One can obtain the classical Ginzburg--Landau free energy functional from (\ref{helm}) by choosing the kernel function $J(x,y)=k^{d+2} \chi_{[0,1]}(|k(x-y)|^2)$ and letting $k \to \infty$, see Refs.~\citen{frigeri2015a,guan2014convergent}, and therefore, the well-known Cahn--Hilliard model can be interpreted as an approximation of its nonlocal version.
	
	We modify the nonlocal Helmholtz free energy functional (\ref{helm}) to account for chemotactic effects, 
	\begin{equation*} \begin{aligned}
	\mathcal{E}(\phi_T,\phi_\sigma) &=\int_\Omega \Psi(\phi_T(x))\, \text{d}x + \frac14 \int_\Omega \int_\Omega J(x-y) \big(\phi_T(x)-\phi_T(y)\big)^2 \, \text{d}y \text{d}x  - \chi_0 \int_\Omega \phi_\sigma(x) \phi_T(x) \, \text{d}x. \end{aligned}
	\end{equation*}
	The chemical potential is given by the first variation of the system's underlying free energy functional $\mathcal{E}$ and therefore, we consider a class of long-range interactions in which nonlocal effects are characterized by chemical potentials of the form,
	\begin{equation*} \label{mu}
	\mu = \frac{\delta \mathcal{E}}{\delta \phi_T} = \Psi'(\phi_T)+ \int_\Omega J(x-y) \big(\phi_T(x)-\phi_T(y)\big) \, \textup{d} y - \chi_0 \phi_\sigma,
	\end{equation*}
	and recalling the convolution operator, which we denote by $*$,
	we can rewrite the chemical potential as
	\begin{equation*} \mu= \Psi'(\phi_T) + \phi_T \cdot J*1 - J*\phi_T - \chi_0 \phi_\sigma. \end{equation*}
	This leads directly to a nonlocal model governed by the system,
	\begin{equation} \begin{aligned}
	\partial_t \phi_T+ \div(\phi_T v) &=  \div (m_T(\phi_T,\phi_\sigma) \nabla \mu) + S_T(\phi_T,\phi_\sigma), \\
	\mu &= \Psi'(\phi_T) + \phi_T \cdot J*1 - J*\phi_T - \chi_0 \phi_\sigma, \\
	\partial_t \phi_\sigma + \div(\phi_\sigma v)&=  \div\!\big(m_\sigma(\phi_T,\phi_\sigma) (\delta_\sigma^{-1} \nabla \phi_\sigma - \chi_0   \nabla \phi_T)\big) + S_\sigma(\phi_T,\phi_\sigma), \\
	\partial_t v + \alpha v &= \div (\nu(\phi_T,\phi_\sigma) \textup{D} v)- F_1 |v| v - F_2 |v|^2 v - \nabla p  +  S_v(\phi_T,\phi_\sigma),\\
	\div v &= 0,
	\end{aligned} 
	\label{Eq_ModelNonLocal} \end{equation}
	with the initial-boundary data as before, see (\ref{Eq_InitialBoundary}).
	
	The nonlocal  Cahn--Hilliard equation has been analyzed in Refs.~\citen{bates2005neumann,bates2005dirichlet,bates2006some,gajewski2003nonlocal,gal2017nonlocal}. In Ref.~\citen{colli2012global,frigeri2012nonlocal,frigeri2013strong,frigeri2015a} it has been coupled to the Navier--Stokes equation, in Ref.~\citen{della2018nonlocal} to the Darcy equation, in Ref.~\citen{della2016nonlocal} to the Brinkman equation and in Ref.~\citen{frigeri2017diffuse} to a reaction-diffusion equation. We briefly discuss modifications in the energy estimates of the nonlocal  model (\ref{Eq_ModelNonLocal}) in contrast to the estimates derived in the proof of Theorem 1 for the local model. Due to the integro-differential structure, some inequalities have to be analyzed again, but overall the existence of a weak solution remains valid for the nonlocal problem. Notice that we cannot expect $\phi_T \in L^\infty(0,T;H^1)$ since no Laplacian appears in the potential equation. Therefore, the term $\chi_0 \div (m_\sigma \nabla \phi_T)$ has to be treated again in the estimates with an additional assumption on the chemotaxis constant $\chi_0$.
	
	\begin{definition}[Weak solution] We call a quadruple $(\phi_T,\mu,\phi_\sigma,v)$ a weak solution of the system \eqref{Eq_ModelNonLocal} if the functions $\phi_T,\mu,\phi_\sigma : (0,T) \times \Omega \to \mathbb{R}, v:(0,T) \times \Omega \to \mathbb{R}^d$ have the regularity
		\begin{alignat*}{2}
		&\phi_T &&\in W^{1,\tfrac{4}{d}}(0,T;(H^1)')\cap L^\infty(0,T;L^2)  \cap L^2(0,T;H^1), \\
		&\mu &&\in L^2(0,T;H^1), \\[-0.15cm]
		&\phi_\sigma &&\in W^{1,\tfrac{4}{d}}(0,T;H^{-1}) \cap \big(1+L^2 (0,T;H^1_0)\big) , \\[-0.15cm]
		&v &&\in W^{1,\tfrac{4}{d}}(0,T;V') \cap L^4(0,T;[L^4]^d) \cap L^2 (0,T;V), \end{alignat*}
		fulfill the initial data $\phi_T(0)=\phi_{T,0}$, $\phi_{\sigma}(0)=\phi_{\sigma,0}$, $v(0)=v_0$ and satisfy the following variational form of \eqref{Eq_Model}
		\begin{subequations}
			\label{Eq_ModelNonLocalWeak}
			\begin{align}
			\langle \partial_t \phi_T, \varphi_1 \rangle_{H^1}&= (\phi_T v, \nabla \varphi_1) - (m_T \nabla \mu, \nabla \varphi_1) + (S_T,\varphi_1),\label{Eq_ModelNonLocalWeakTum} \\
			(\mu,\varphi_2) &= (\Psi'(\phi_T),\varphi_2) + (\phi_T \cdot J*1,\varphi_2)-(J*\phi_T,\varphi_2) - \chi_0 (\phi_\sigma,\varphi_2) \label{Eq_ModelNonLocalWeakMu}\\
			\langle \partial_t \phi_\sigma, \varphi_3 \rangle_{H_0^1} &= (\phi_\sigma v,\nabla \varphi_3) - \delta_\sigma^{-1} (m_\sigma \nabla \phi_\sigma,\nabla \varphi_3) + \chi_0 (m_\sigma \nabla \phi_T, \nabla \varphi_3) \label{Eq_ModelNonLocalWeakNut} + (S_\sigma, \varphi_3),  \\
			\langle \partial_t v, \varphi_4 \rangle_V &= - \alpha( v, \varphi_4) - (\nu \textup{D} v, \textup{D} \varphi_4) - F_1( |v| v, \varphi_4) - F_2 ( |v|^2 v, \varphi_4)  \label{Eq_ModelNonLocalWeakVel}  + (S_v,\varphi_4),
			\end{align}
		\end{subequations}
		for all $\varphi_1,\varphi_2 \in H^1, \varphi_3 \in H_0^1, \varphi_4 \in V$. 
		\label{Def_WeakNonLocal}
	\end{definition}

	\begin{theorem}[Existence of a global weak solution]
		\label{Thm_ExistenceNonLocal} Let \textup{(A1)--(A4)} hold and additionally: \\
		\textup{(A6)} $\Psi \in C^2(\mathbb{R})$ is such that $|\Psi(x)| \geq C_1 |x|^2 - C_2$, $\Psi''(x) \geq C_3
		-(J*1)(x)$,  and $|\Psi'(x)|\leq C_4(|x|+1)$ for $C_2,C_4>0$, $C_1 > \frac12 |J|_{L^1}-\tfrac12 (J*1)(x)$ for a.e. $x \in \Omega$ and 
		$$C_3 > \sqrt{\frac{2\chi_0 m_\infty}{m_0}} \cdot \frac{4\chi_0^2m_\infty^2 \delta_\sigma+2m_0 \chi_0 m_\infty \delta_\sigma + \chi_0^2 \delta_\sigma m_0^2}{2 m_0^2}.$$
		\textup{(A7)} $J \in W^{1,1}(\mathbb{R}^d)$ is even and $(J*1)(x) \geq 0$ for a.e. $x\in \Omega$. \\
		Then there exists a solution quadruple $(\phi_T, \mu, \phi_\sigma, v)$ to \eqref{Eq_ModelNonLocal} in the sense of Definition \ref{Def_WeakNonLocal}. 
		Moreover, the solution quadruple has the regularity
		\begin{alignat*}{2}
		&\phi_T &&\in \begin{cases} C([0,T];L^2), &d=2, \\ C([0,T];(H^1)')  \cap C_w([0,T];L^2), &d=3,\end{cases} \\
		&\phi_\sigma &&\in \begin{cases} C([0,T];L^2), &d=2, \\ C([0,T];H^{-1})  \cap C_w([0,T];L^2), &d=3,\end{cases} \\ &v &&\in \begin{cases} C([0,T];H), &d=2, \\ C([0,T];V') \cap C_w([0,T];H), &d=3. \end{cases} 
		\end{alignat*} Additionally, there is a unique $p \in W^{-1,\infty}(0,T;L_0^2)$ such that $(\phi_T,\mu,\phi_\sigma,v,p)$ is a solution quintuple to \eqref{Eq_ModelNonLocal} in the distributional sense.
	\end{theorem}
	\begin{proof}
		Since the approach is similar to the proof of Theorem \ref{Thm_Existence}, we will directly derive an energy estimate in the continuous setting.
		We take, as the test function in (\ref{Eq_ModelNonLocalWeakMu}), $\varphi_2=-\partial_t \phi_T$, which gives
		\begin{align*}-(\mu,\partial_t \phi_T) &=-\frac{\textup{d} t }{\textup{d} t} \! \left( |\Psi(\phi_T)|_{L^1} \!+ \! \frac12 |(J*1)^{1/2} \phi_T|_{L^2}^2 \!-\! \frac12 (\phi_T,J*\phi_T) \right) \! + \chi_0 (\phi_\sigma, \partial_t \phi_T)
		\\ &= -\frac{\textup{d} }{\textup{d} t} \! \left( |\Psi(\phi_T)|_{L^1}+ \frac14 \int_\Omega \int_\Omega J(x-y) \big(\phi_T(x)-\phi_T(y)\big)^2 \text{d$x$d$y$} \right)  + \chi_0 (\phi_\sigma, \partial_t \phi_T),
		\end{align*}
		where we used the fact that
		\begin{align*} &\frac{\textup{d} }{\textup{d} t} \int_\Omega \int_\Omega J(x-y) \big(\phi_T(x)-\phi_T(y)\big)^2\text{d$y$d$x$} 
		\\ &=2 \int_\Omega\int_\Omega J(x-y) [\phi_T(x)-\phi_T(y)]\cdot[\partial_t \phi_T(x)-\partial_t \phi_T(y)] \, \text{d$y$d$x$} 
		\\ &=4 \int_\Omega\int_\Omega J(x-y) [\phi_T(x)-\phi_T(y)] \partial_t \phi_T(x) \text{d$y$d$x$},
		\end{align*}
		since $J$ is assumed to be even, see (A7) in Theorem \ref{Thm_ExistenceNonLocal}. As in the local model, we use the test functions $\varphi_1=\mu+\chi_0\phi_\sigma$, $\varphi_3 = K(\phi_\sigma-1)$ and $\varphi_4=v$ in (\ref{Eq_ModelNonLocalWeak}), which gives, after adding the equations,
		\begin{equation} 
		\begin{aligned}
		&\frac{\textup{d} }{\textup{d} t} \left[ |\Psi(\phi_T)|_{L^1} + \frac{1}{4} \int_\Omega \int_\Omega J(x-y) \big(\phi_T(x)-\phi_T(y)\big)^2  \text{d$x$d$y$} + \frac{K}{2} |\phi_\sigma-1|_{L^2}^2 + \frac12 |v|_{L^2}^2 \right] 
		\\&\quad\, + (m_T,|\nabla \mu|^2) + K \delta_\sigma^{-1} (m_\sigma,|\nabla \phi_\sigma|^2) + (\nu, |\textup{D} v|^2)+ F_1 |v|_{L^3}^3 + F_2 |v|_{L^4}^4 +\alpha |v|_{L^2}^2 
		\\ &= -\chi_0 (m_T \nabla \mu, \nabla \phi_\sigma)+ (S_T,\mu+\chi_0 \phi_\sigma) + \chi_0 K (m_\sigma \nabla \phi_T,\nabla \phi_\sigma) +  K (S_\sigma, \phi_\sigma),
		\end{aligned}
		\label{Eq_EnergyNonLocal}
		\end{equation}
		where we again used that the tested convective terms cancel with $S_v$. We estimate the terms involving the source functions $S_T, S_\sigma$ as before in the local case, see (\ref{Eq_EnergyST}) and (\ref{Eq_EnergySs}). Also, as before in the local model, we additionally test with $\varphi_2=1$ in (\ref{Eq_ModelNonLocalWeakMu}) to deduce the following estimate on $\mu$:
		$$|\mu|_{L^1} \leq |\Psi'(\phi_T)|_{L^1}  + \chi_0 |\phi_\sigma|_{L^1} \leq C\left(1 + |\phi_T|_{L^2} + |\phi_\sigma|_{L^2} \right).$$
		
		Here, we used the fact that $(\phi_T,J*1)=(J*\phi_T,1)$, since $J$ is assumed to be even, see (A7) in Theorem \ref{Thm_ExistenceNonLocal}. Therefore, we can estimate the right hand side in (\ref{Eq_EnergyNonLocal}) as
		$$\begin{aligned}(\text{RHS}) &\leq \frac{m_0}{2} |\nabla \mu|_{L^2}^2 + \frac{\chi_0^2 m_\infty^2}{m_0} |\nabla \phi_\sigma|_{L^2}^2 + C\left(1+|\phi_\sigma|_{L^2}^2+|\phi_T|_{L^2}^2\right) +\frac{\chi_0 K^2 m_\infty}{2} |\nabla \phi_T|_{L^2}^2 + \frac{\chi_0  m_\infty}{2} |\nabla \phi_\sigma|_{L^2}^2.\end{aligned} $$
		Notice that at this point we have no information on $\phi_T$ on the left hand side of (\ref{Eq_EnergyNonLocal}), which is crucially needed to absorb the terms from the right hand side. We can do the following calculation due to the growth estimate on $\Psi$, see (A6),
		\begin{equation}
		\begin{aligned}
		&|\Psi(\phi_T)|_{L^1}+\frac{1}{4}\int_\Omega \int_\Omega J(x-y) \big(\phi_T(x)-\phi_T(y)\big)^2 \text{d$x$d$y$}
		\\ &= |\Psi(\phi_T)|_{L^1}+\frac12 |(J*1)^{1/2} \phi_T|_{L^2}^2 - \frac12 (\phi_T,J*\phi_T)
		\\ &\geq \int_\Omega \left[ C_1 + \frac12 (J*1)(x)  - \frac12 |J|_{L^1} \right] |\phi_T(x)|^2 \textup{d} x -C_2 |\Omega| 
		\\ &\geq C( |\phi_T|_{L^2}^2 - 1),
		\end{aligned}
		\label{Eq_NonLocalEstimate1} \end{equation}
		where we used the Young convolution inequality\cite{lieb2001analysis} to get
		$$(\phi_T,J*\phi_T)\leq |\phi_T|_{L^2} |J*\phi_T|_{L^2} \leq |\phi_T|_{L^2}^2 |J|_{L^1}.$$
		This estimate will give us $\phi_T \in L^\infty(0,T;L^2)$. Indeed, integrating \eqref{Eq_EnergyNonLocal} with respect to time on the interval $(0,s)$, $0<s<T$, and introducing the result into (\ref{Eq_NonLocalEstimate1}) and employing the estimates of the source terms gives
		\begin{equation}
		\begin{aligned}
		&C |\phi_T(s)|_{L^2}^2 + \frac{K}{2}|\phi_\sigma(s)|_{L^2}^2 + \frac12 |v(s)|_{L^2}^2 + \frac{m_0}{2} \|\nabla \mu\|_{L^2L^2}^2+ m_0 \|\textup{D} v\|_{L^2L^2}^2+ F_2 \|v\|_{L^4L^4}^4 \\&\quad\,+ \left( \frac{K m_0}{\delta_\sigma} - \frac{\chi_0^2 m_\infty^2}{m_0} - \frac{\chi_0 m_\infty}{2} \right) \|\nabla \phi_\sigma\|_{L^2L^2}^2 -\frac{\chi_0 K^2 m_\infty}{2} \|\nabla \phi_T\|_{L^2L^2}^2
		\\ &\leq C \left(1+ \text{IC}+\|\phi_T\|_{L^2L^2}^2+\|\phi_\sigma\|_{L^2L^2}^2 \right),
		\end{aligned}
		\label{Eq_NonLocalEnergy2}
		\end{equation}
		where
		$$\text{IC}=|\Psi(\phi_{T,0})|_{L^1} + \frac{1}{4} \int_\Omega \int_\Omega J(x-y) \big(\phi_{T,0}(x)-\phi_{T,0}(y)\big)^2  \text{d$x$d$y$} + \frac{K}{2} |\phi_{\sigma,0}-1|_{L^2}^2 + \frac12 |v_0|_{L^2}^2,$$
		and, due to assumption (A6), we have $|\Psi(\phi_{T,0})|_{L^1} \leq C(|\phi_{T,0}|_{L^2}^2+1)$. Note that we have a negative term on the left hand side in the inequality (\ref{Eq_NonLocalEnergy2}), which we still have to overpower so that we can apply the Gronwall lemma.
		We have, on one hand,
		$$ (\nabla \mu,\nabla \phi_T) \leq \frac{1}{C_3} |\nabla \mu|_{L^2}^2 + \frac{C_3}{4} |\nabla \phi_T|_{L^2}^2,$$
		and, on the other hand, using the growth estimate on $\Psi''$, see (A6), and again the Young convolution inequality, we deduce that
		\begin{align*}
		&(\nabla \mu,\nabla \phi_T) \\
		&=(\Psi''(\phi_T),|\nabla \phi_T|^2) + (|\nabla \phi_T|^2,J*1)+(\nabla(J*1),\phi_T \nabla \phi_T) - (\nabla J * \phi_T,\nabla \phi_T)  - \chi_0 (\nabla \phi_\sigma,\nabla \phi_T)
		\\ &\geq C_3 |\nabla \phi_T|_{L^2}^2  -|\nabla J|_{L^1} |\phi_T|_{L^2} |\nabla \phi_T|_{L^2} - \chi_0 |\nabla \phi_\sigma|_{L^2} |\nabla \phi_T|_{L^2}
		\\ &\geq \frac{C_3}{2} |\nabla \phi_T|_{L^2}^2 - C |\phi_T|_{L^2}^2 - \frac{\chi_0^2}{C_3} |\nabla \phi_\sigma|_{L^2}^2.
		\end{align*}
		Combining these two inequalities gives the following estimate on $\nabla \phi_T$:
		$$\frac{m_0 C_3^2}{16} |\nabla \phi_T|_{L^2}^2 \leq \frac{m_0}{4} |\nabla \mu|_{L^2}^2 + C|\phi_T|_{L^2}^2  + \frac{m_0 \chi_0^2}{4} |\nabla \phi_\sigma|_{L^2}^2.$$
		Adding this inequality to (\ref{Eq_NonLocalEnergy2}) results in
		\begin{equation}
		\begin{aligned}
		&C |\phi_T(s)|_{L^2}^2 + \frac{K}{2}|\phi_\sigma(s)|_{L^2}^2 + \frac12 |v(s)|_{L^2}^2  + \frac{m_0}{2} \|\nabla \mu\|_{L^2L^2}^2 + m_0 \|\textup{D} v\|_{L^2L^2}^2 + F_2 \|v\|_{L^4L^4}^4 \\&\quad\,+  \left( \frac{K m_0}{\delta_\sigma} - \frac{\chi_0^2 m_\infty^2}{m_0} - \frac{\chi_0 m_\infty}{2} - \frac{m_0 \chi_0^2}{4} \right) \|\nabla \phi_\sigma\|_{L^2L^2}^2 \\ &\quad \,+ \left( \frac{m_0 C_3^2}{16} - \frac{\chi_0 K^2 m_\infty}{2} \right) \|\nabla \phi_T\|_{L^2L^2}^2
		\\ &\leq C \left(1+\text{IC} + \|\phi_T\|_{L^2L^2}^2+\|\phi_\sigma\|_{L^2L^2}^2\right).
		\end{aligned}
		\end{equation}
		We choose $K$ large enough such that the prefactor of $\|\nabla \phi_\sigma\|_{L^2L^2}^2$ is strictly positive. As a consequence, the prefactor of $\|\nabla \phi_T\|_{L^2L^2}^2$ is also strictly positive due to the assumption on $C_3$, see (A6). Hence, by applying the Gronwall lemma, see Lemma \ref{Lem_Gronwall}, we can deduce the final energy estimate
		\begin{align*}
		& \|\phi_T\|_{L^\infty L^2}^2 + \|\phi_\sigma\|_{L^\infty L^2}^2 + \|v\|_{L^\infty L^2}^2   +  \|\nabla \mu\|_{L^2L^2}^2 \\ &\quad \,+ \|\textup{D} v\|_{L^2L^2}^2+ \|v\|_{L^4L^4}^4+ \|\nabla \phi_\sigma\|_{L^2L^2}^2 +\|\nabla \phi_T\|_{L^2L^2}^2
		\\&\leq C(1+\text{IC}) e^{CT}.
		\end{align*}
		In the same way as before, we can bound the time derivatives of $\phi_T,\phi_\sigma,v$ and conclude existence of strongly converging sequences in the discrete case. The only difference to the local model in the limit process is the integral term, but it can be treated immediately since the functional
		$$\phi_T \mapsto \int_0^T (\phi_T \cdot J*1 - J * \phi_T,w_j) \eta(t) \,\textup{d} t \leq 2  |\eta|_{\infty} |w_j|_{L^2} \|\phi_T\|_{L^2L^2} |J|_{L^1}$$
		is linear and continuous on $L^2(0,T;L^2)$.
	\end{proof}
	
	The first steps in deriving the energy estimate were the same as in the local case. The lack of regularity on $\phi_T$, which is caused by the new chemical potential, required us to derive additional estimates as a replacement. We achieved the regularity $\phi_T \in L^\infty(0,T;L^2) \cap L^2(0,T;H^1)$ instead of $L^\infty(0,T;H^1)$ as in the local case. Nonetheless, this regularity is enough to prove the existence of weak solutions. 
	
	The new assumptions in Theorem \ref{Thm_ExistenceNonLocal} were crucial for the proof to be completed. We assumed a lower bound assumption on $\Psi''$, directly involving the new constant $C_3$, which has to be sufficiently large. In the case of a sufficiently small chemotaxis constant $\chi_0$, the assumption on $C_3$ is fulfilled. Moreover, we introduce assumptions on properties for the new function $J$ in the chemical potential, which are fulfilled by a typical kernel function.
	
	\section{Sensitivity Analysis} \label{Section_Sens}
	The relative effects of model parameters in determining key quantities of interest (QoIs), such as the evolution of tumor mass over time, are very important in the development of predictive models of tumor growth. Accordingly, in this section we address the question of sensitivity of solutions of our system (\ref{Eq_Model}) to variations in model parameters, and we provide a sensitivity analyses using both statistical methods\cite{oden_2018,saltelli2000sensitivity,saltelli2008global,saltelli2010variance} and data-dependent methods based on the notion of active subspaces\cite{constantine2014active,constantine2015active}. We first introduce each method and then compare the sensitivities of the parameters 
	$$\theta = \left( \varepsilon_T, \chi_0, \delta_\sigma, \lambda_T, \lambda_\sigma, \lambda_A, M_T, M_\sigma, \overline{E}, \alpha, \nu, F_1, F_2 \right)^\top \in \mathbb{R}^{13}$$
	in our model for each method. 
	
	As the quantity of interest in the sensitivity analysis of both methods, we choose the volume of the tumor mass at different times $t\in \mathcal{I}$, i.e.
	$$Q(\theta) =\left[\frac{1}{|\Omega|} \int_\Omega \phi_T(t,x) \, \textup{d} x\right]_{t \in \mathcal{I}} \!\!\!\in \mathbb{R}^{\text{dim}(\mathcal{I})},$$
	which depends on the choice of the parameter setting $p$. Further, we choose the following uniformly distributed priors,
	\begin{equation} \begin{alignedat}{5} 
	\varepsilon_T &\sim \mathcal{U}(0.01,0.10), \quad &\lambda_T &\sim \mathcal{U}(0.01,1.00), \quad &M_T &\sim \mathcal{U}(0.10,1.00), \\
	\chi_0 &\sim \mathcal{U}(0.10,1.00),\quad &\lambda_\sigma &\sim \mathcal{U}(0.01,1.00),\quad &M_\sigma &\sim \mathcal{U}(0.10,1.00), \\
	\delta_\sigma &\sim \mathcal{U}(0.01,0.10),\quad  &\lambda_A &\sim \mathcal{U}(0.00,0.05),\quad &\overline{E} &\sim \mathcal{U}(0.25,1.00), \\
	\alpha &\sim \mathcal{U}(0.10,10.0), \quad &\nu &\sim \mathcal{U}(0.10,10.0), \quad &F_1,F_2 &\sim \mathcal{U}(0.10,10.0).
	\end{alignedat}
	\label{Eq_Priors}
	\end{equation}
	
	\subsection*{Variance-based method} 
	The statistical method of sensitivity analysis employed in this work is a
	variance-based method, developed by Ref.~\citen{sobol2001global}, and described in detail by Ref.~\citen{saltelli2008global}. The variance-based method takes into account uncertainties from the
	input factors, showing how the variance of the output is dependent on these uncertainties. The main drawback of this method is the high computational cost, which is $N(k + 2)$, where $N$ is the number of samples and $k$ is the number of parameters (here $k = 13$).
	
	The algorithm implementing the variance-based method consists of initially generating two matrices, $A$ and $B$, with size $N \times k$, given as:
	\begin{align*} A &= \begin{pmatrix} \theta_1^{(A1)} & \theta_2^{(A1)} & \cdots & \theta_i^{(A1)} & \cdots &\theta_k^{(A1)} \\
	\theta_1^{(A2)} & \theta_2^{(A2)} & \cdots & \theta_i^{(A2)} & \cdots& \theta_k^{(A2)} \\
	\vdots & \vdots & \ddots & \vdots & \ddots & \vdots \\
	\theta_1^{(AN)} & \theta_2^{(AN)} & \cdots & \theta_i^{(AN)} & \cdots& \theta_k^{(AN)} 
	\end{pmatrix}, \\[0.2cm]
	B &=  \begin{pmatrix} \theta_1^{(A1)} & \theta_2^{(B1)} & \cdots & \theta_i^{(B1)} & \cdots &\theta_k^{(B1)} \\
	\theta_1^{(B2)} & \theta_2^{(B2)} & \cdots & \theta_i^{(B2)} & \cdots& \theta_k^{(B2)} \\
	\vdots & \vdots & \ddots & \vdots & \ddots & \vdots \\
	\theta_1^{(BN)} & \theta_2^{(BN)} & \cdots & \theta_i^{(BN)} & \cdots& \theta_k^{(BN)}  \end{pmatrix} .
	\end{align*}
	
	Each row represents one set of values from the vector of parameters sampled from the priors given in (\ref{Eq_Priors}). The following step is to generate $k$ matrices $C_i$, where column $i$ comes from matrix $B$ and all other $k$ columns come from matrix $A$, such as:
	
	\begin{equation*}
	C_i = \begin{pmatrix} \theta_1^{(A1)} & \theta_2^{(A1)} & \cdots & \theta_i^{(B1)} & \cdots &\theta_k^{(A1)} \\
	\theta_1^{(A2)} & \theta_2^{(A2)} & \cdots & \theta_i^{(B2)} & \cdots& \theta_k^{(A2)} \\
	\vdots & \vdots & \ddots & \vdots & \ddots & \vdots \\
	\theta_1^{(AN)} & \theta_2^{(AN)} & \cdots & \theta_i^{(BN)} & \cdots& \theta_k^{(AN)}  \end{pmatrix} .
	\end{equation*}
	
	The output for all the sample matrices, that means $A$, $B$ and $C_i$, are computed, and stored as the vectors $Y_A$, $Y_B$ and $Y_{C_i}$. Each line of the vectors $Y_J$, $J \in \{A,B,C_i\}$, of
	size $N$ represents the QoI computed for each row of the matrix $J$.
	The last step from the variance-based method is to computed the first-
	order sensitivity indices, $S_i$, and the total effect indices, $S_{T_i}$. The first-order sensitivity indices are computed as
	\begin{equation*}
	S_i = \frac{Y_A \cdot Y_C - f_0^2}{Y_A \cdot Y_B - f_0^2},
	\end{equation*}
	where
	$$f_0^2=\left(\frac{1}{N} \sum_{n=1}^N Y_A^{(n)} \right)^2.$$
	These indices are always between $0$ and $1$. High values of $S_i$ indicate a sensitive parameter, and low $S_i$, for additive models, indicates a low-sensitive parameter. For non-additive models, the total effect indices take the first-order effects and the contribution of higher-order effects into account due
	to interactions between the model parameters. These indices are given as:
	$$S_{T_i}=1-\frac{Y_B\cdot Y_{C_i} - f_0^2}{Y_A\cdot Y_A -f_0^2}.$$
	
	According to Ref.~\citen{saltelli2008global}, for $\theta_i$ to be a non-influential parameter, it is necessary and sufficient that $S_{T_i} = 0$. If the total effect index from the $i$th parameter
	is close to zero, then the parameter can be fixed to any value within the
	uncertainty range without affecting the variance of the QoI\cite{saltelli2008global}.
	
	\subsection*{Active subspace method}
	The active subspace method is used for dimensional reduction of subspaces of probability distributions of model parameters and identifies the directions of the sensitive parameters\cite{constantine2014active,constantine2015active}. Let $\rho$ be the probability density function corresponding to the distribution of the parameters $\theta$ as chosen in (\ref{Eq_Priors}) and let $f$ be the data misfit function, which is given by 
	$$f: \mathbb{R}^{13} \to \mathbb{R}, \quad p \mapsto \frac{1}{2} \big\|\Gamma^{-1/2} \big(d - Q(p)\big)\big\|^2_2,$$
	where $d=Q(\theta)+\eta$ is some data for a zero-centered Gaussian noise $\eta \sim \mathcal{N}(0,\Gamma)$ with covariance $\Gamma$. Our data\cite{lima2018calibration} consists of in vitro data observed during the evolution of tumor cells with different initial tumor confluences without any treatment in a well with radius $0.32\,\text{cm}$ over the course of $21$ days, see Figure \ref{Figure_Data}. 
	
	\begin{figure}[H]
		\centering
		\begin{tikzpicture}
\begin{axis}[
    tick label style={/pgf/number format/fixed},
    width=12cm,
    height=8cm,
    xlabel={Days},
    ylabel={Confluence of the tumor mass},
    xmin=0, xmax=21,
    ymin=0, ymax=0.138,
    ytick={0,0.02,0.04,0.06,0.08,0.1,0.12},
    legend pos=north west,
    ymajorgrids=true,
    grid style=dashed,
    legend pos=south east,
    legend cell align={left}
]

\addlegendimage{empty legend}
\addlegendentry{\hspace{-.6cm}\textbf{Initial tumor confluence}}
 
\addplot
    coordinates {
                       (0,   0.005620000000000)
   (0.700000000000000,  0.009280000000000)
   (1.960000000000000,   0.006410000000000)
   (2.640000000000000,   0.013100000000000)
   (3.780000000000000,   0.014300000000000)
   (4.950000000000000,   0.020600000000000)
   (5.630000000000000,   0.028300000000000)
   (6.630000000000000,   0.040100000000000)
   (7.290000000000000,   0.065100000000000)
   (8.690000000000000,   0.062200000000000)
   (9.640000000000001,   0.085200000000000)
  (10.699999999999999,   0.100000000000000)
  (11.699999999999999,   0.095700000000000)
  (12.680000000000000,   0.094800000000000)
  (13.900000000000000,   0.103000000000000)
  (14.680000000000000,   0.099200000000000)
  (15.590000000000000,   0.099500000000000)
  (16.850000000000001,   0.108000000000000)
  (17.750000000000000,   0.108000000000000)
  (18.780000000000001,   0.104000000000000)
  (19.550000000000001,   0.101000000000000)
  (20.660000000000000,   0.115000000000000)
    };
\addlegendentry{0.00562}

\addplot
    coordinates {
   (                0,   0.008710000000000)
   (0.700000000000000,   0.013200000000000)
   (1.960000000000000,   0.010700000000000)
   (2.640000000000000,   0.020300000000000)
   (3.780000000000000,   0.025300000000000)
   (4.950000000000000,   0.042700000000000)
   (5.630000000000000,   0.060900000000000)
   (6.630000000000000,   0.063700000000000)
   (7.290000000000000,   0.081500000000000)
   (8.690000000000000,   0.103000000000000)
   (9.640000000000001,   0.103000000000000)
  (10.699999999999999,   0.119000000000000)
  (11.699999999999999,   0.108000000000000)
  (12.680000000000000,   0.110000000000000)
  (13.900000000000000,   0.118000000000000)
  (14.680000000000000,   0.099400000000000)
  (15.590000000000000,   0.114000000000000)
  (16.850000000000001,   0.103000000000000)
  (17.750000000000000,   0.108000000000000)
  (18.780000000000001,   0.111000000000000)
  (19.550000000000001,   0.110000000000000)
  (20.660000000000000,   0.113000000000000)
    };
\addlegendentry{0.00871}
    
\addplot+[color=black!40!green, mark=diamond*, mark options={scale=1.5,fill=black!50!green}]
    coordinates {
   (                0,   0.01410000000000)
   (0.700000000000000,   0.030900000000000)
   (1.960000000000000,   0.022000000000000)
   (2.640000000000000,   0.040300000000000)
   (3.780000000000000,   0.039500000000000)
   (4.950000000000000,   0.060000000000000)
   (5.630000000000000,   0.094000000000000)
   (6.630000000000000,   0.10700000000000)
   (7.290000000000000,   0.098100000000000)
   (8.690000000000000,   0.131000000000000)
   (9.640000000000001,   0.121000000000000)
  (10.699999999999999,   0.121000000000000)
  (11.699999999999999,   0.106000000000000)
  (12.680000000000000,   0.111000000000000)
  (13.900000000000000,   0.116000000000000)
  (14.680000000000000,   0.098300000000000)
  (15.590000000000000,   0.118000000000000)
  (16.850000000000001,   0.102000000000000)
  (17.750000000000000,   0.105000000000000)
  (18.780000000000001,   0.116000000000000)
  (19.550000000000001,   0.109000000000000)
  (20.660000000000000,   0.101000000000000)
    };
\addlegendentry{0.01410}
 
\end{axis}
\end{tikzpicture}
		\caption{Data of the evolution of the tumor confluence over the duration of 21 days for three different initial tumor confluences}
		\label{Figure_Data}
	\end{figure}
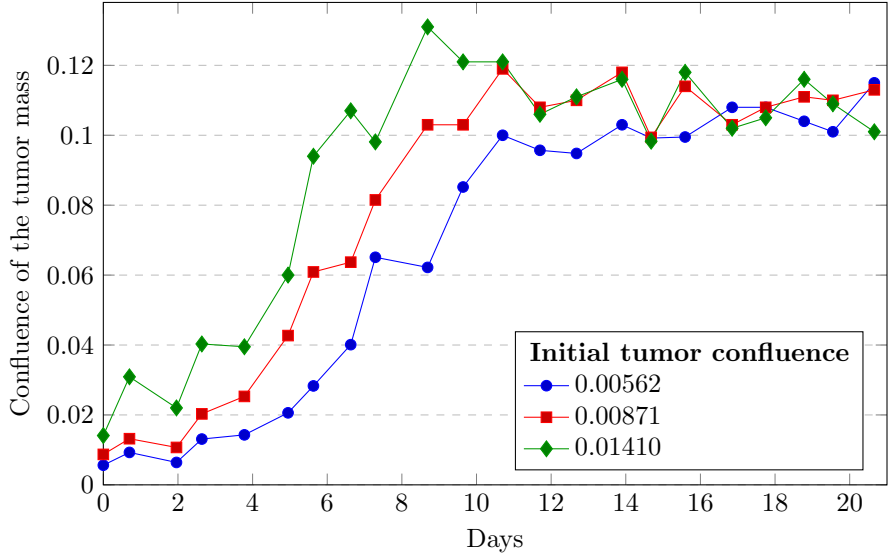

We selected initial tumor cells corresponding to confluences of $0.00562$, $0.00871$ and $0.01410$. For example, an initial confluence of $0.00562$ corresponds to an initial tumor volume of $A_{\text{init}}=0.00562 \cdot \pi \cdot 0.32^2\, \text{cm}^2 \approx 0.00181 \, \text{cm}^2$ with radius  $r_{\text{init}}=\sqrt{0.00562 \cdot 0.32^2} \text{cm} \approx 0.0240 \,\text{cm}$.

	The covariance matrix $V$ of $f$ is given by the integral of the outer product of the gradient of $f$ weighted with $\rho$,
	$$V = \int_{\mathbb{R}^{13}} \nabla f(x) (\nabla f(x))^\top \rho(x) \, \textup{d} x.$$
	In our algorithm we approximate the covariance matrix $V$ by
	$$V \approx \frac{1}{N} \sum_{j=1}^N \nabla f(X_j) (\nabla f(X_j))^\top, \quad X_j \sim \rho,$$
	whose eigenvalues are close to the true eigenvalues of $V$ for a sufficiently large $N$, see Refs.~\citen{constantine2014active,constantine2015active}.

	\noindent We consider an orthogonal eigendecomposition of $V$; that means for the eigenvalue matrix $\Lambda = \text{diag}(\lambda_1,\dots,\lambda_{13})$ with descending eigenvalues and the corresponding eigenvector matrix $W=[w_1 \cdot \cdot \cdot w_{13}]$ we have
	$$V = W \Lambda W^\top.$$
	Since $V$ is a symmetric and positive semi-definite matrix, its eigenvectors can be chosen to form an orthonormal basis in $\mathbb{R}^n$ and its eigenvalues are non-negative. In particular, the $i$-th eigenvalue is given by
	$$\lambda_i = w_i^\top V w_i = \int_{\mathbb{R}^{13}} \big(w_i^\top \nabla f(x)\big)^2 \rho(x) \, \textup{d} x,$$
	from which we can see that it reflects the sensitivity of $f$ in the direction of the $i$-th eigenvector. In other words, on average $f$ changes a little in the direction of an eigenvector with a small corresponding eigenvalue, and $f$ may change significantly in the direction of an eigenvector with a large corresponding eigenvalue. But this is not enough to identify the sensitivity of each parameter. Following Ref.~\citen{constantine2017global}, we define the activity score $\alpha_i$ for the $i$-th parameter $p_i$ as
	$$\alpha_i = \sum_{i=1}^{13} \lambda_j w_{i,j}^2, \quad i=1,\dots,13,$$
	and use the resulting number to rank the importance of each parameter. As discussed in Ref.~\citen{constantine2017global} this metric has fared well in comparison to other standard sensitivity metrics such as the Sobol' sensitivity index\cite{sobol2001global} when adequate data is available. 
	
	\subsection*{Comparison of the sensitivity methods} In Figure \ref{Figure_BarPlot1} we list the relative sensitivity of each parameter for each method. We see that the proliferation rate $\lambda_T$ is highly sensitive for both methods. For the active subspace method the effects of other parameters besides $\lambda_T$ and $\lambda_A$ are nearly zero for this choice of QoI. Therefore, to notice the difference, we also listed the comparison of the sensitives in Figure \ref{Figure_BarPlot2} with a logarithmic scale.

	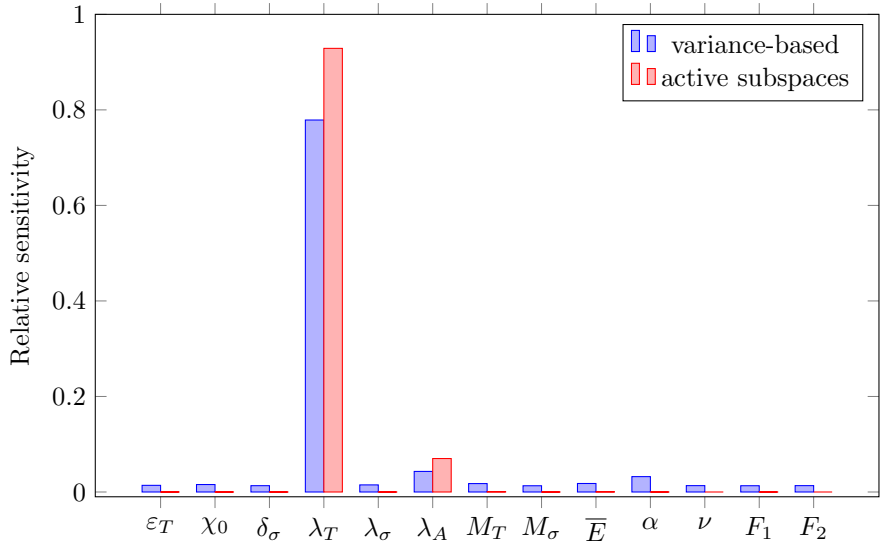
\begin{figure}[H]
		\centering
	\begin{tikzpicture} 
\begin{axis}[ 
        width=12cm,
        height=8cm,
        ybar=0pt, 
        ymin=-0.01,ymax=1,
        legend style={}, 
        ylabel={Relative sensitivity}, 
        symbolic x coords={$\varepsilon_T$, $\chi_0$, $\delta_\sigma$, $\lambda_T$, $\lambda_\sigma$, $\lambda_A$, $M_T$, $M_\sigma$, $\overline{E}$, $\alpha$, $\nu$, $F_1$, $F_2$}, 
        xtick=data, 
        bar width=7pt,
        ] 
\addplot coordinates {
        ($M_T$,0.01761807)
        ($M_\sigma$,0.01306688)
        ($\lambda_T$,0.77878138)
        ($\lambda_A$,0.04308982)
        ($\lambda_\sigma$,0.01486681)
        ($\overline{E}$,0.0178098)
        ($\varepsilon_T$,0.01393578)
        ($\chi_0$,0.01561902)
        ($\delta_\sigma$,0.01323875)
        ($\alpha$,0.03218827)
        ($\nu$,0.01332796)
        ($F_1$,0.01313001)
        ($F_2$,0.01332745)}; 
\addplot coordinates {
        ($\varepsilon_T$,6.40499196e-05) 
        ($\chi_0$, 4.31502163e-06)
        ($\delta_\sigma$,1.00033517e-06)
        ($\lambda_T$,9.28802174e-01)
        ($\lambda_\sigma$,2.52999623e-06)
        ($\lambda_A$,7.00567981e-02)
        ($M_T$,5.80711803e-04) 
        ($M_\sigma$,5.58987791e-07)
        ($\overline{E}$,4.81583383e-04) 
        ($\alpha$,6.27820362e-06)
        ($\nu$,3.35086341e-09)
        ($F_1$,3.68741255e-06)
        ($F_2$,4.53701572e-08)}; 
\legend{variance-based, active subspaces} 
\end{axis} 
\end{tikzpicture}
	\caption{Comparison of the relative sensitivities for the variance-based and the active subspaces method; linear scale} 
	\label{Figure_BarPlot1}
\end{figure}
	\begin{figure}[H]
		\centering
	\begin{tikzpicture} 
\begin{axis}[ 
        width=12cm,
        height=8cm,
        ybar=0pt, 
        ymode=log,
        log origin=infty,
        legend style ={},
        ylabel={Relative sensitivity}, 
        symbolic x coords={$\varepsilon_T$, $\chi_0$, $\delta_\sigma$, $\lambda_T$, $\lambda_\sigma$, $\lambda_A$, $M_T$, $M_\sigma$, $\overline{E}$, $\alpha$, $\nu$, $F_1$, $F_2$}, 
        xtick=data,
        bar width=7pt,
        ] 
\addplot coordinates {
        ($M_T$,0.01761807)
        ($M_\sigma$,0.01306688)
        ($\lambda_T$,0.77878138)
        ($\lambda_A$,0.04308982)
        ($\lambda_\sigma$,0.01486681)
        ($\overline{E}$,0.0178098)
        ($\varepsilon_T$,0.01393578)
        ($\chi_0$,0.01561902)
        ($\delta_\sigma$,0.01323875)
        ($\alpha$,0.03218827)
        ($\nu$,0.01332796)
        ($F_1$,0.01313001)
        ($F_2$,0.01332745)};   
\addplot coordinates {
        ($\varepsilon_T$,6.40499196e-05) 
        ($\chi_0$, 4.31502163e-06)
        ($\delta_\sigma$,1.00033517e-06)
        ($\lambda_T$,9.28802174e-01)
        ($\lambda_\sigma$,2.52999623e-06)
        ($\lambda_A$,7.00567981e-02)
        ($M_T$,5.80711803e-04) 
        ($M_\sigma$,5.58987791e-07)
        ($\overline{E}$,4.81583383e-04) 
        ($\alpha$,6.27820362e-06)
        ($\nu$,3.35086341e-09)
        ($F_1$,3.68741255e-06)
        ($F_2$,4.53701572e-08)}; 
\legend{variance-based, active subspaces} 
\end{axis} 
\end{tikzpicture} 
	\caption{Comparison of the relative sensitivities for the variance-based and the active subspaces method; logarithmic scale} 
	\label{Figure_BarPlot2}
\end{figure}
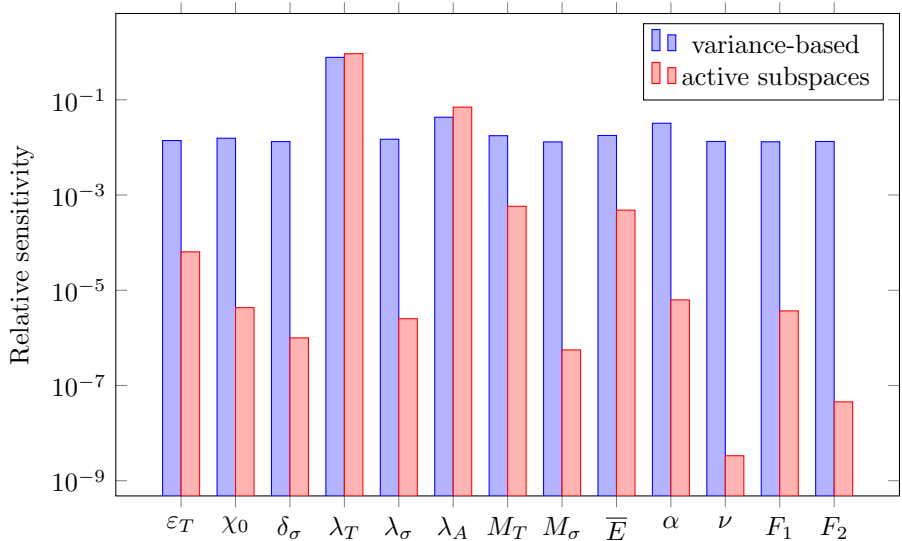
\vspace{-0.2cm}

		We remark that the dominance of the proliferation rate $\lambda_T$ in the results shown is largely due to the choice of the QoI; other choices favor different parameters of the model. For example, numerical simulations described in the next section suggest that the velocity parameters play an important role in determine the evolution of tumor shape and, for example, surface area.

	\section{Numerical Discretization and Examples} \label{Section_Numerical}
	We choose a similar computational framework as in Refs.~\citen{lima2014hybrid,lima2015analysis} to solve the deterministic system (\ref{Eq_Model}) with the initial and boundary data (\ref{Eq_InitialBoundary}). This framework includes a discrete-time local semi-implicit scheme with an energy convex-nonconvex splitting. Here, the stable contractive part is treated implicitly and the expansive part explicitly. In particular, introducing the Ginzburg--Landau energy\cite{gurtin1996generalized}
	$$E=\int_\Omega \Psi(\phi_T)+\frac{\varepsilon_T^2}{2} |\nabla \phi_T|^2 + \frac{1}{2\delta_\sigma} \phi_\sigma^2 -\chi_0 \phi_T\phi_\sigma \, \textup{d} x$$
	we can rewrite the chemical potential in the following way,
	$$\mu =  \frac{\delta E}{\delta \phi_T},$$
	where $\frac{\delta E}{\delta \phi_T}$ is the first variation of $E$ with respect to $\phi_T$. We split the energy in its contractive and expansive parts via $E=E_c-E_e$. 
	
	Let the time domain be divided into the steps $\Delta t_n = t_{n+1} - t_n$ for $n \in \{0,1,\dots\}$. We assume $\Delta t_n = \Delta t$ for all $n$. We write $\phi_{T_n}$ for the approximation of $\phi_T^h(t_n)$ and likewise for the other variables. The backward Euler method applied to the system (\ref{Eq_Model}) reads
	\begin{equation}
	\begin{aligned}
	\frac{\phi_{T_{n+1}} - \phi_{T_n}}{\Delta t}+\div(\phi_{T_{n+1}} v_{n+1}) &= \div \! \big(m_T(\phi_{T_{n+1}},\phi_{\sigma_{n+1}}) \nabla \mu_{n+1} \big) - \lambda_A \phi_{T_{n+1}}
	\\[-0.1cm] &\quad\,+ \lambda_T \phi_{\sigma_{n+1}} \mathcal{C}\big(\phi_{T_{n+1}}(1-\phi_{T_{n+1}})\big), \\[0.1cm]
	\mu_{n+1} &= D_{\phi_T} E_c(\phi_{T_{n+1}},\phi_{\sigma_{n+1}})-D_{\phi_T} E_e (\phi_{T_{n}},\phi_{\sigma_{n}}), \\
	\frac{\phi_{\sigma_{n+1}} - \phi_{\sigma_n}}{\Delta t} +\div( \phi_{\sigma_{n+1}}  v_{n+1})&= \div\! \big(m_\sigma(\phi_{T_{n+1}},\phi_{\sigma_{n+1}}) (\delta_\sigma^{-1} \nabla \phi_{\sigma_{n+1}} - \chi_0 \nabla \phi_{T_{n+1}}) \big) \\ &\quad\, -\lambda_\sigma \phi_{\sigma_{n+1}} \mathcal{C}(\phi_{T_{n+1}}), \\   
	\frac{v_{n+1}-v_n}{\Delta t} + \alpha v_{n+1} +  \nabla p_{n+1}&= \div \!\big(\nu(\phi_{T_{n+1}},\phi_{\sigma_{n+1}}) \nabla v_{n+1}\big) -F_1 |v_{n+1}|v_{n+1} - F_2 |v_{n+1}|^2 v_{n+1}  \\[-0.1cm] &\quad\, + (\mu_{n+1} + \chi_0 \phi_{\sigma_{n+1}} ) \phi_{T_{n+1}}, \\[0.1cm]
	\div v_{n+1} &=0,
	\end{aligned}
	\label{Eq_Numerics}
	\end{equation}
	where $\mathcal{C}(\phi_{T_{n+1}}) = \max\!\big(0,\min(1,\phi_{T_{n+1}})\big)$ is the cut-off operator.
	
	We uncouple the equations and use an iterative Gau\ss--Seidel method for solving each equation. In Algorithm \ref{Alg_Alg} below, the subscript $0$ stands for the initial solution, $k$ the iteration index, $n_{\text{iter}}$ the maximum number of iterations at each time step and TOL the tolerance for the iteration process.
	In each iterative loop three linear systems are solved and the convergence of the nonlinear solution is achieved at each time if $\max |\phi_{T_{n+1}}^{k+1}-\phi_{T_{n+1}}^{k}| < \text{TOL}$. We obtain the algebraic systems using a Galerkin finite element approach. In this regard, let $\mathcal{T}^h$ be a quasiuniform family of triangulations of $\Omega$ and let the piecewise linear finite element space be given by
	$$\mathcal{V}^h = \{ v \in H^1(\overline{\Omega}) : v|_{T} \in P_1(T) \text{ for all } T \in \mathcal{T}^h \} \subset H^1(\Omega),$$
	where $P_1(T)$ denotes the set of all affine linear function on $T$. Moreover, we introduce the piecewise linear finite element space with homogeneous Dirichlet boundary
	$$\mathcal{V}^h_0 = \{ v \in \mathcal{V}^h : v = 0 \text{ on }  \partial \Omega \},$$
	and for the divergence-free space $V$ we consider the Brezzi-Douglas-Marini (BDM) space of order 1, see Ref.~\citen{boffi2013mixed}. In Ref.~\citen{pan2012mixed} it was shown that the mixed finite element space $\text{BDM}_1$--$\text{DG}_0$ is stable for the mixed formulation of the Darcy--Forchheimer equation. 
	
	We formulate the discrete problem as follows: for each $k$, find $$\phi_{T_{n+1}}^{k+1} \in \mathcal{V}^h,\quad  \mu_{n+1}^{k+1} \in \mathcal{V}^h,\quad  \phi_{\sigma_{n+1}}^{k+1} \in 1+\mathcal{V}_0^h,\quad v^{k+1}_{n+1} \in \text{BDM}_1,\quad p^{k+1}_{n+1}\in \text{DG}_0,$$ 
	for all 
	$$\varphi_T\in \mathcal{V}^h,\quad \varphi_\mu \in \mathcal{V}^h,\quad \varphi_\sigma \in \mathcal{V}^h_0,\quad \varphi_v \in \text{BDM}_1,\quad \varphi_p \in \text{DG}_0,$$ such that:
	\begin{equation} \begin{aligned}
	(v_{n+1}^{k+1}-v_n,\varphi_v) &+ \Delta t \alpha ( v_{n+1}^{k+1},\varphi_v) 
	\\&+ \Delta t \big(\nu(\phi_{T_{n+1}}^k,\phi^k_{\sigma_{n+1}}) \nabla v_{n+1}^{k+1},\nabla \varphi_v\big) \\&+\Delta t F_1 (|v_{n+1}^{k+1} |v_{n+1}^{k+1},\varphi_v) \\&+ \Delta t F_2 (|v_{n+1}^{k+1}|^2 v_{n+1}^{k+1},\varphi_v) \\& - \Delta t ( p_{n+1}^{k+1},\div \varphi_v) \\ &= - \Delta t\big( (\mu_{n+1}^k + \chi_0 \phi^k_{\sigma_{n+1}} ) \phi^k_{T_{n+1}},\varphi_v\big);
	\label{Alg_v}
	\end{aligned}
	\end{equation}
	\begin{equation} \begin{aligned}
	(v_{n+1}^{k+1},\nabla \varphi_p) &=0;
	\label{Alg_p}
	\end{aligned}\end{equation}
	\begin{equation} \begin{aligned}
	(\phi_{\sigma_{n+1}}^{k+1} - \phi_{\sigma_n},\varphi_\sigma) &- \Delta t  (\phi_{\sigma_{n+1}}^{k+1} v_{n+1}^{k+1},\nabla \varphi_\sigma) \\&+\Delta t\big(m_\sigma(\phi_{T_{n+1}}^k,\phi_{\sigma_{n+1}}^{k+1}) \cdot (\delta_\sigma^{-1} \nabla \phi_{\sigma_{n+1}}^{k+1} - \chi_0 \nabla \phi^k_{T_{n+1}}),\nabla \varphi_\sigma\big) 
	\\&+ \Delta t \lambda_\sigma \big(\phi^{k+1}_{\sigma_{n+1}} \mathcal{C}(\phi^k_{T_{n+1}}) ,\varphi_\sigma\big) = 0;
	\label{Alg_sigma}
	\end{aligned}
	\end{equation}
		\begin{equation} \begin{aligned}
	(\phi_{T_{n+1}}^{k+1} - \phi_{T_n},\varphi_T) &- \Delta  t (\phi_{T_{n+1}}^{k+1} v_{n+1}^{k+1},\nabla \varphi_T)
	\\&+\Delta t  \big(m_T(\phi_{T_{n+1}}^{k+1},\phi^{k+1}_{\sigma_{n+1}}) \nabla \mu_{n+1}^{k+1},\nabla \varphi_T\big) \\&-\Delta t \lambda_T \big(\phi^{k+1}_{\sigma_{n+1}} \mathcal{C}\big(\phi^{k+1}_{T_{n+1}}(1-\phi^{k+1}_{T_{n+1}})\big),\varphi_T\big)=0;
	\label{Alg_T}
	\end{aligned}
	\end{equation}
	\begin{equation} \begin{aligned}
	(\mu_{n+1}^{k+1},\varphi_\mu) - \big(D_{\phi_T} E_c(\phi^{k+1}_{T_{n+1}},\phi^{k+1}_{\sigma_{n+1}}),\varphi_\mu\big)=\big(D_{\phi_T} E_e (\phi_{T_{n}},\phi_{\sigma_{n}}),\varphi_\mu\big).
	\label{Alg_mu}
	\end{aligned}
	\end{equation}
	
	\begin{algorithm}[H]
		\SetAlgoLined
		\caption{Semi-implicit scheme for (\ref{Eq_Numerics})} \label{Alg_Alg} \vspace{2mm}
		\textbf{Input}: $\phi_{T_0}, \phi_{\sigma_0}, v_0, \Delta t, T, \text{TOL}$\\[2mm]
		\textbf{Output}: $\phi_{T_n}, \mu_n, \phi_{\sigma_n}, v_n$ for all $n$ \\[2mm]
		$t=0$, $n=0$\vspace{2mm}\\
		\While{$t\leq T$ \vspace{1mm}}{ 
			$\phi^0_{T_{n+1}} = \phi_{T_n}$ \\[2mm]
			\While{$\displaystyle \textup{max} \| \phi_{T_{n+1}}^{k+1} - \phi_{T_{n+1}}^k\| > \textup{TOL}$ \vspace{1mm}}{
				$\phi_{T_{n+1}}^k = \phi_{T_{n+1}}^{k-1}$ \\[2mm]
				\textbf{solve \bm{$v_{n+1}^{k+1}$}, \bm{$p_{n+1}^{k+1}$} using \eqref{Alg_v} and \eqref{Alg_p}, given} $\phi_{\sigma_n}, \phi_{T_{n+1}}^k, \mu_{n+1}^k, \phi_{\sigma_{n+1}}^k $\\[2mm]
				\textbf{solve \bm{$\phi_{\sigma_{n+1}}^{k+1}$} using \eqref{Alg_sigma}, given} $\phi_{\sigma_n}, \phi_{T_{n+1}}^k, \bm{v_{n+1}^{k+1}}$\\[2mm]
				\textbf{solve \bm{$\phi_{T_{n+1}}^{k+1}$}, \bm{$\mu_{n+1}^{k+1}$} using \eqref{Alg_T} and \eqref{Alg_mu}, given} $\phi_{T_n}, \phi_{T_{n+1}}^k, \bm{\phi_{\sigma_{n+1}}^{k+1}}, \bm{v_{n+1}^{k+1}}$ \\[2mm]
				$k \mapsto k+1$\\} 
			$\phi_{T_{n+1}}= \phi_{T_{n+1}}^{k+1}$ \\[2mm]
			$\mu_{n+1} = \mu_{n+1}^{k+1}$ \\[2mm]
			$\phi_{\sigma_{n+1}} = \phi_{\sigma_{n+1}}^{k+1}$ \\[2mm]
			$v_{n+1} = v_{n+1}^{k+1}$ \\[2mm]
			$p_{n+1} = p_{n+1}^{k+1}$ \\[2mm]
			$n \mapsto n+1$, $\, t \mapsto t+\Delta t$} 
	\end{algorithm}
	
	We implemented Algorithm \ref{Alg_Alg} in FEniCS\cite{logg2012automated}, an open-source computing platform for solving partial differential equations using finite element methods. We use this implementation to obtain the numerical results below. 
	
	\subsection*{Spherically symmetric case}
	For specificity and demonstration purposes, we assume a spherically symmetric tumor volume and use polar coordinates to transform our model (\ref{Eq_Model}) to a system of equations depending solely on time $t$ and radius $r$. We consider the domain $R=[0,0.32]$ imitating our data setting as described in Section \ref{Section_Sens}. We choose parameters matching the priors we used in (\ref{Eq_Priors}); in particular we choose the dimensionless values
	\begin{equation} \begin{alignedat}{7} 
	\varepsilon_T &=0.01,          \quad &\lambda_T &=1.0,    \quad &M_T &=1.0,    \quad &\alpha &=1.0, \\
	\chi_0 &=0.5,   \quad &\lambda_\sigma &=1.0, \quad &M_\sigma &=1.0,               \quad &\nu &=10.0, \\
	\delta_\sigma &=0.05,    \quad  &\lambda_A &=0.01,  \quad &\overline{E} &=0.25,              \quad &F_1,F_2 &=10.0.
	\end{alignedat}
	\label{Eq_ParametersSimulation}
	\end{equation}
	
	We consider a smooth approximation of the Heaviside function matching the initial tumor confluence $0.00562$ of our data setting, as it can be seen in Figure \ref{Figure_Simulation1DTumor}\,(a) below. The approximation is given by
	$$\phi_T(0,r)=\frac{1}{1+\exp(M(r-r_{\text{init}})) },$$
	where a larger $M>0$ is increasing the steepening of the function around zero, and  $r_{\text{init}}=0.32\sqrt{0.00562}$ again represents the radius of the initial tumor cell.
	
	The simulation of the evolution of the tumor confluence is depicted in Figure \ref{Figure_Simulation1DConfluence} below. We observe that the confluence grows continuously in time with an increasing rate.
			\begin{figure}[H]
		\centering
		\input{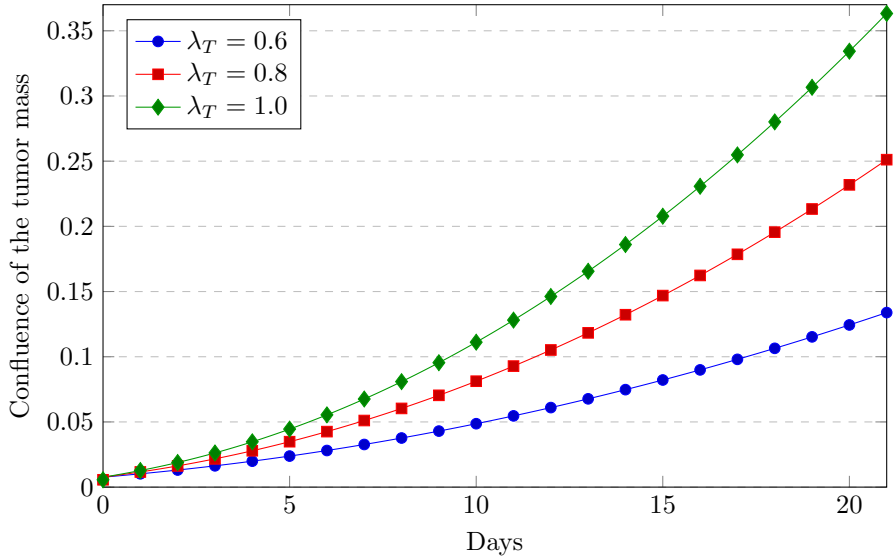} \\[-0.1cm]
		\caption{Simulation of the tumor confluence for three different proliferation rates $\lambda_T$ over the duration of 21 days}
		\label{Figure_Simulation1DConfluence}
	\end{figure}
	
	In Figure \ref{Figure_Simulation1DTumor}, the simulation of the tumor cell $\phi_T$ is shown at different time spots. First, the initial tumor cell is illustrated, then after 7, 14 and 21 days. We plot three different curves for different proliferation rates, 0.6, 0.8, and 1.0. We observe that a higher proliferation rate is increasing the expansion of the tumor cell volume fraction $\phi_T$, and the tumor cell is continuously growing over time.
		\begin{figure}[H]
		\centering
		\input{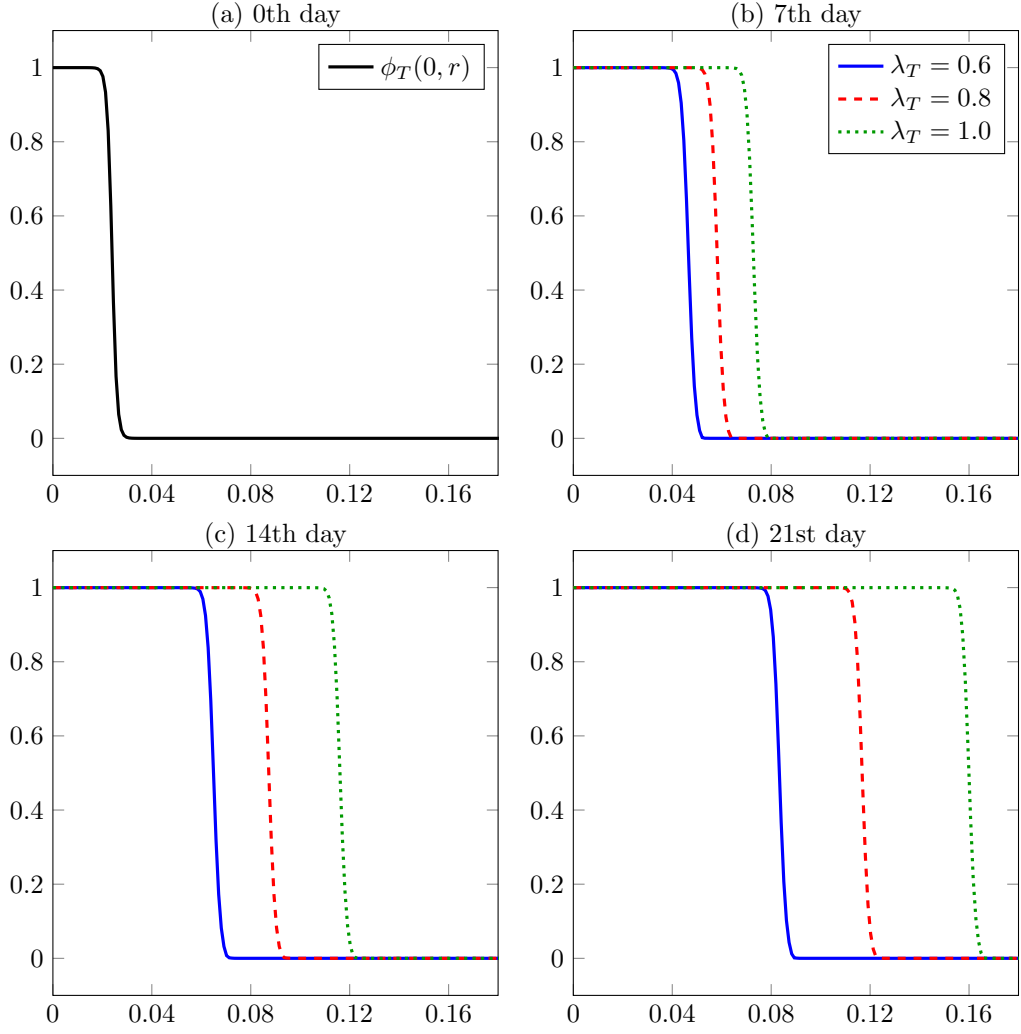} \\[-0.1cm]
		\caption{Simulation of the evolution of the tumor cell volume fraction $\phi_T$ over the duration of 21 days for three different proliferation rates}
		\label{Figure_Simulation1DTumor}
	\end{figure}

	\subsection*{Two-dimensional case}
	We simulate the tumor growth on the circular domain $$\Omega=\{x \in \mathbb{R}^2: x_1^2+x_2^2 = 0.32^2\}$$ with the same parameters as chosen in the one-dimensional setting; see (\ref{Eq_ParametersSimulation}). For the initial tumor volume we select the following four possibilities, which are depicted in Figure \ref{Figure_Initial2D}:
	
	(a) $\phi_{T}(0,x) = \begin{cases} 1 , &\text{if }~  0.9 \cdot x_1^2 + x_2^2 \leq  r_\text{init}^2,\\0, &\text{else}, \end{cases} $ \\[0.1cm]
	
	(b) $\phi_{T}(0,x) = \begin{cases} 1 , &\text{if }~  0.15 \cdot x_1^2 + x_2^2 \leq  r_\text{init}^2,\\0, &\text{else}, \end{cases} $ \\[0.1cm]
	
	(c) $\phi_{T}(0,x) = \begin{cases} 1 , &\text{if }~  0.9 \cdot (x_1\pm 0.05)^2 + x_2^2 \leq  r_\text{init}^2,\\0, &\text{else}, \end{cases} $ \\[0.1cm]
	
	(d) $\phi_{T}(0,x) = \begin{cases} 1 , &\text{if }~  (\sin(7.2x_1+5.6x_2)+1) \cdot (4x_1-0.2)^2 + (\sin(8x_1)+1) \cdot 64 x_2^2 \leq  1,\\0, &\text{else}. \end{cases}$
	
	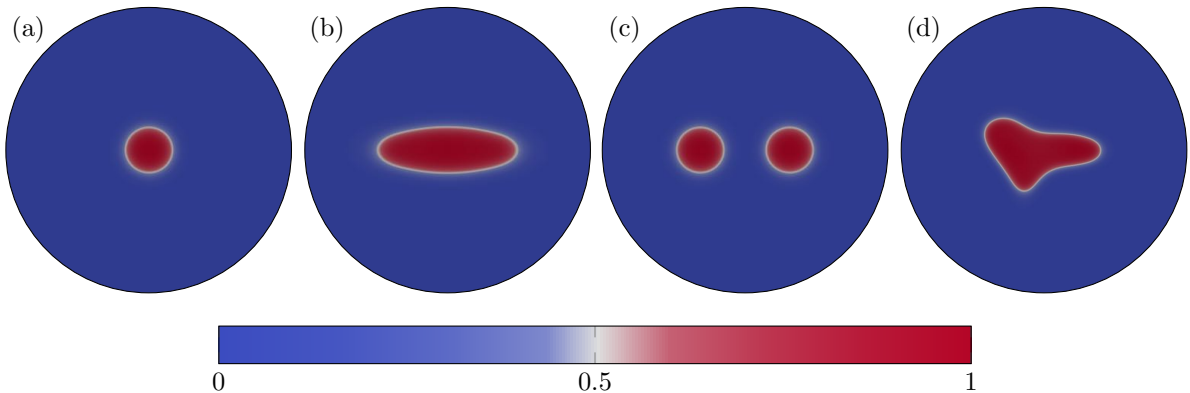
\begin{figure}[H]
		\centering
		\begin{tikzpicture}
		\node [draw,circle, minimum width=.23\textwidth,
		path picture = {
			\node at (path picture bounding box.center) { \includegraphics[width=.7\textwidth]{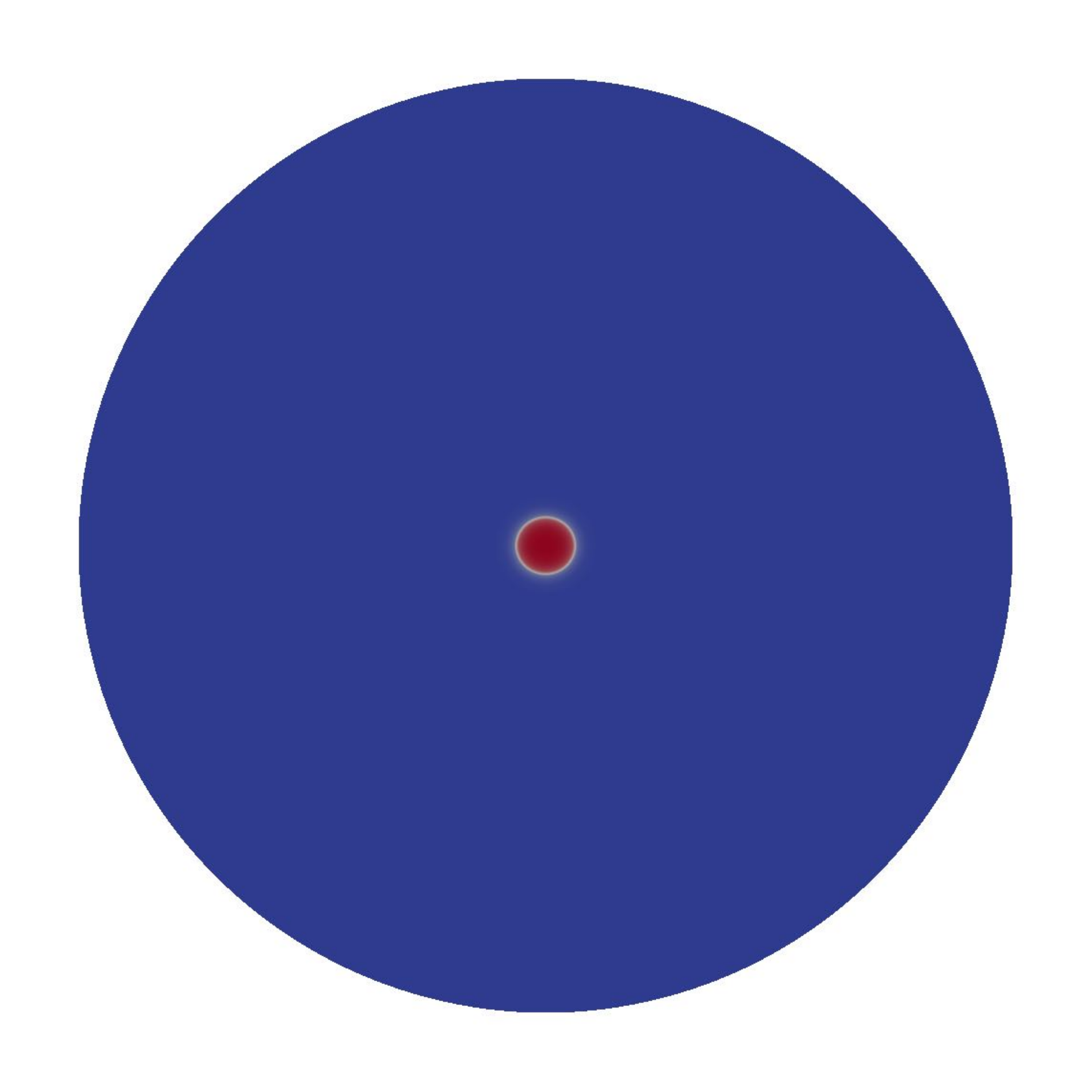}};}] {};
		\node at (-1.6,1.6) {(a)};
		\end{tikzpicture}
		\begin{tikzpicture}
		\node [draw,circle, minimum width=.23\textwidth,
		path picture = {
			\node at (path picture bounding box.center) {\includegraphics[width=.7\textwidth]{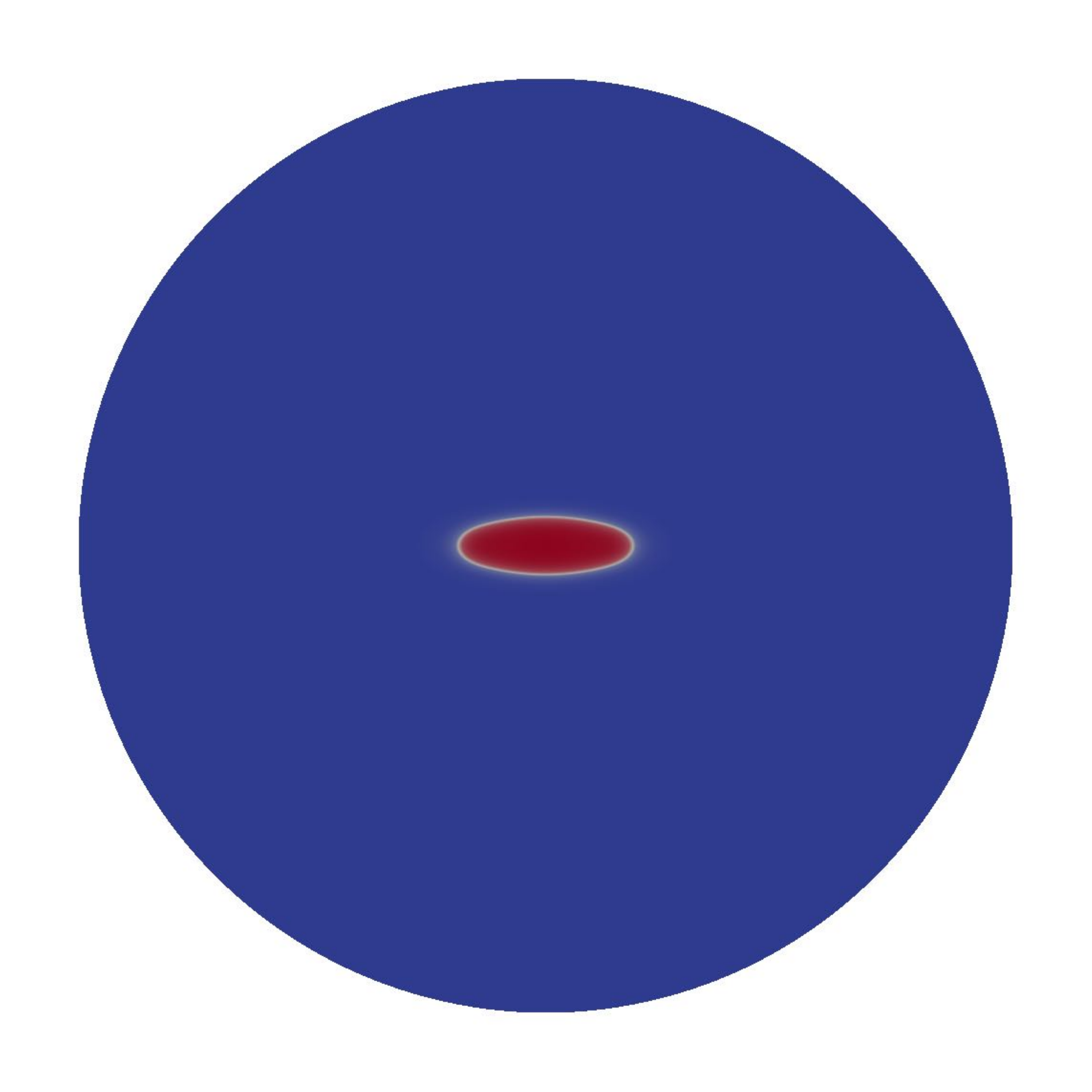}};}] {};
		\node at (-1.6,1.6) {(b)};
		\end{tikzpicture}
		\begin{tikzpicture}
		\node [draw,circle, minimum width=.23\textwidth,
		path picture = {
			\node at (path picture bounding box.center) { \includegraphics[width=.7\textwidth]{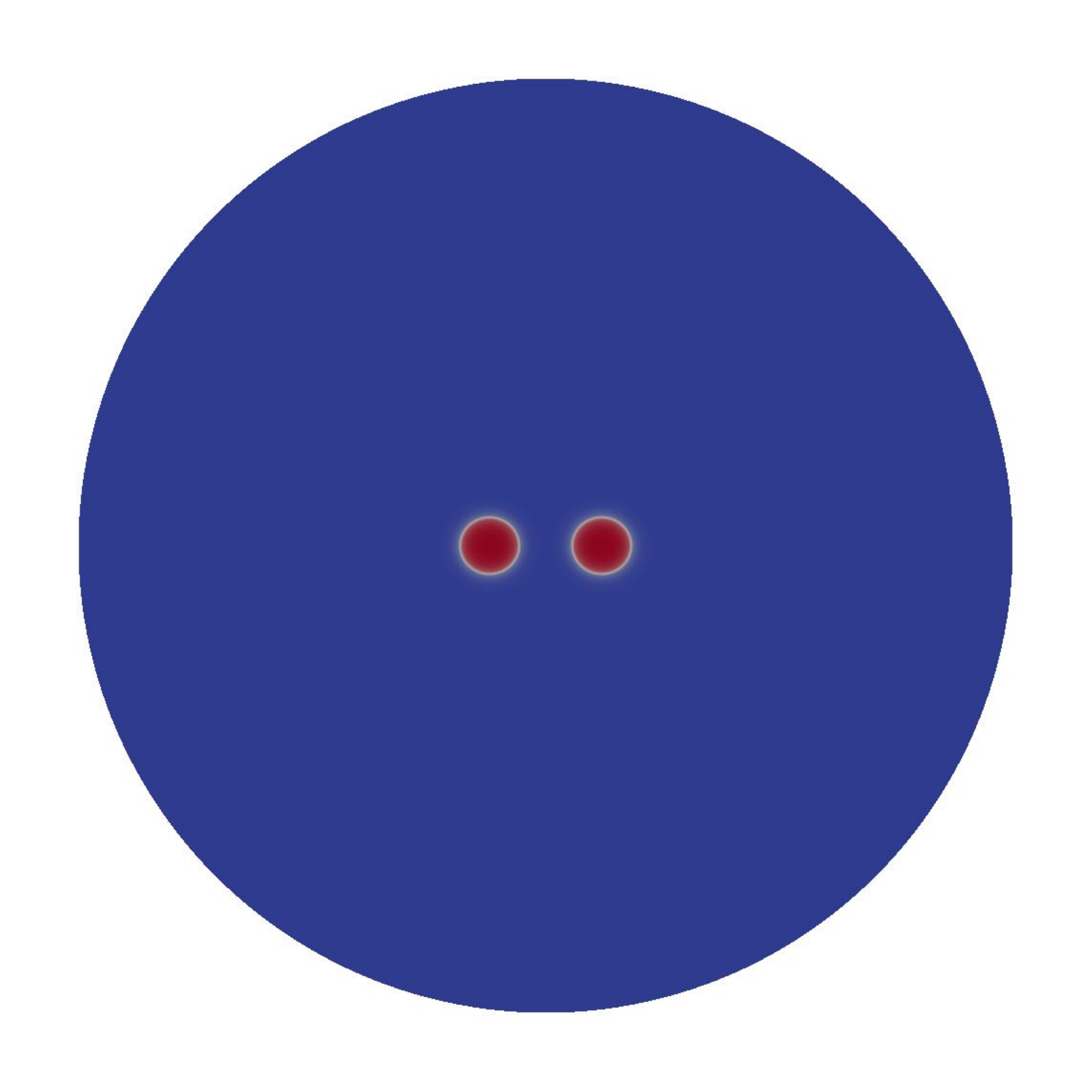}};}] {};
		\node at (-1.6,1.6) {(c)};
		\end{tikzpicture}
		\begin{tikzpicture}
		\node [draw,circle, minimum width=.23\textwidth,
		path picture = {
			\node at (.25,0) { \includegraphics[width=.45\textwidth]{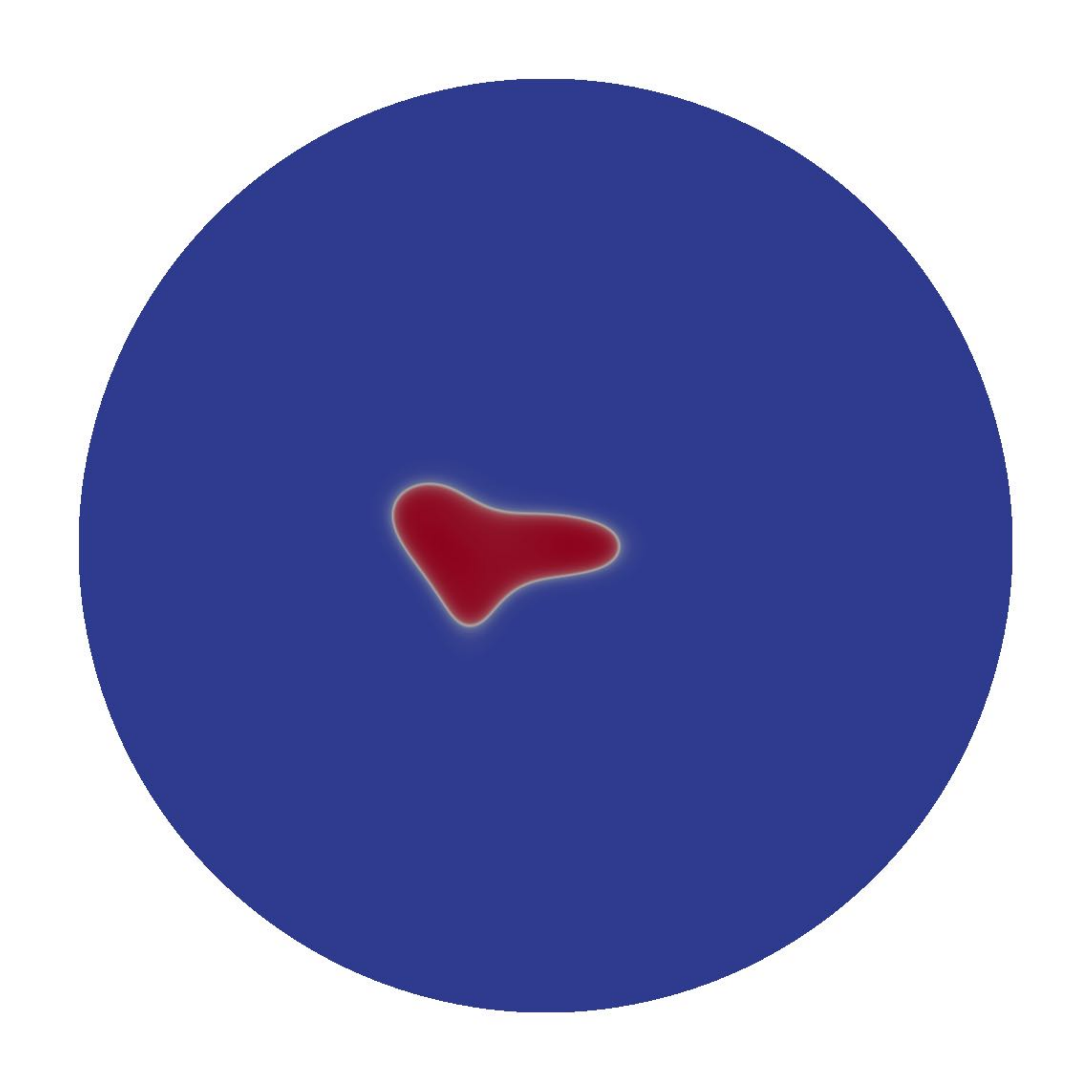}};}] {};
		\node at (-1.6,1.6) {(d)};
		\end{tikzpicture} \\[0.1cm]
		\begin{tikzpicture}
		\begin{axis}[
		hide axis,
		scale only axis,
		height=0pt,
		width=0pt,
		colormap name=special,
		colorbar horizontal,
		point meta min=0,
		point meta max=1,
		colorbar style={
			samples=100,
			height=.5cm,
			xtick={0,0.5,1},
			width=10cm
		},
		]
		\addplot [draw=none] coordinates {(0,0)};
		\end{axis}
		\end{tikzpicture}
		\caption{Choices for the initial tumor mass $\phi_{T,0}$, (a) slightly elliptic, (b) highly elliptic, (c) separated, (d) irregularly perturbed}
		\label{Figure_Initial2D}
	\end{figure}
\vspace{-0.5cm}

	\noindent In Figure \ref{Figure_Simulation2DElliptic} below,  we show the evolution of the slightly elliptic initial tumor mass (a) using 
	\begin{enumerate}[label=\Roman*,align=Center]
		\item model (\ref{Eq_Model}) without any influence of the velocity, that means we set $v \equiv 0$ and neglect the convection terms in the equations of $\phi_T$ and $\phi_\sigma$;
		\item model (\ref{Eq_Model}) without the effect of the Forchheimer law, that means we set $F_1=F_2=0$ in the velocity equation;
		\item the full model (\ref{Eq_Model}) without any restrictions.
	\end{enumerate} 

		Afterwards, we simulate the entire model III together with the initial tumor volumes (b), (c) and (d), see Figure \ref{Figure_Perturbed3} below for the results.
	
		The simulation of the full local model III with a slightly elliptic initial tumor mass (a) is depicted in the bottom row of Figure \ref{Figure_Simulation2DElliptic}. Similar to Cristini \textit{et al}\cite{cristini2009nonlinear,cristini2003nonlinear,wise2008three,wise2011adaptive,zheng2005nonlinear}, Macklin \textit{et al}\cite{macklin2005evolving,macklin2006improved,macklin2008new}, and Garcke \textit{et al}\cite{garcke2016cahn}, we notice an evolving shape instability. Starting from the slightly elliptic initial tumor mass, the ellipticity is enforced at the beginning and on the 10th day we see a clearly elliptic tumor volume. At the 15th day, a slight bulge forms along the horizontal direction and two buds form at the horizontal end points. These buds continue to evolve vertically with a new bulge oriented along the vertical directions and therefore, for each bud two new buds are forming, see the simulation on the 21st day. This behavior of the tumor cells implies that the instability repeats itself and  this highly complex evolution in tumor shape is captured by the high-order phase-field structure of the model and is indicative of examples in tumor growth in living tissue.
	
	 To inspect the effects of the velocity itself in the model, we redo the first simulation with the same initial data (a) but without the presence of any velocity, which means we are in the case of model I as described above. We depict the results of the simulation in the first row of Figure \ref{Figure_Simulation2DElliptic}. We observe that the tumor stays in its symmetric shape, resembling the results in Ref.~\citen{hawkins2012numerical}. We conclude that the velocity highly influences the shape of the tumor, even though the velocity parameters do not mainly impact the tumor volume as the sensitivity analysis has shown in Section \ref{Section_Sens}.
	
		In the next simulation, see the middle row of Figure \ref{Figure_Simulation2DElliptic} for the result, we use the slightly elliptic initial data (a) and model II, that means we set the Forchheimer constants $F_1$ and $F_2$ equal to zero. We observe that the result largely resembles the simulation of the full case III\,(a), in the sense that the Forchheimer terms delay the tumor evolution. We notice that the tumor mass splits into two parts, which begin to approach each other on the lower and upper bulbs. Eventually, these buds reconnect and therefore, trapping the healthy tissue within the tumor, which has also been observed in Ref.~\citen{cristini2003nonlinear}.	
		
		\pagebreak

		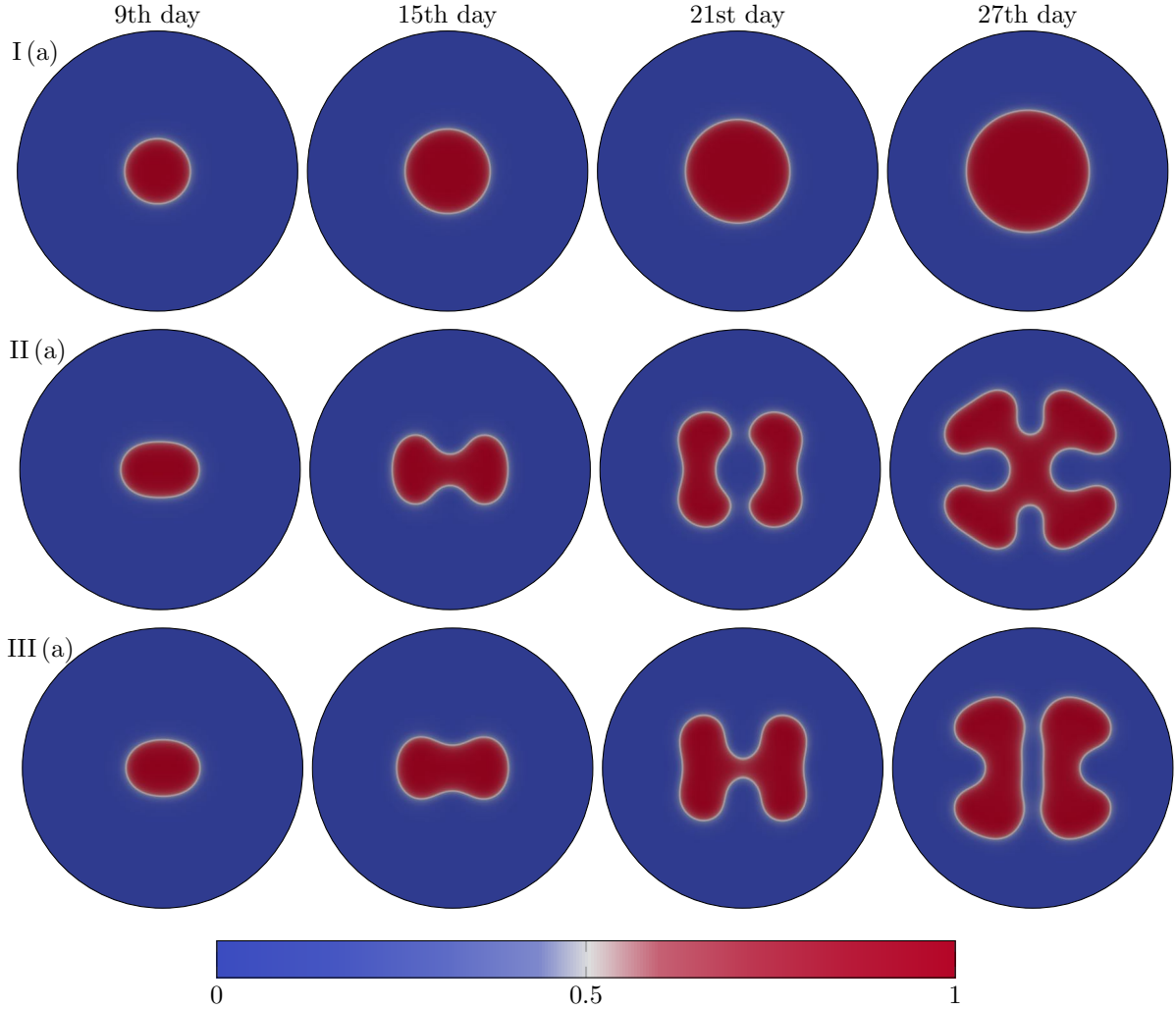
\begin{figure}[H]
		\centering
		\begin{tikzpicture}
		\node [draw,circle, minimum width=.23\textwidth,
		path picture = {
			\node at (path picture bounding box.center) { \includegraphics[width=.6\textwidth]{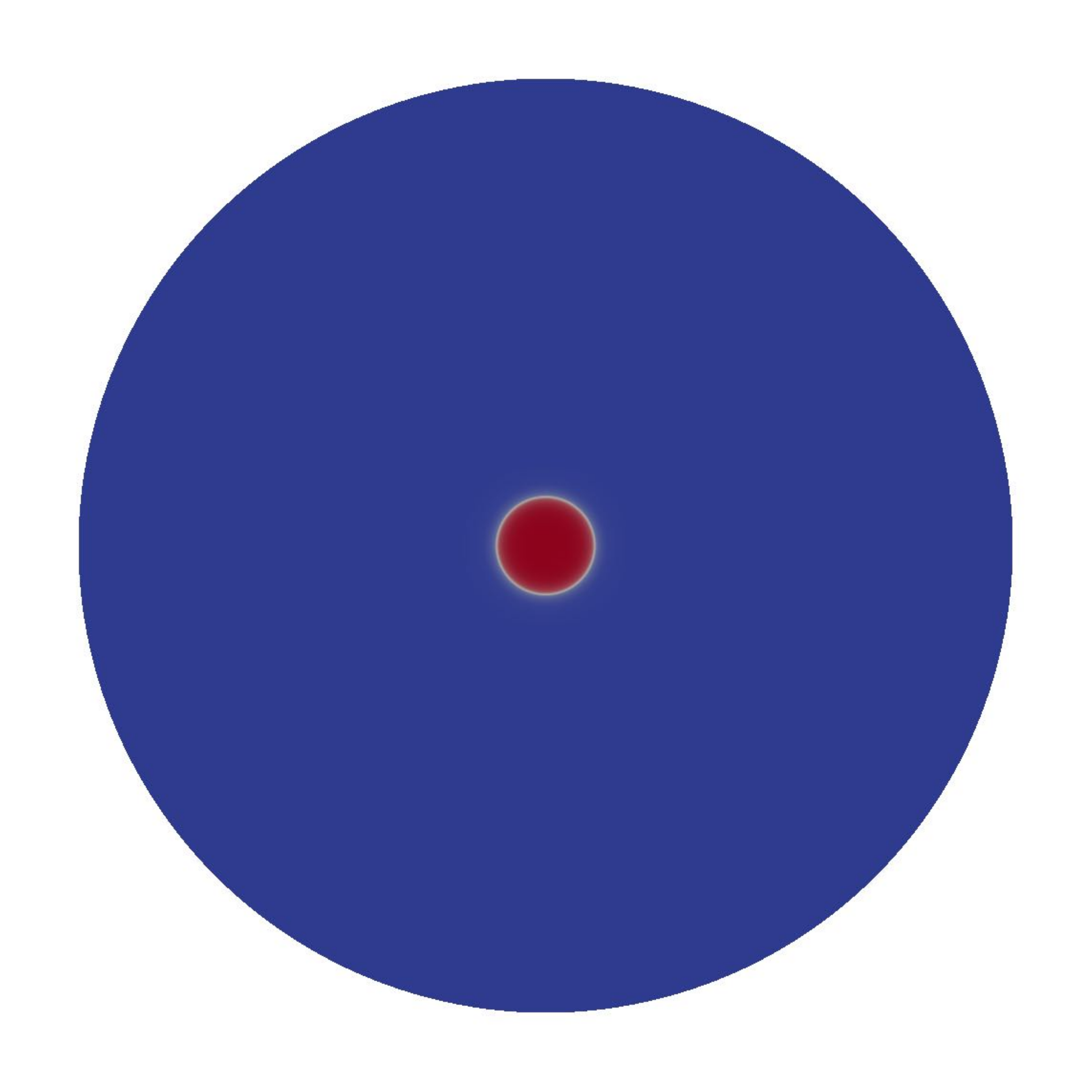}};}] {};
		\node at (-1.65,1.6) {I\,(a)};
		\node at (0,2.1) {9th day};
		\end{tikzpicture}
		\begin{tikzpicture}
		\node [draw,circle, minimum width=.23\textwidth,
		path picture = {
			\node at (path picture bounding box.center) {\includegraphics[width=.6\textwidth]{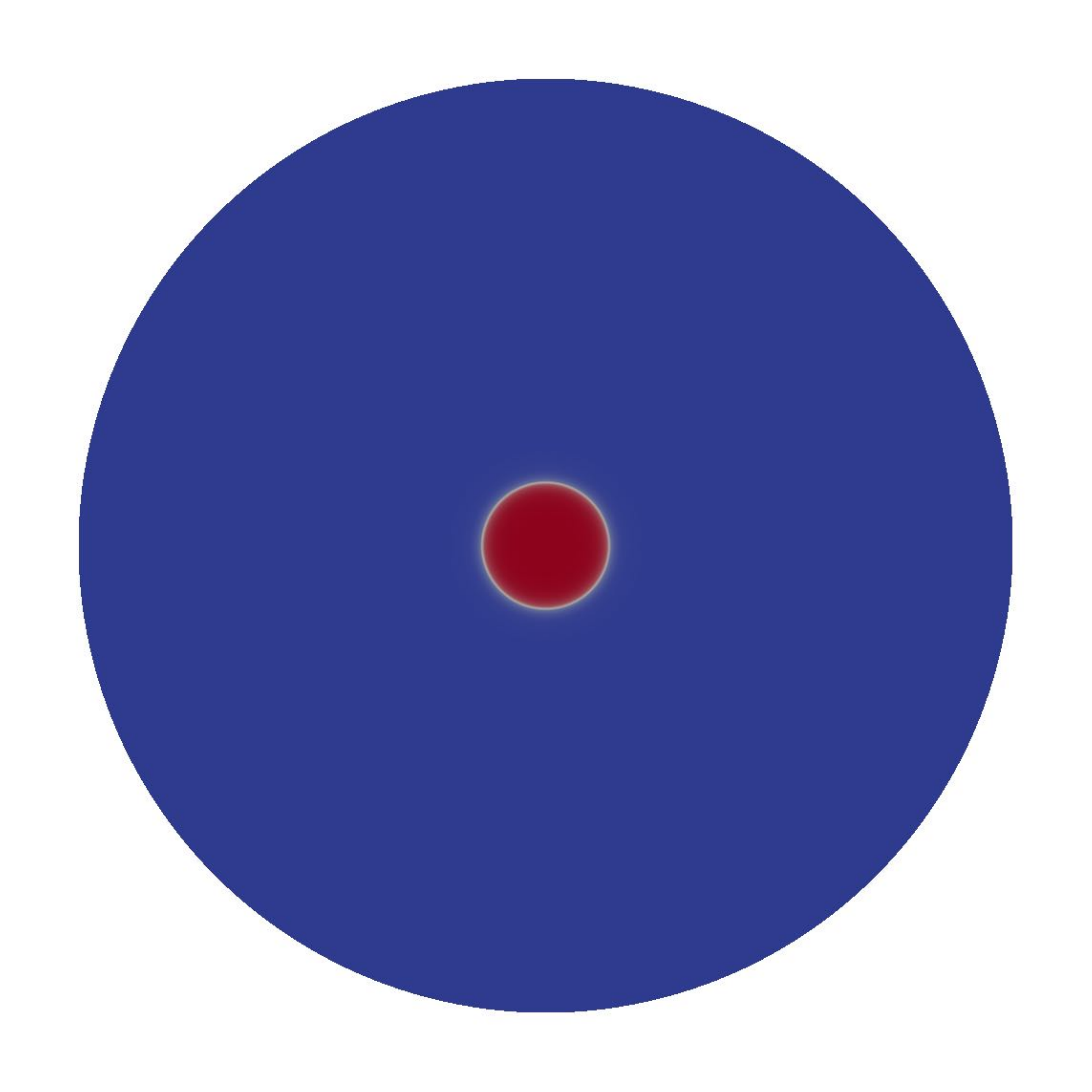}};}] {};
		\node at (0,2.1) {15th day};
		\end{tikzpicture}
		\begin{tikzpicture}
		\node [draw,circle, minimum width=.23\textwidth,
		path picture = {
			\node at (path picture bounding box.center) { \includegraphics[width=.6\textwidth]{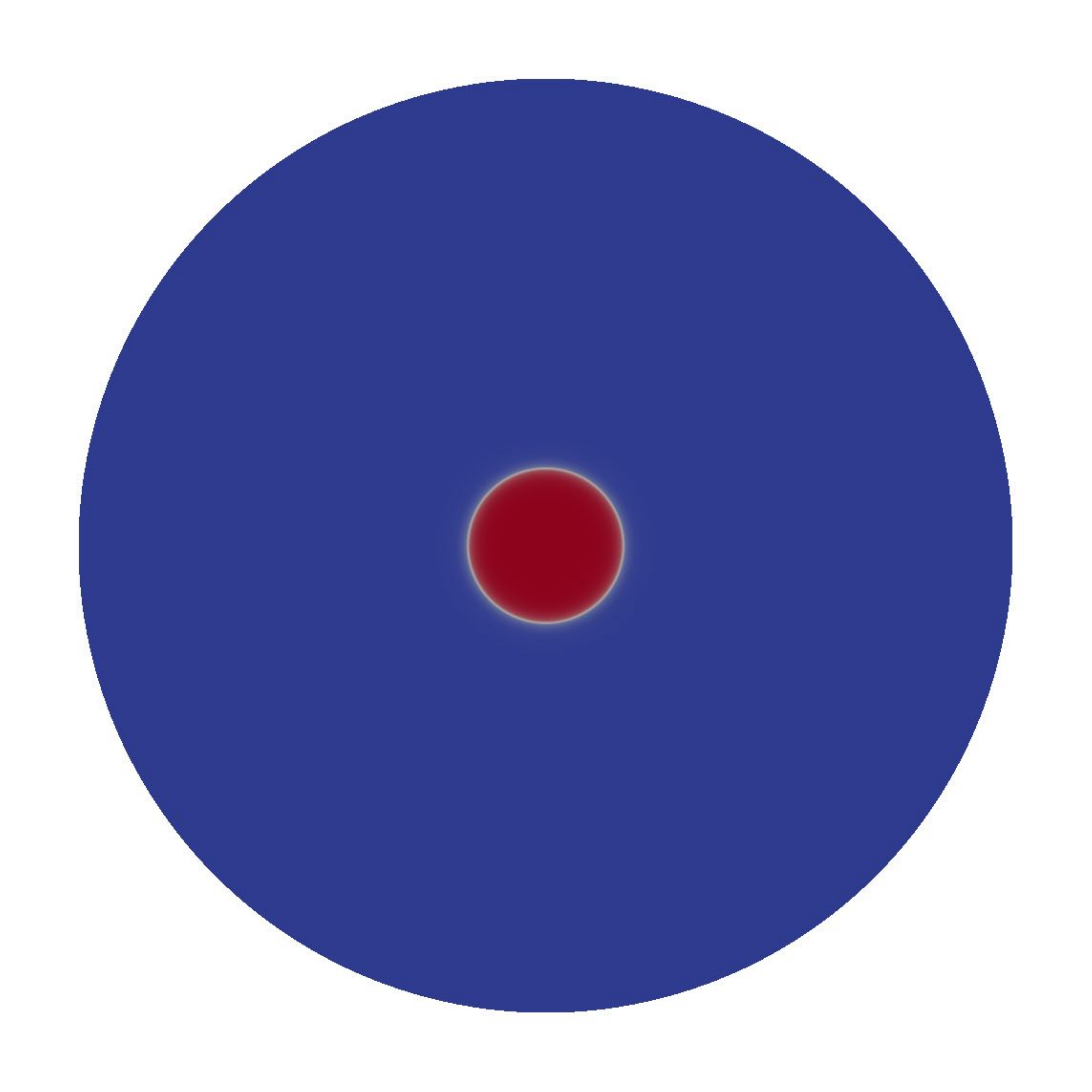}};}] {};
		\node at (0,2.1) {21st day};
		\end{tikzpicture}
		\begin{tikzpicture}
		\node [draw,circle, minimum width=.23\textwidth,
		path picture = {
			\node at (path picture bounding box.center) { \includegraphics[width=.6\textwidth]{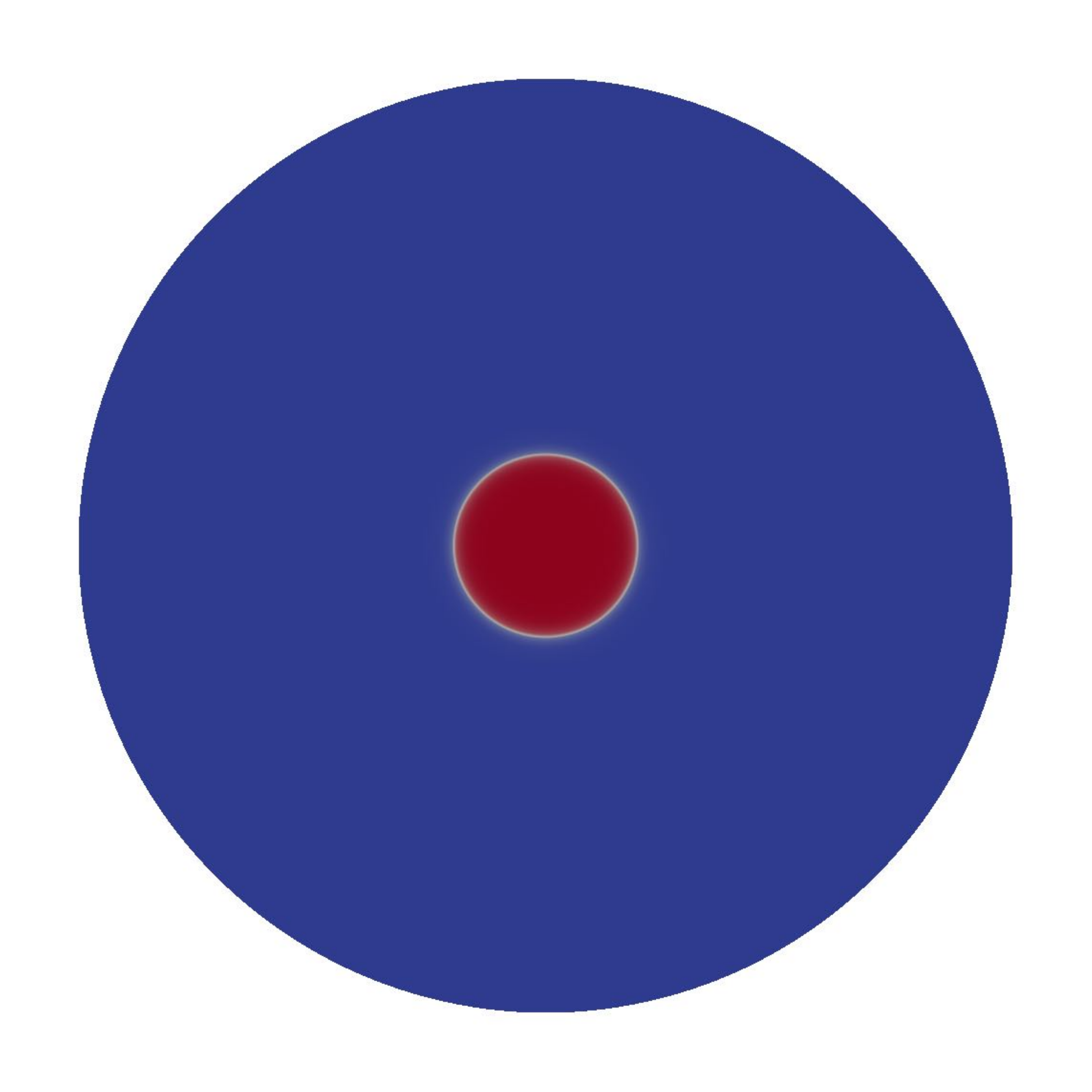}};}] {};
		\node at (0,2.1) {27th day};
		\end{tikzpicture} \\[0.2cm]
		\begin{tikzpicture}
		\node [draw,circle, minimum width=.23\textwidth,
		path picture = {
			\node at (path picture bounding box.center) { \includegraphics[width=.6\textwidth]{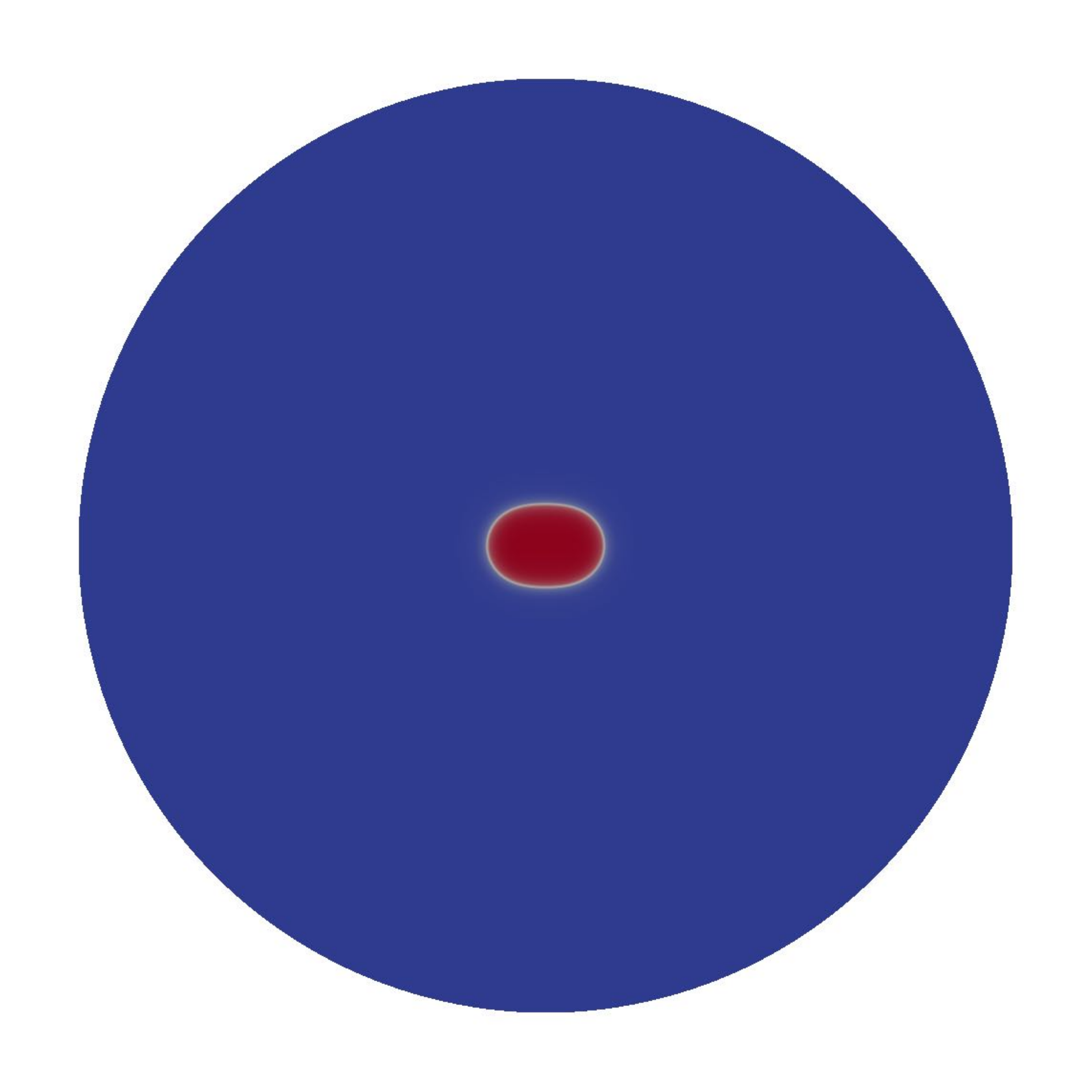}};}] {};
		\node at (-1.65,1.6) {II\,(a)};
		\end{tikzpicture}
		\begin{tikzpicture}
		\node [draw,circle, minimum width=.23\textwidth,
		path picture = {
			\node at (path picture bounding box.center) { \includegraphics[width=.6\textwidth]{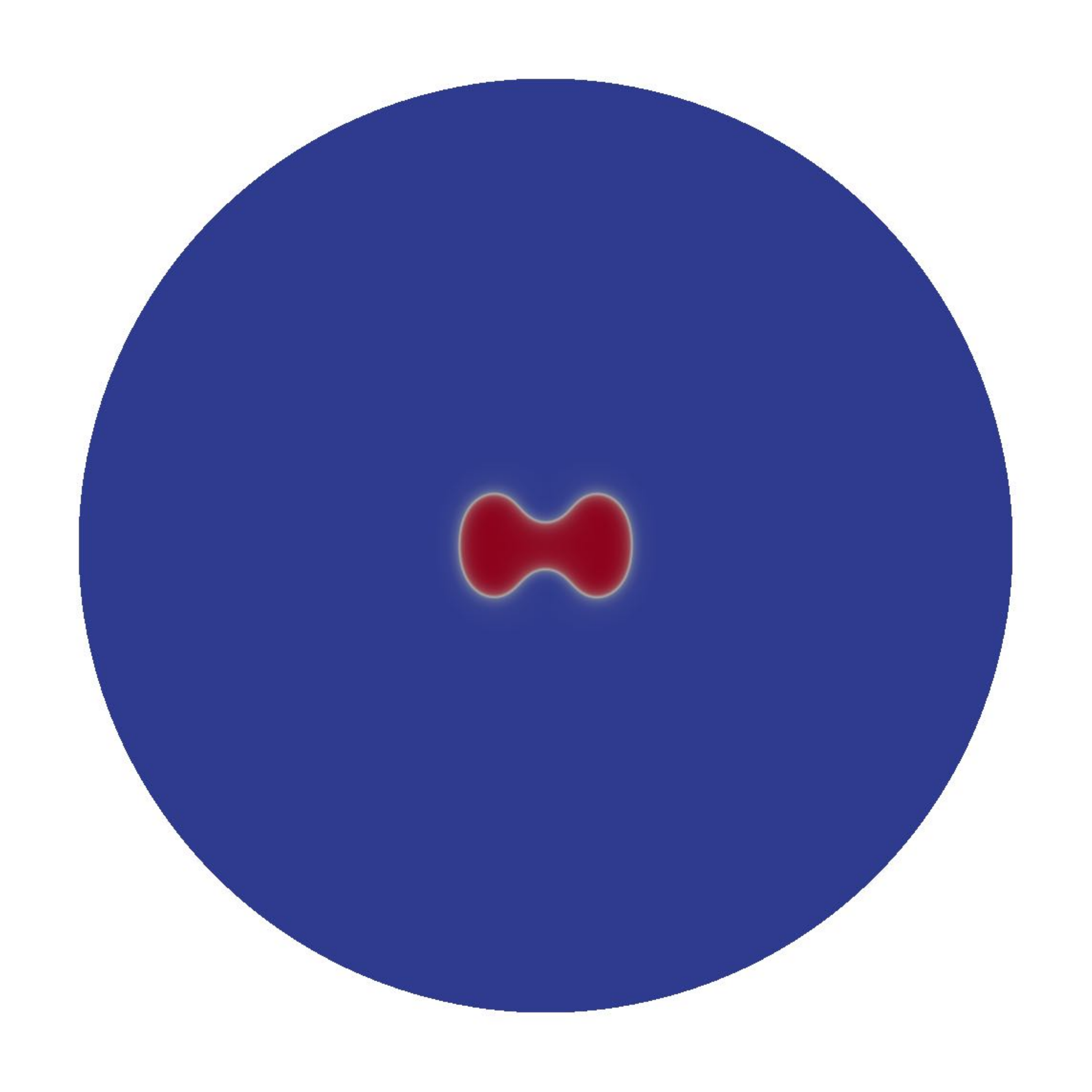}};}] {};
		\end{tikzpicture}
		\begin{tikzpicture}
		\node [draw,circle, minimum width=.23\textwidth,
		path picture = {
			\node at (path picture bounding box.center) { \includegraphics[width=.6\textwidth]{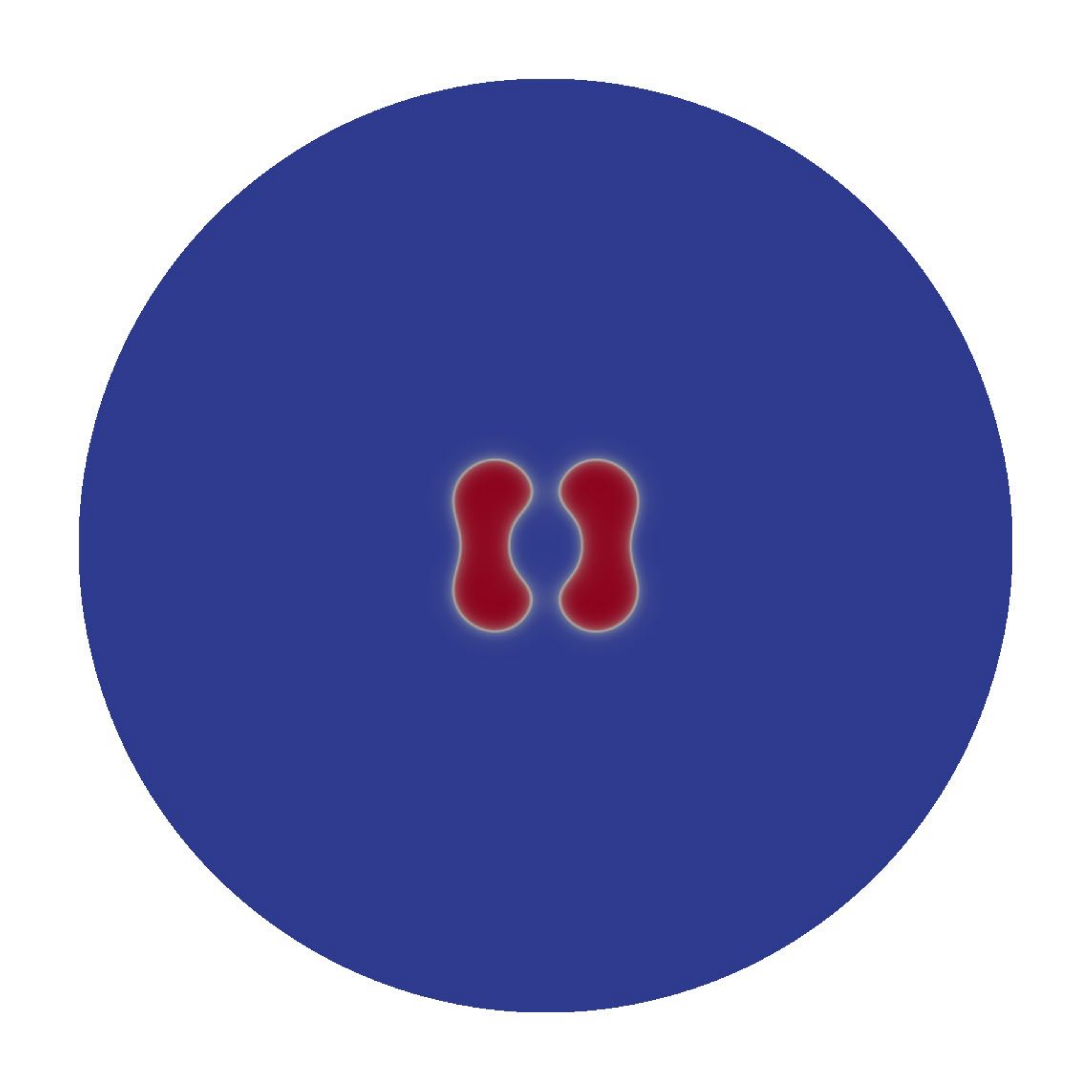}};}] {};
		\end{tikzpicture}
		\begin{tikzpicture}
		\node [draw,circle, minimum width=.23\textwidth,
		path picture = {
			\node at (path picture bounding box.center) { \includegraphics[width=.6\textwidth]{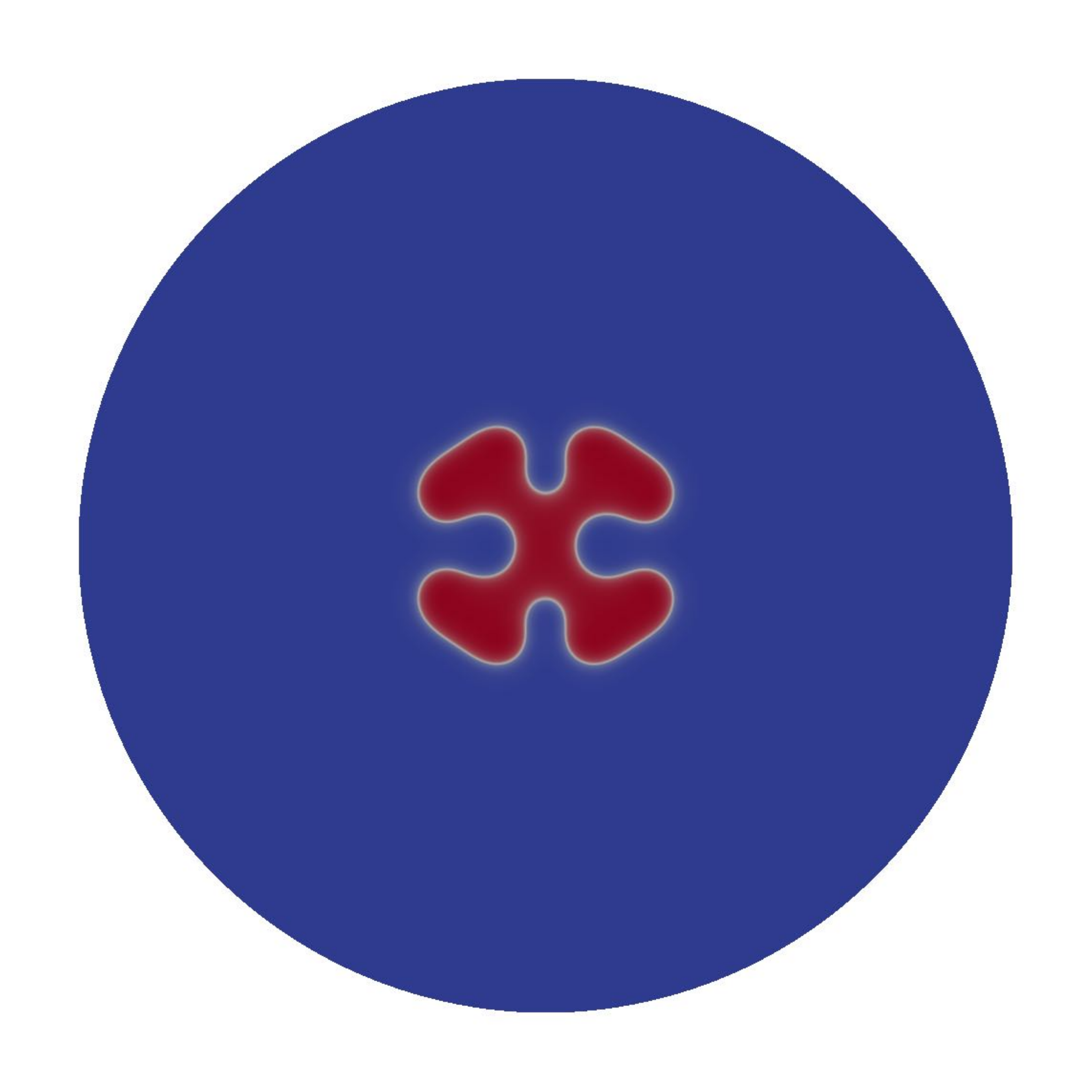}};}] {};
		\end{tikzpicture} \\[0.2cm]
		\begin{tikzpicture}
		\node [draw,circle, minimum width=.23\textwidth,
		path picture = {
			\node at (path picture bounding box.center) { \includegraphics[width=.6\textwidth]{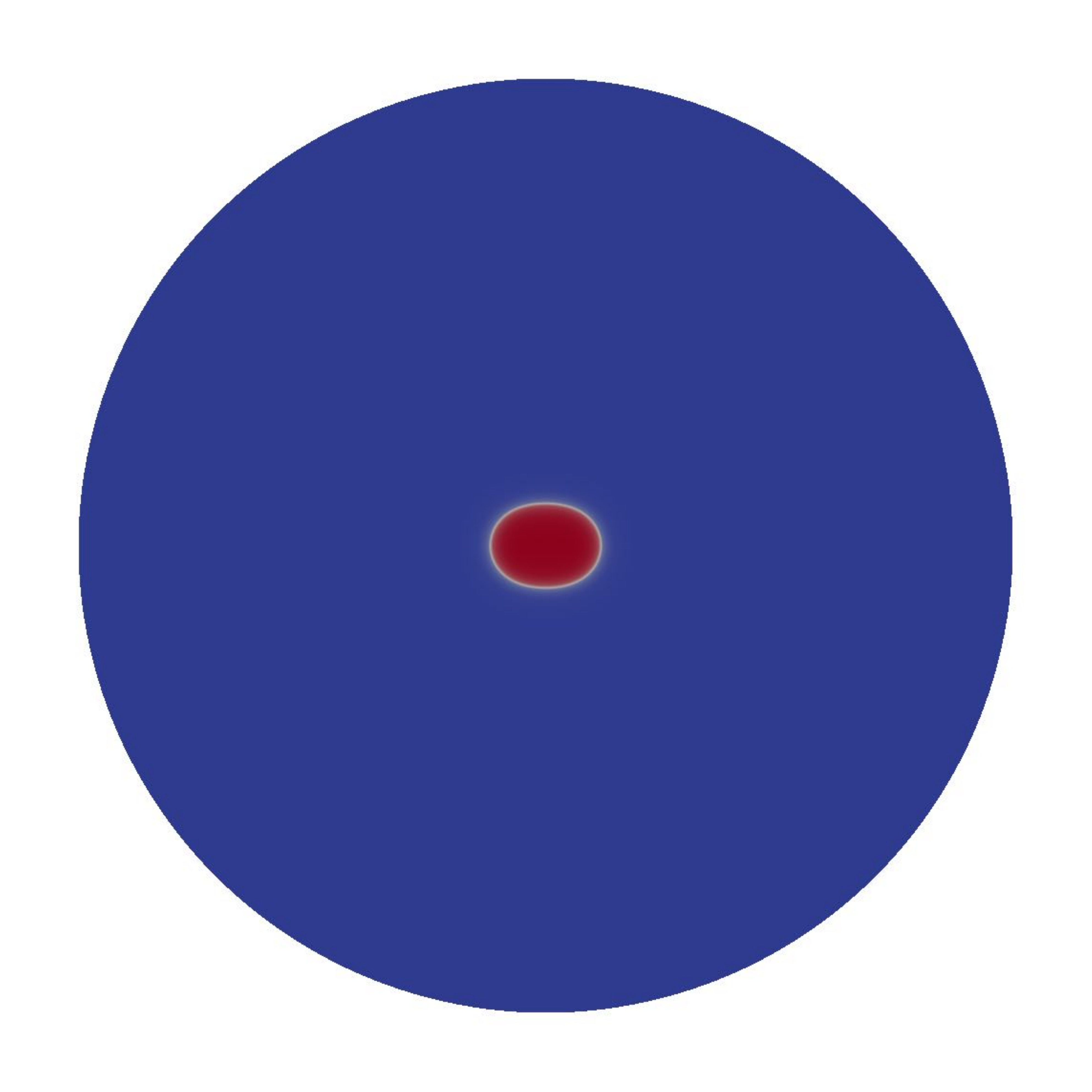}};}] {};
		\node at (-1.65,1.6) {III\,(a)};
		\end{tikzpicture}
		\begin{tikzpicture}
		\node [draw,circle, minimum width=.23\textwidth,
		path picture = {
			\node at (path picture bounding box.center) { \includegraphics[width=.6\textwidth]{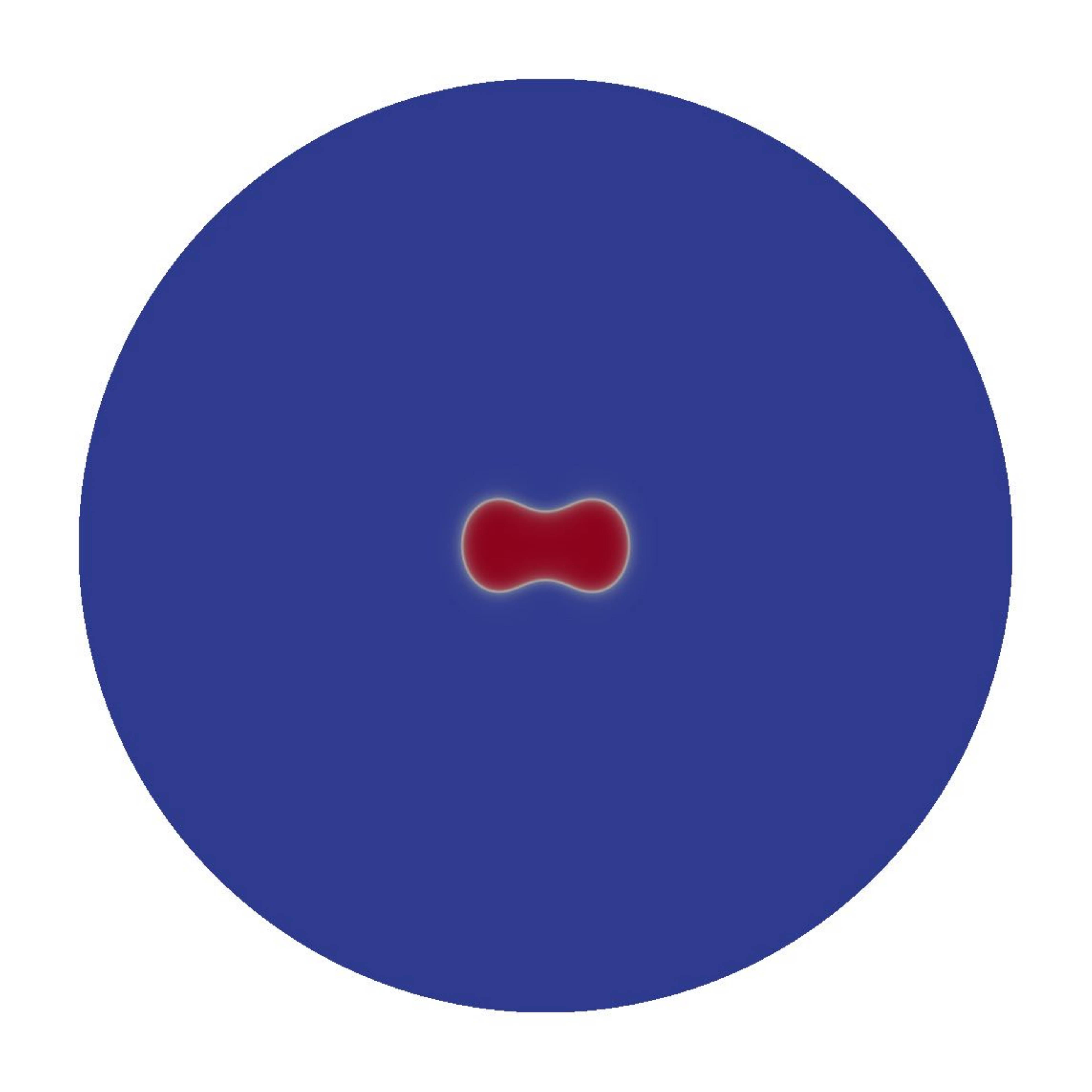}};}] {};
		\end{tikzpicture}
		\begin{tikzpicture}
		\node [draw,circle, minimum width=.23\textwidth,
		path picture = {
			\node at (path picture bounding box.center) { \includegraphics[width=.6\textwidth]{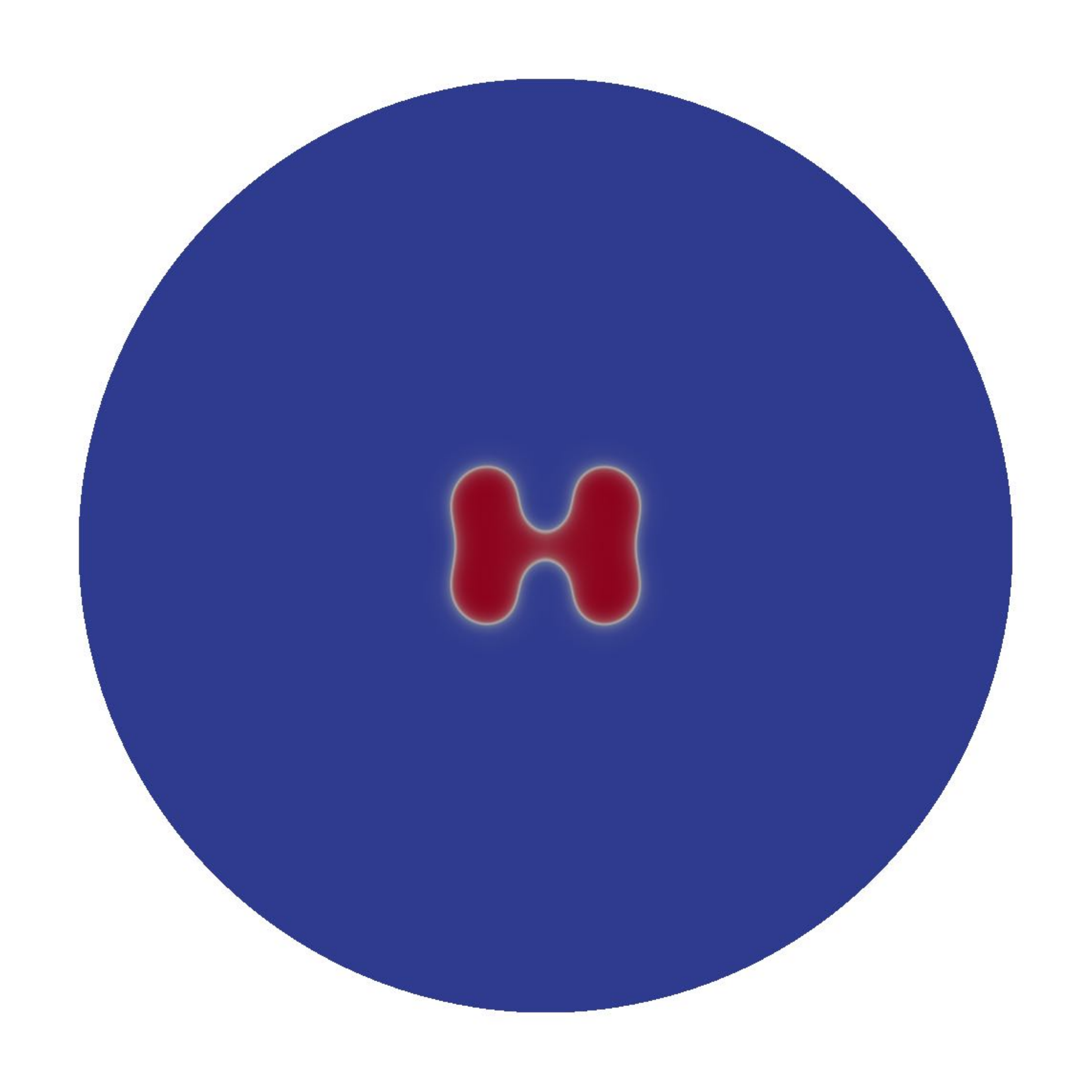}};}] {};
		\end{tikzpicture}
		\begin{tikzpicture}
		\node [draw,circle, minimum width=.23\textwidth,
		path picture = {
			\node at (path picture bounding box.center) { \includegraphics[width=.6\textwidth]{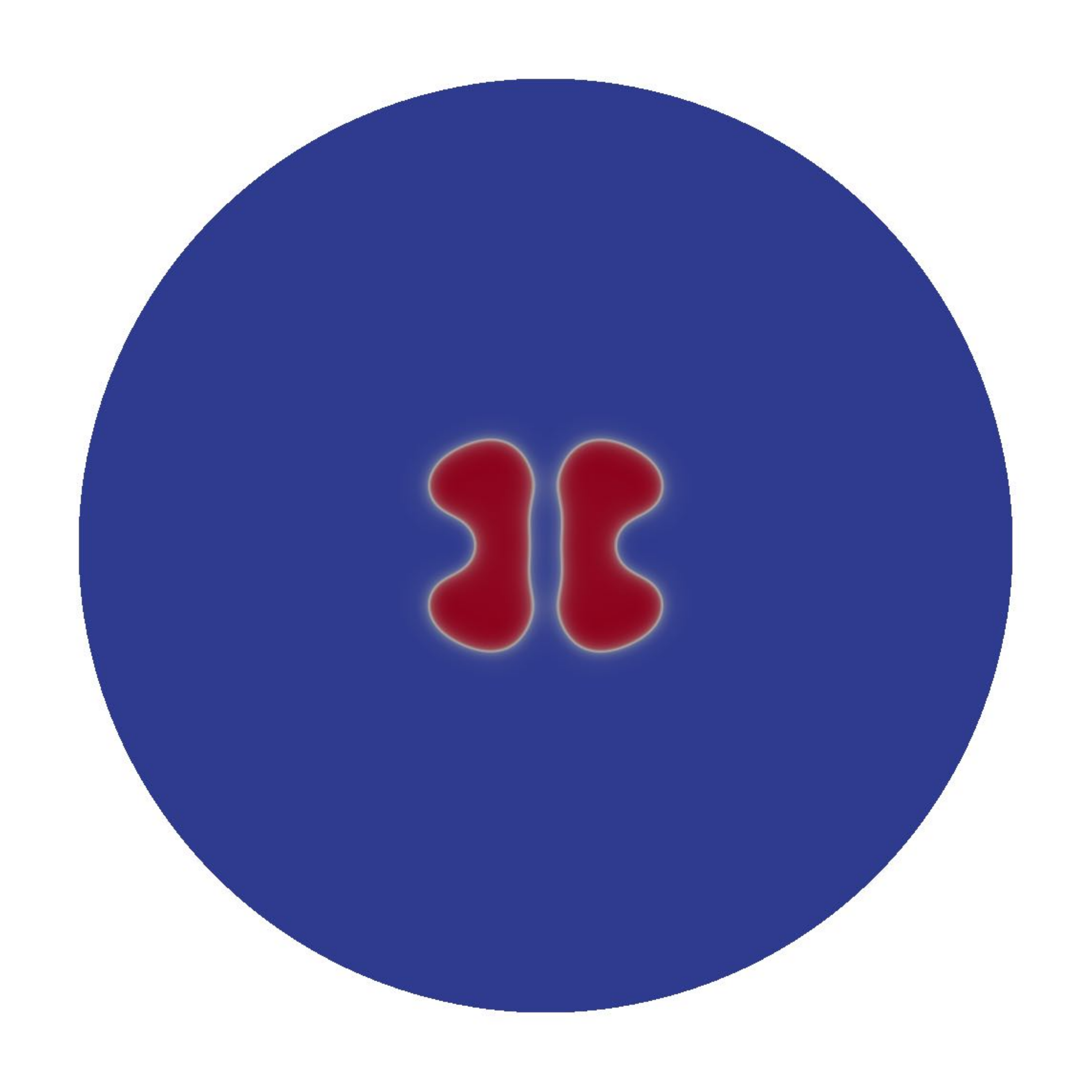}};}] {};
		\end{tikzpicture} \\[0.1cm]
		\begin{tikzpicture}
		\begin{axis}[
		hide axis,
		scale only axis,
		height=0pt,
		width=0pt,
		colormap name=special,
		colorbar horizontal,
		point meta min=0,
		point meta max=1,
		colorbar style={
			samples=100,
			height=.5cm,
			xtick={0,0.5,1},
			width=10cm
		},
		]
		\addplot [draw=none] coordinates {(0,0)};
		\end{axis}
		\end{tikzpicture}
		\caption{Evolution of the tumor volume fraction $\phi_T$ starting from the initial slightly elliptic tumor mass (a) using the models, I without velocity, II without the Forchheimer law, III full model}
		\label{Figure_Simulation2DElliptic}
	\end{figure}
	
		Next, we start from the three initial conditions (b)--(d), which have been depicted in Figure \ref{Figure_Initial2D}, and simulate the evolution of the tumor cell volume fraction using the full model III. See Figure \ref{Figure_Perturbed3} for the simulation results. 

	In the case of the highly elliptic initial tumor shape (b), we observe on the 6th day that three buds are forming, two at the end points of the horizontal shape and one in the middle. These buds continue to evolve vertically and eventually separate from each other, see the result on the 18th day. The lower and upper parts of the buds will connect again, and therefore, trapping the health tissue in between. Finally on the 27th day, we observe that the tumor shape has formed a simply connected domain.
	
	In the middle row of Figure \ref{Figure_Perturbed3}, we see the results of the simulation of model III starting with the separated initial tumor shape (c). We observe on the 6th day that the tumor cells are moving towards each other, until they connect, form buds and separate again. As in the case of the highly elliptic initial tumor shape (b), eventually, the tumor mass is forming a simply connected domain.
	
	Lastly, we simulate model III together with the irregularly perturbed tumor mass (d), see the last row of Figure \ref{Figure_Perturbed3}. Before, we always used for the initial tumor mass a symmetric shape. Now, the tumor volume fraction is starting irregularly and it keeps this form while growling in the evolving buds.
	
	\pagebreak
	
		\begin{figure}[H]
		\centering
		\begin{tikzpicture}
		\node [draw,circle, minimum width=.23\textwidth,
		path picture = {
			\node at (path picture bounding box.center) { \includegraphics[width=.4\textwidth]{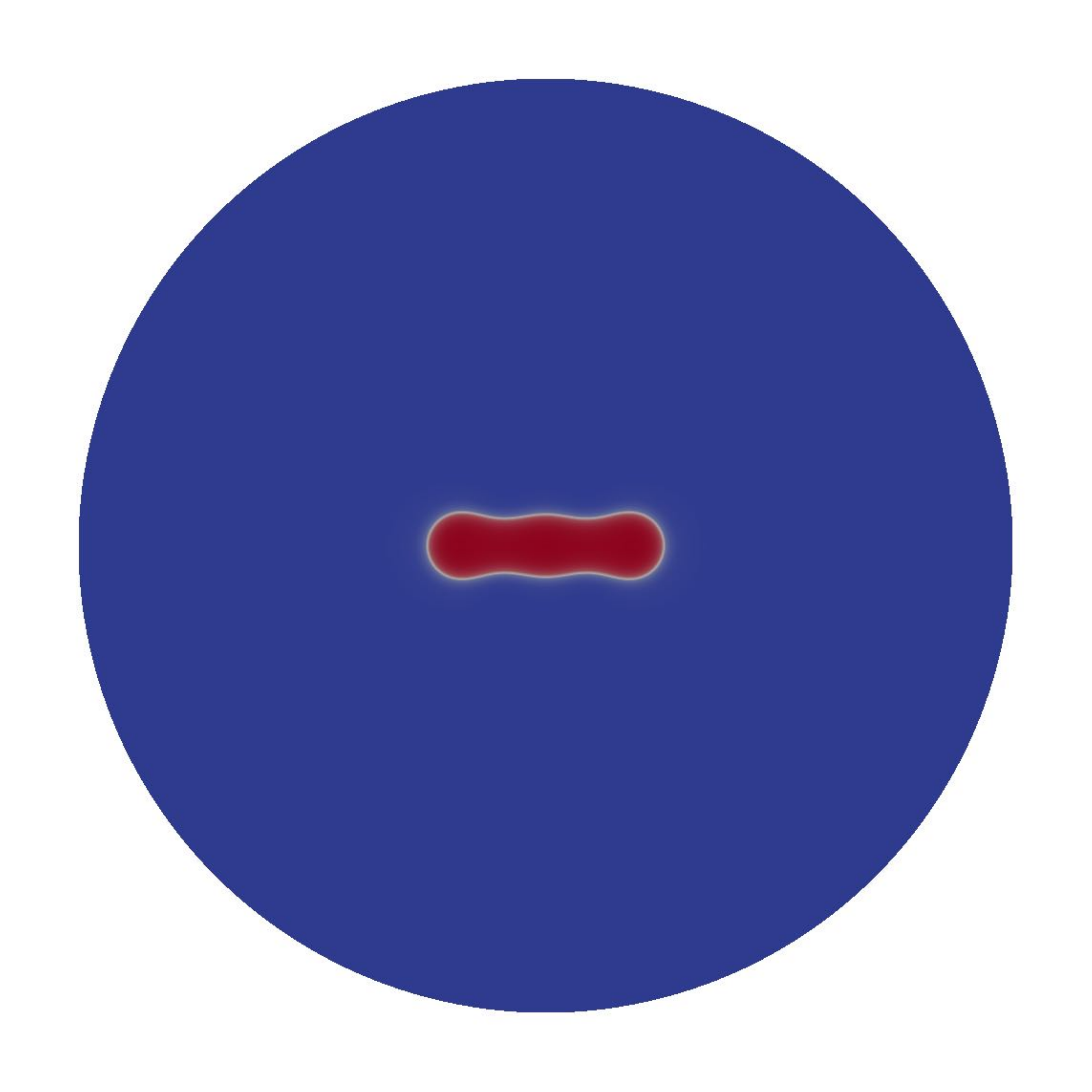}};}] {};
		\node at (-1.65,1.6) {III\,(b)};
		\node at (0,2.1) {6th day};
		\end{tikzpicture}
		\begin{tikzpicture}
		\node [draw,circle, minimum width=.23\textwidth,
		path picture = {
			\node at (path picture bounding box.center) {\includegraphics[width=.4\textwidth]{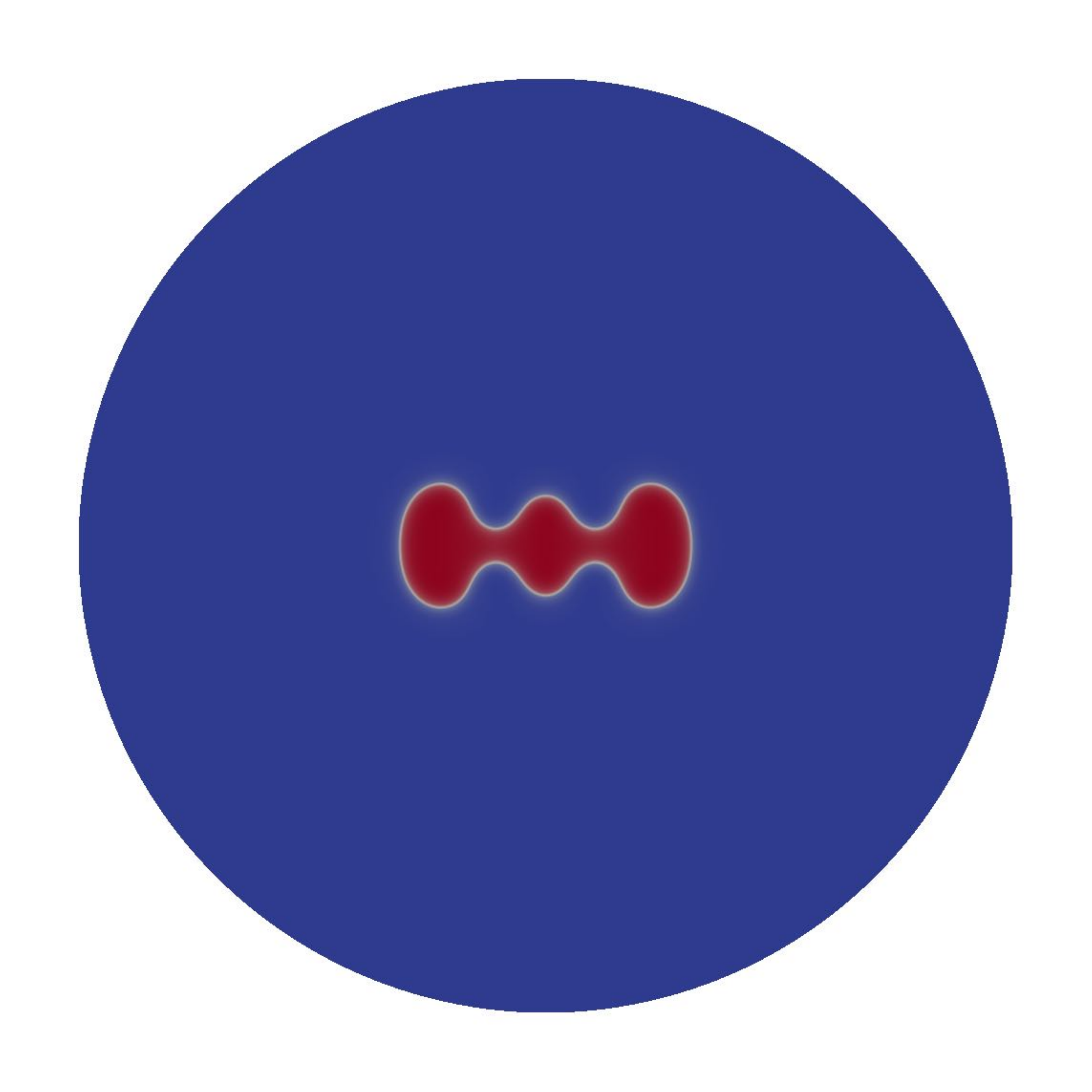}};}] {};
		\node at (0,2.1) {12th day};
		\end{tikzpicture}
		\begin{tikzpicture}
		\node [draw,circle, minimum width=.23\textwidth,
		path picture = {
			\node at (path picture bounding box.center) { \includegraphics[width=.4\textwidth]{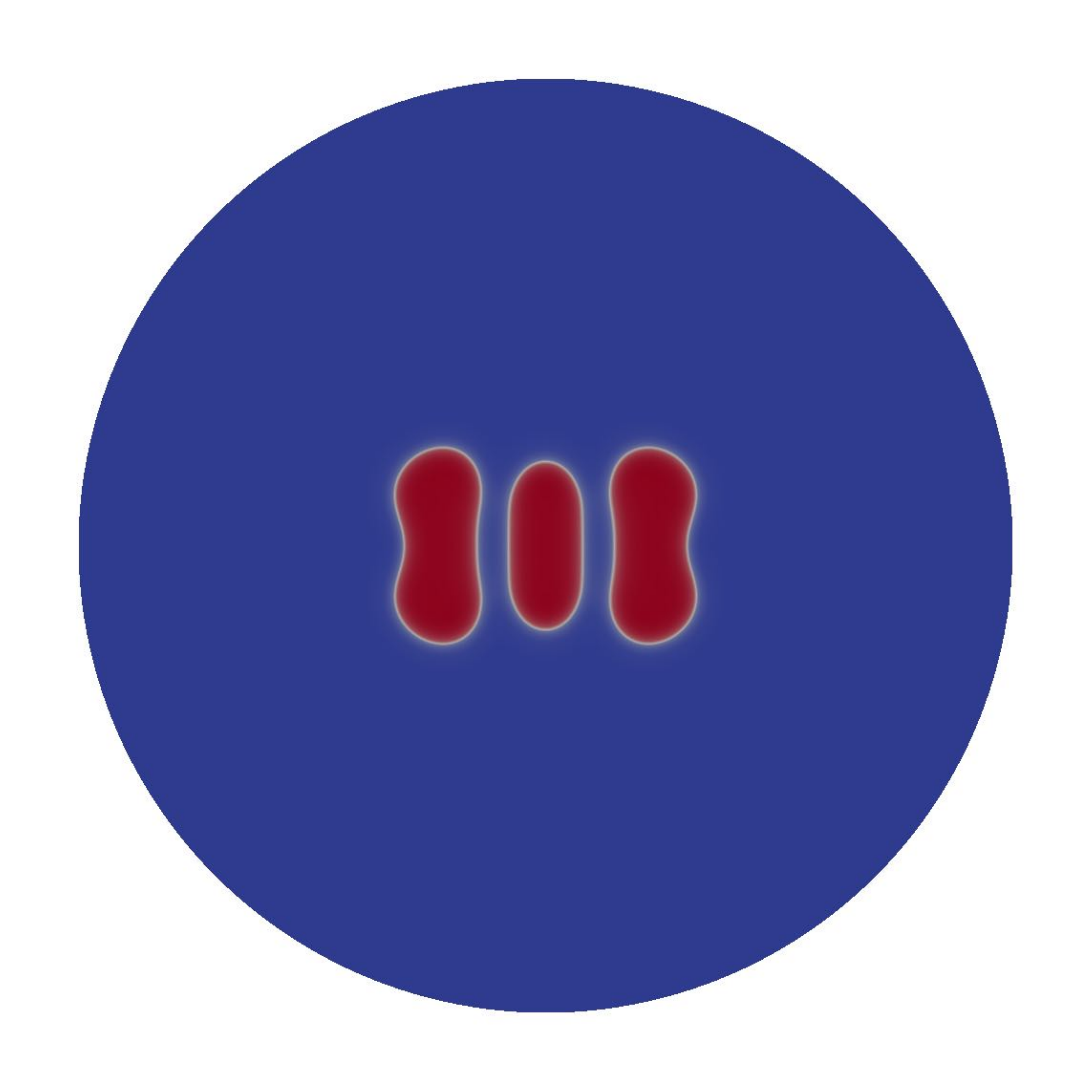}};}] {};
		\node at (0,2.1) {18th day};
		\end{tikzpicture}
		\begin{tikzpicture}
		\node [draw,circle, minimum width=.23\textwidth,
		path picture = {
			\node at (path picture bounding box.center) { \includegraphics[width=.4\textwidth]{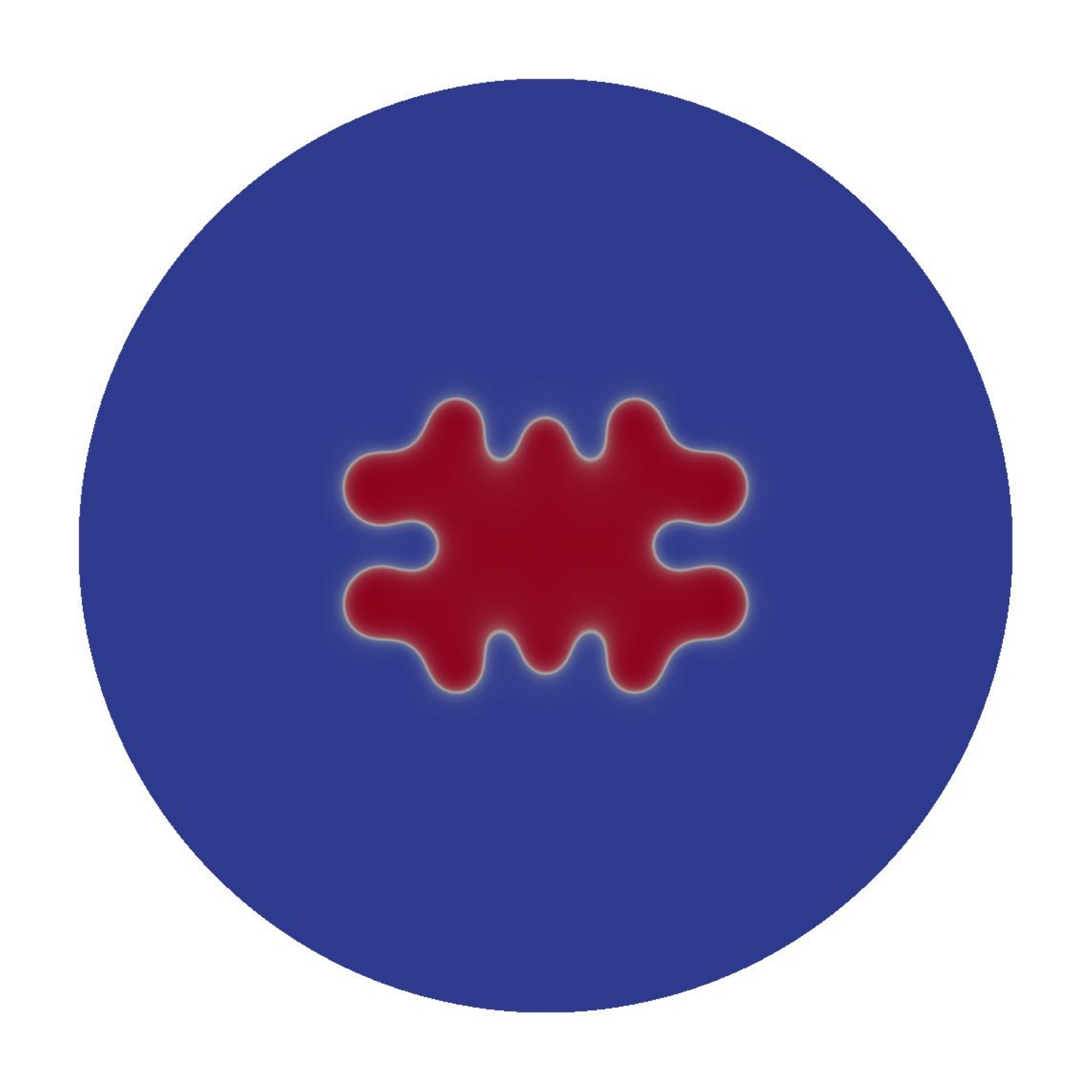}};}] {};
		\node at (0,2.1) {27th day};
		\end{tikzpicture} \\[0.2cm]
		\begin{tikzpicture}
		\node [draw,circle, minimum width=.23\textwidth,
		path picture = {
			\node at (path picture bounding box.center) { \includegraphics[width=.4\textwidth]{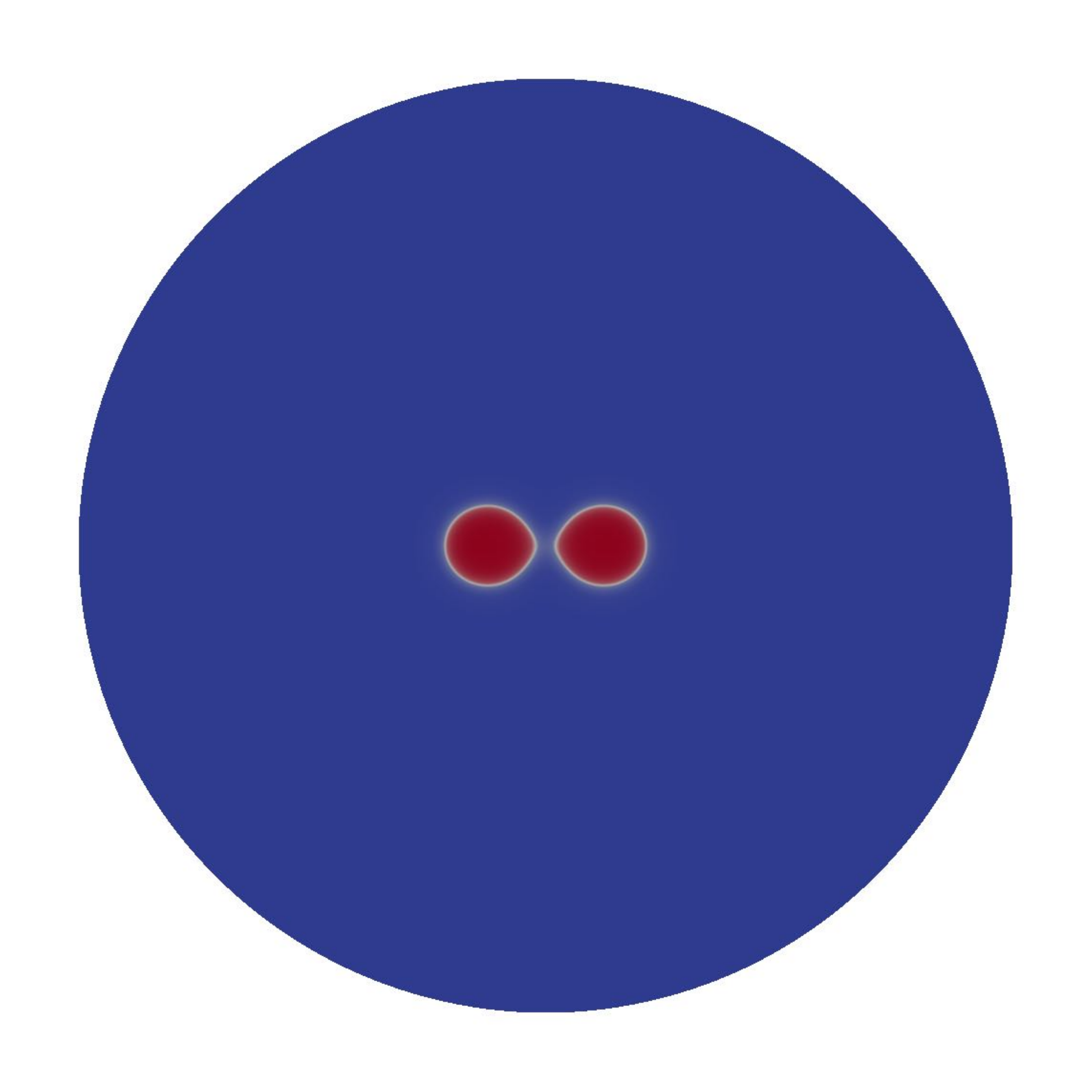}};}] {};
		\node at (-1.65,1.6) {III\,(c)};
		\end{tikzpicture}
		\begin{tikzpicture}
		\node [draw,circle, minimum width=.23\textwidth,
		path picture = {
			\node at (path picture bounding box.center) { \includegraphics[width=.4\textwidth]{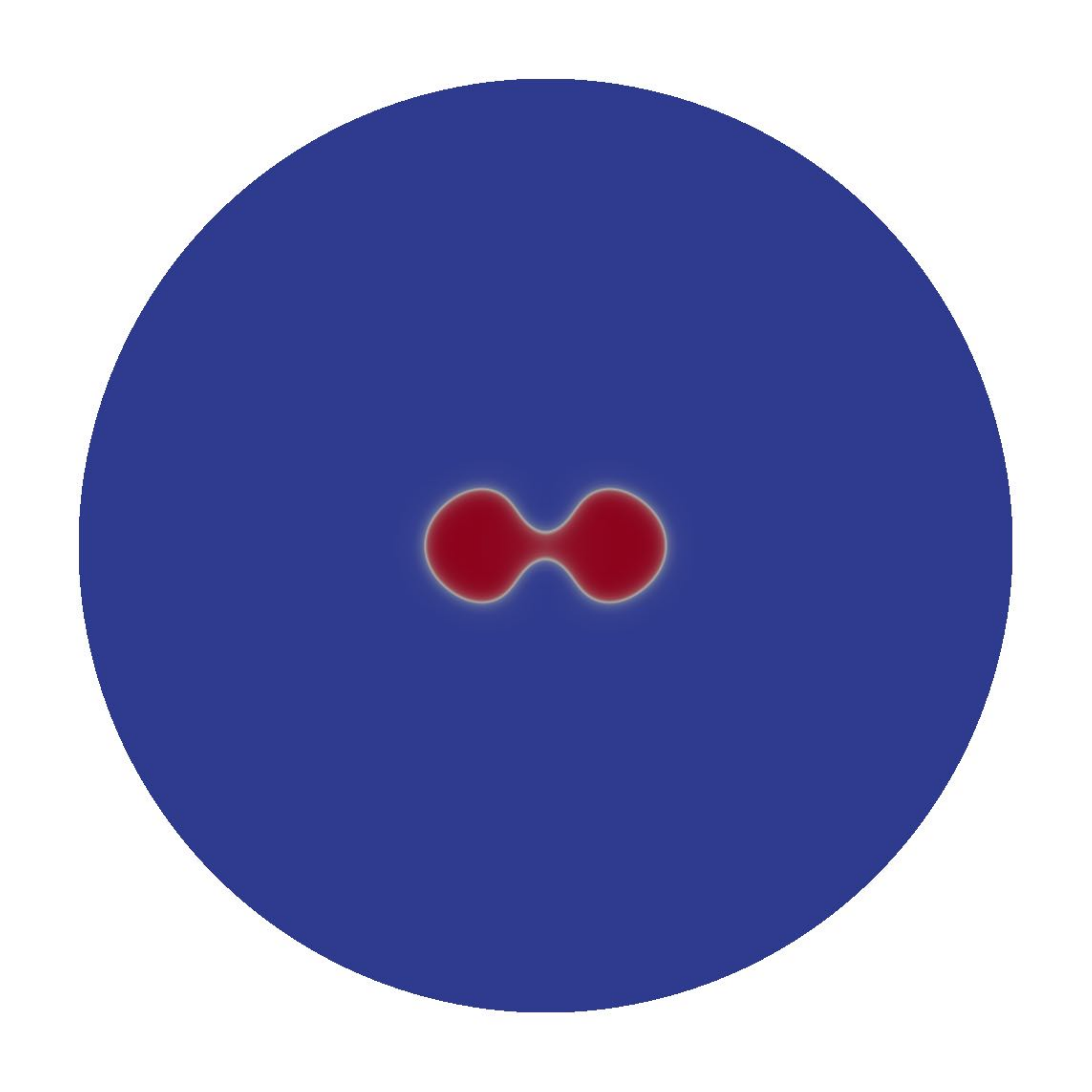}};}] {};
		\end{tikzpicture}
		\begin{tikzpicture}
		\node [draw,circle, minimum width=.23\textwidth,
		path picture = {
			\node at (path picture bounding box.center) { \includegraphics[width=.4\textwidth]{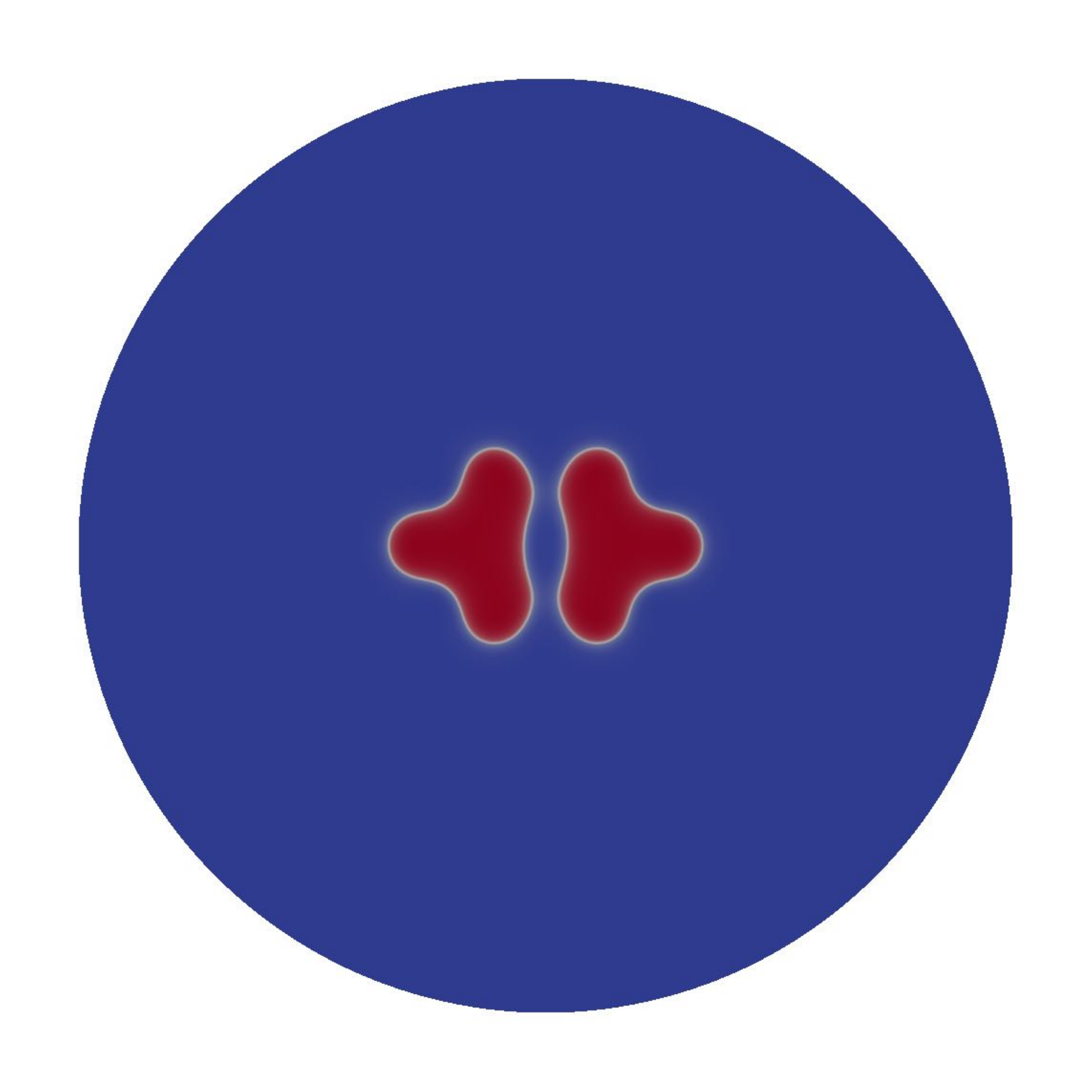}};}] {};
		\end{tikzpicture}
		\begin{tikzpicture}
		\node [draw,circle, minimum width=.23\textwidth,
		path picture = {
			\node at (path picture bounding box.center) { \includegraphics[width=.4\textwidth]{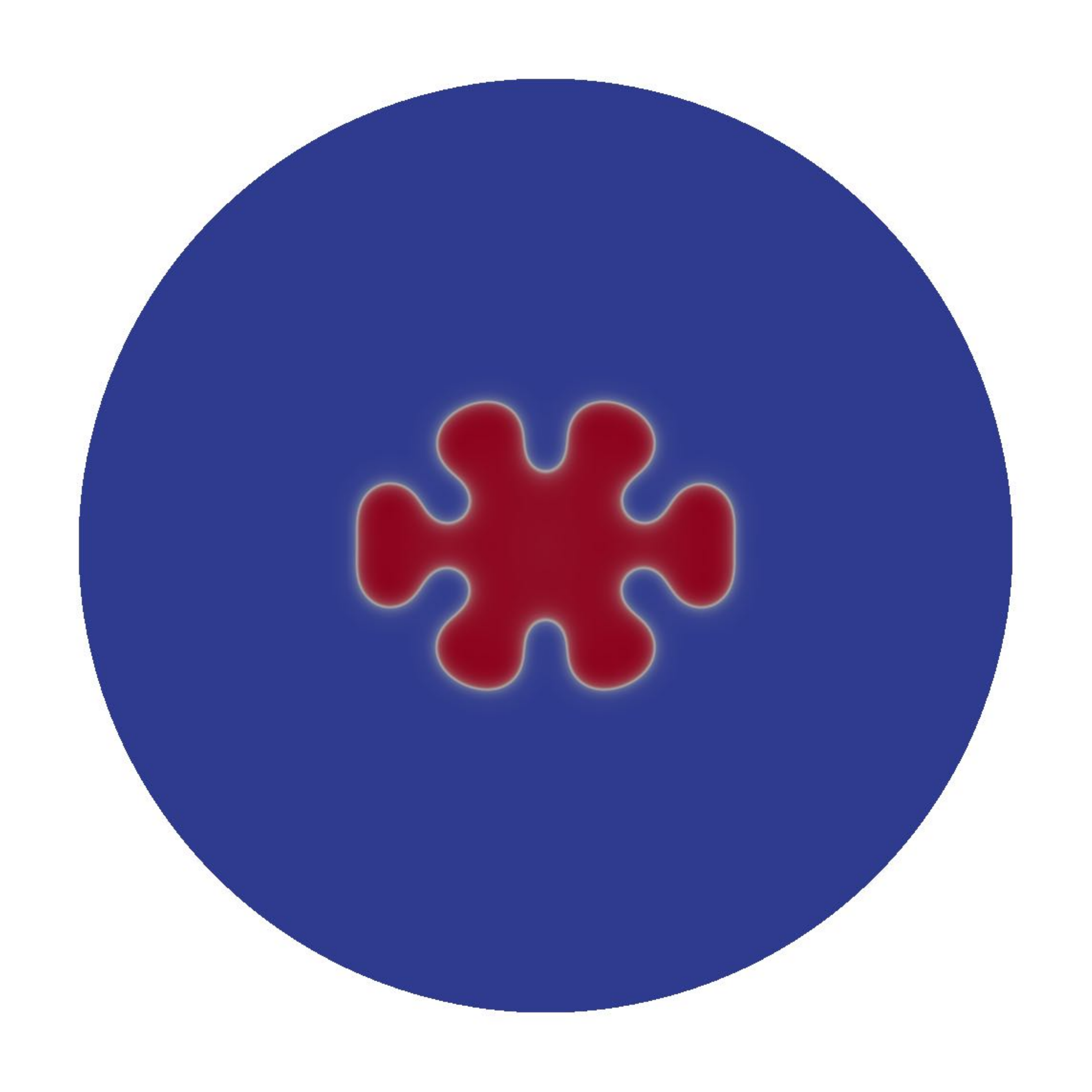}};}] {};
		\end{tikzpicture} \\[0.2cm]
		\begin{tikzpicture}
		\node [draw,circle, minimum width=.23\textwidth,
		path picture = {
			\node at (0.2,0) { \includegraphics[width=.4\textwidth]{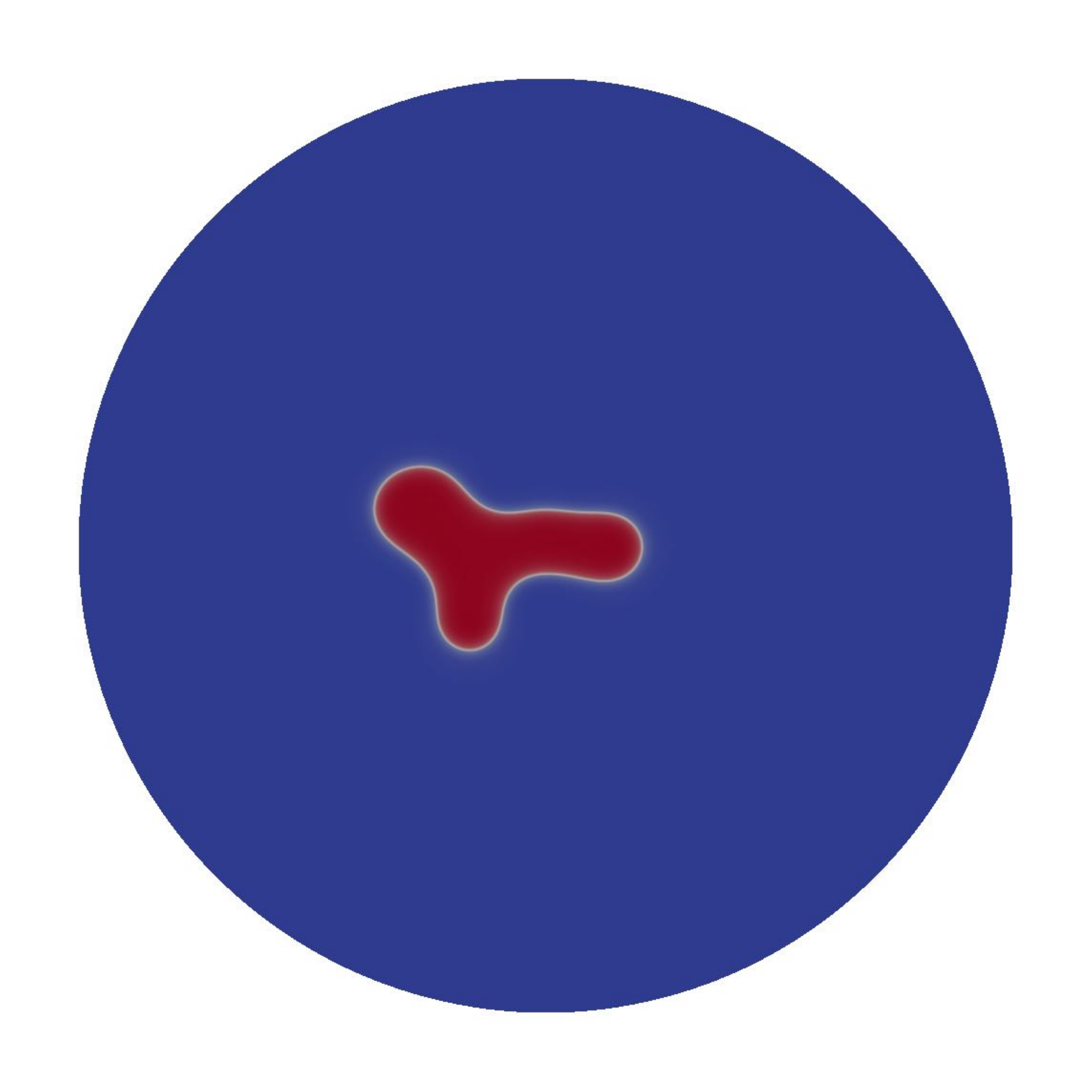}};}] {};
		\node at (-1.65,1.6) {III\,(d)};
		\end{tikzpicture}
		\begin{tikzpicture}
		\node [draw,circle, minimum width=.23\textwidth,
		path picture = {
			\node at (0.2,0) { \includegraphics[width=.4\textwidth]{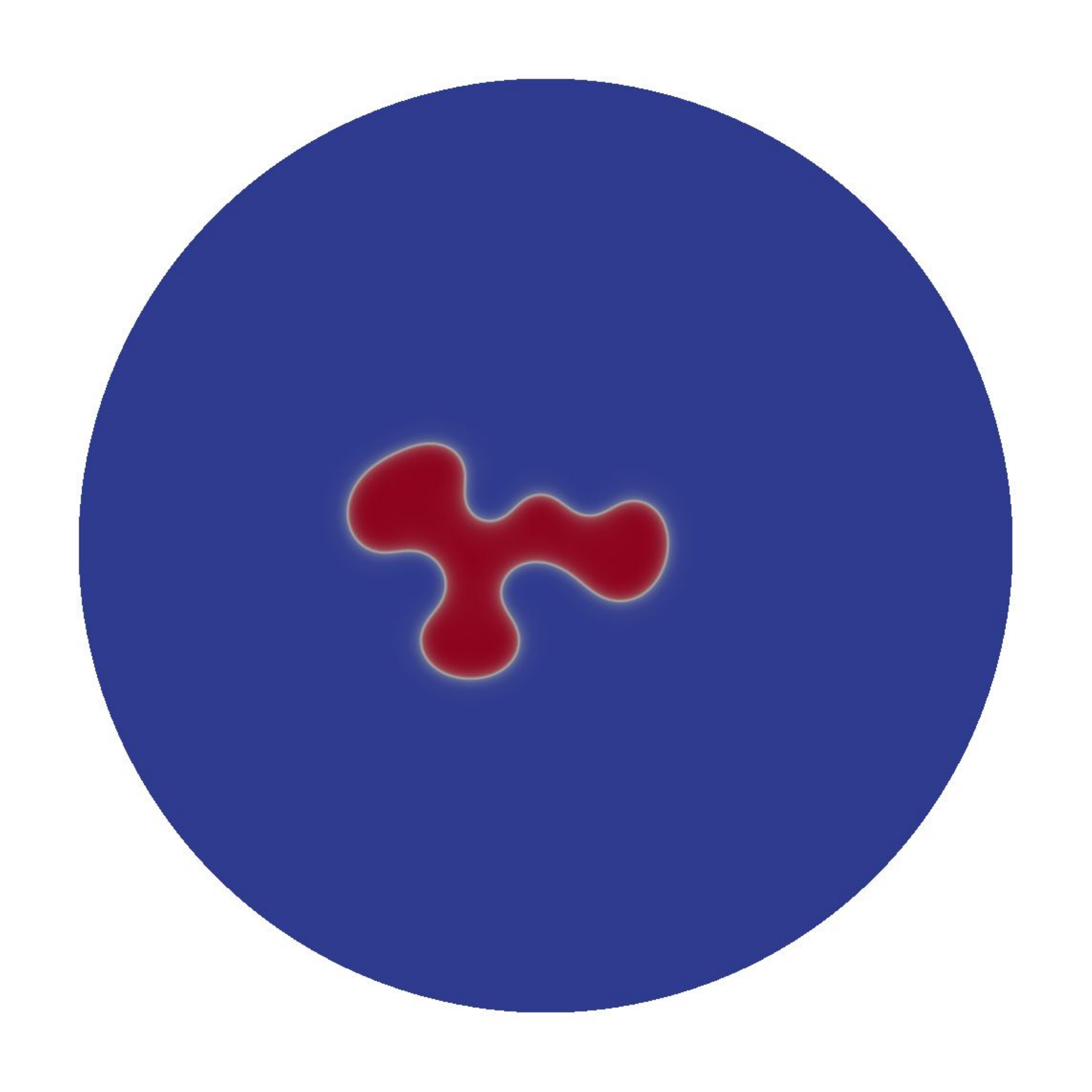}};}] {};
		\end{tikzpicture}
		\begin{tikzpicture}
		\node [draw,circle, minimum width=.23\textwidth,
		path picture = {
			\node at (0.2,0) { \includegraphics[width=.4\textwidth]{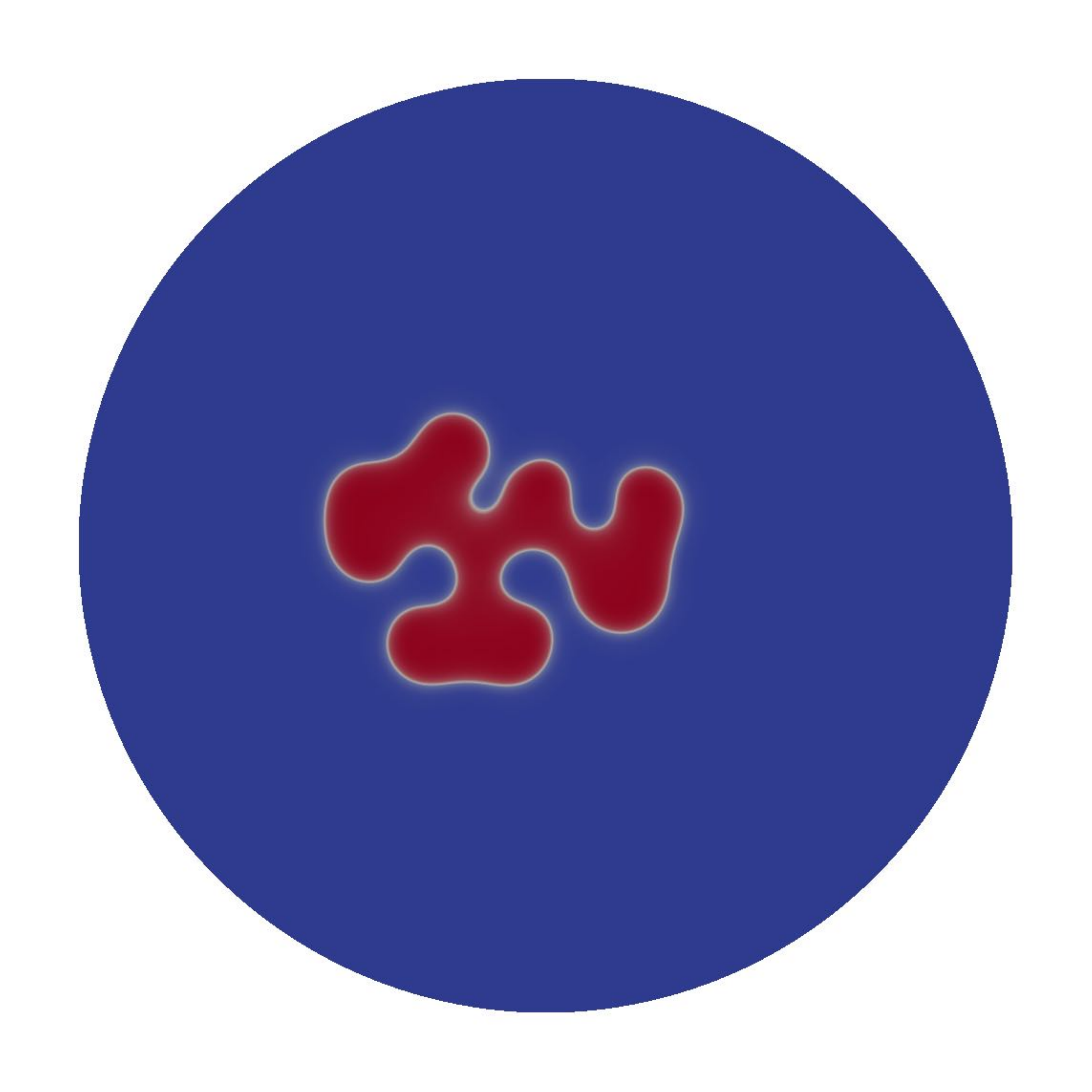}};}] {};
		\end{tikzpicture}
		\begin{tikzpicture}
		\node [draw,circle, minimum width=.23\textwidth,
		path picture = {
			\node at (0.2,0) { \includegraphics[width=.4\textwidth]{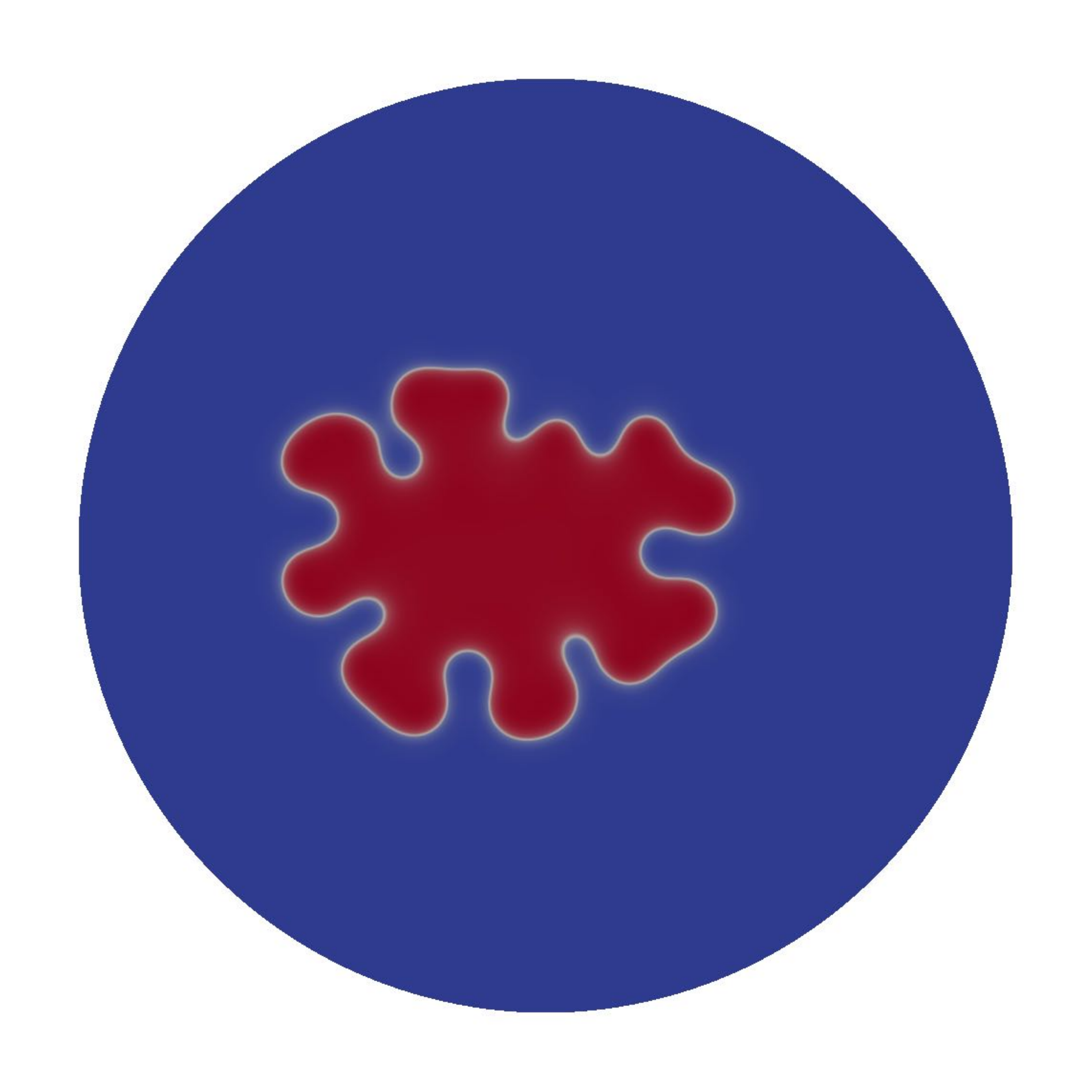}};}] {};
		\end{tikzpicture}\\[0.1cm]
		\begin{tikzpicture}
		\begin{axis}[
		hide axis,
		scale only axis,
		height=0pt,
		width=0pt,
		colormap name=special,
		colorbar horizontal,
		point meta min=0,
		point meta max=1,
		colorbar style={
			samples=100,
			height=.5cm,
			xtick={0,0.5,1},
			width=10cm
		},
		]
		\addplot [draw=none] coordinates {(0,0)};
		\end{axis}
		\end{tikzpicture}
		\caption{Evolution of the tumor volume fraction $\phi_T$ using the full model III, starting from  (b) a highly elliptic tumor mass, (c) two separated tumor masses, (d) an irregularly perturbed tumor mass}
		\label{Figure_Perturbed3}
	\end{figure}
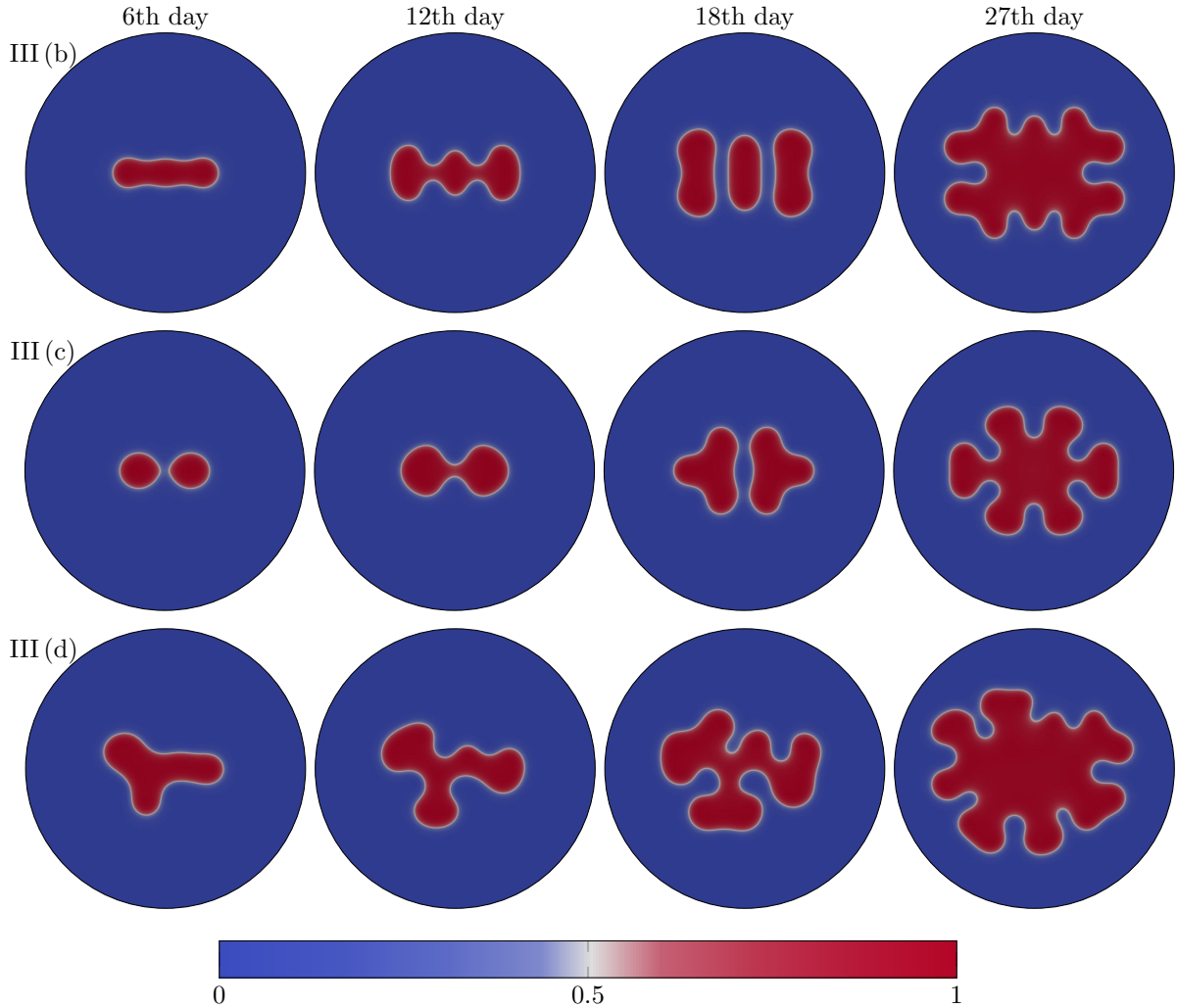

	\section{Concluding Remarks} \label{Section_Conclusion}
	In this paper, we present a mathematical analysis of a class of phase-field models of the growth and decline of tumors in living organisms, in which convective velocities of tumor cells are assumed to obey a time-dependent Darcy--Forchheimer--Brinkman flow and in which long-range interactions of cell species are accounted for through nonlocal integro-differential operators. Under some mild assumptions on mathematical properties of the governing operators, we are able to establish existence of weak solutions in the topologies of the underlying function spaces. 
	
	In addition, we explore the sensitivity of key quantities of interest, such as the evolving tumor volume, on model parameters. We demonstrate that when observational data are available, the method of active subspaces can be used to estimate parameter sensitivity. In parallel, we consider methods of output-variance-sensitivity as an alternative measure of parameter-sensitivity. Remarkably, for certain quantities of interest, such as tumor volume or mass, these two approaches yield very similar estimates. In the case in which the tumor volume is selected as the quantity of interest, the tumor proliferation parameter of the model was found to be, by far, the dominant factor compared to other parameters. 
	
	To determine the effects of various flow terms in the evolution of tumor shape and growth, we performed numerical experiments using finite-element approximations of the model for representative cases. These numerical results reveal that nonlinear flow regimes, expected to be relevant in certain types of tumors, can apparently affect the shape, connectivity, and distribution of tumors.
	
	\section*{Acknowledgments}
	
	The authors gratefully acknowledge the support of the German Science Foundation (DFG) for funding part of this work through grant WO 671/11-1, the Cancer Prevention Research Institute of Texas (CPRIT) under grant number RR160005, the NCI through the grants U01CA174706 and R01CA186193, and the US Department of Energy, Office of Science, Office of Advanced Scientific Computing Research, Mathematical Multifaceted Integrated Capability Centers (MMICCS) program, under award number  DE-SC0019393.

\end{document}